\documentclass[aop]{imsart}
\RequirePackage{amsthm,amsmath,amsfonts,amssymb,todonotes}
\RequirePackage[numbers]{natbib}
\RequirePackage[colorlinks,citecolor=blue,urlcolor=blue]{hyperref}
\RequirePackage{graphicx}
\usepackage{tikz}
\usepackage[normalem]{ulem}
\usepackage{comment}
\usetikzlibrary{decorations}
\usetikzlibrary{shapes.geometric}
\usepackage{pgfplots}
\usetikzlibrary{positioning}
\usepackage{lipsum,enumerate}
\usepackage{bbm}

\newcommand\irregularcircle[2]{% radius, irregularity
  \pgfextra {\pgfmathsetmacro\len{(#1)+rand*(#2)}}
  +(0:\len pt)
  \foreach \a in {10,30,...,350}{
    \pgfextra {\pgfmathsetmacro\len{(#1)+rand*(#2)}}
    -- +(\a:\len pt)
  } -- cycle
}

\startlocaldefs
\newtheorem{theorem}{Theorem}[section]
\newtheorem{lemma}[theorem]{Lemma}
\newtheorem{prop}[theorem]{Proposition}
\newtheorem{corollary}[theorem]{Corollary}
\theoremstyle{remark}
\newtheorem{definition}[theorem]{Definition}
\newtheorem{remark}[theorem]{Remark}

\newcommand{\newproof}[1]{{\color{black}#1}}

\endlocaldefs

\begin{document}

\begin{frontmatter}
%\title{On the Locality of Oriented Percolation on Expanders}
\title{Locality of random digraphs on expanders}
%\title{The Phase Transition for Random Digraphs on Expanders}
%\title{Expansion and the birth of a strongly connected Giant}
%\title{A sample article title with some additional note\thanksref{t1}}
\runtitle{Locality of random digraphs on expanders}
%\thankstext{T1}{A sample additional note to the title.}

\begin{aug}
%%%%%%%%%%%%%%%%%%%%%%%%%%%%%%%%%%%%%%%%%%%%%%
%%Only one address is permitted per author. %%
%%Only division, organization and e-mail is %%
%%included in the address.                  %%
%%Additional information can be included in %%
%%the Acknowledgments section if necessary. %%
%%%%%%%%%%%%%%%%%%%%%%%%%%%%%%%%%%%%%%%%%%%%%%
\author[A]{\fnms{Yeganeh} \snm{Alimohammadi}\ead[label=e1]{yeganeh@stanford.edu}},
\author[B]{\fnms{Christian} \snm{Borgs}\ead[label=e2,mark]{borgs@berkeley.edu}}
\and
\author[A]{\fnms{Amin} \snm{Saberi}\ead[label=e3,mark]{saberi@stanford.edu}}
%%%%%%%%%%%%%%%%%%%%%%%%%%%%%%%%%%%%%%%%%%%%%%
%% Addresses                                %%
%%%%%%%%%%%%%%%%%%%%%%%%%%%%%%%%%%%%%%%%%%%%%%
\address[A]{Management Science and Engineering,
Stanford University,
\printead{e1,e3}}

\address[B]{Electrical Engineering and Computer Sciences,
U.C. Berkeley,
\printead{e2}}
\end{aug}

\begin{abstract}
We study random digraphs on sequences of expanders with bounded average degree {which converge locally in probability}. We prove that the threshold for the existence of a giant strongly connected component, as well as the asymptotic fraction  of nodes with  giant fan-in or  nodes with giant fan-out are local, in the sense that they are the same for two sequences with the same  local limit.  The digraph has a bow-tie structure, with all but a vanishing fraction of nodes lying either in the unique strongly connected giant and its fan-in and fan-out, or in sets with small fan-in and small fan-out. All local quantities are expressed in terms of  percolation on the limiting rooted graph, without  any structural assumptions on the limit, allowing, in particular, for non tree-like graphs.

\newproof{In the course of establishing these results, we generalize previous results on the locality of the size of the giant to expanders of bounded average degree with possibly non-tree like limit.  We also show that regardless of local convergence of a sequence, uniqueness of the giant and convergence of its relative size for unoriented percolation imply the  bow-tie structure for directed percolation.}
 
%In the course of proving these results, we prove that for unoriented percolation, there is a unique giant above criticality, whose size and critical threshold  are again local.  
%
%\ya{It turns out that the 
%uniqueness of the giant in unoriented percolation
%is sufficient to derive the 
%bow-tie structure 
%in the directed percolation, regardless of whether the sequence has a local limit.}
% {In fact, our results for oriented percolation can be reduced to those for unoriented percolation, giving in particular a bow-tie structure for oriented percolation whenever the unoriented problem has a unique giant of linear size, whether or not the sequence has a local limit.}

An application of our methods  shows that the critical threshold for bond percolation and random digraphs on preferential attachment graphs is $p_c=0$, with an infinite order phase transition at $p_c$.
\end{abstract}



\end{frontmatter}
%%%%%%%%%%%%%%%%%%%%%%%%%%%%%%%%%%%%%%%%%%%%%%
%% Please use \tableofcontents for articles %%
%% with 50 pages and more                   %%
%%%%%%%%%%%%%%%%%%%%%%%%%%%%%%%%%%%%%%%%%%%%%%
%\tableofcontents

\section{Introduction}

Many stochastic processes, from statistical physics models to epidemics or information diffusion, take place on an underlying network.  This naturally {gives rise} %leads 
to random subgraphs of the original graph, which in the
simplest cases is described by unoriented or oriented percolation, see \cite{InfectionGraph82} for infections with constant recovery time, and \cite{Kempe-Kleinberg} for information diffusion.  In both cases, the oriented subgraph stems from the fact the process is inherently directed, {with nodes infecting or informing their} neighbors independently
with probability $p$.

This  leads to the question whether the important properties of these processes depend on global properties of the network, like connectivity or bipartiteness (as in the case of anti-ferromagnetic spin models), or whether it is enough to know just local information, represented by the $k$-neighborhoods of random vertices in the original graph.  Specifically, we will look at the relative size of the giant component for unoriented percolation, while for oriented percolation, we will look at the fraction of nodes
with large fan-out (corresponding to the probability that a random {seed} leads to an outbreak / successful campaign) or large fan-in (corresponding to nodes likely to be infected in an outbreak).

As we will see, local information is not quite enough  - in addition, we will need what we call \emph{large-set expansion}, a condition which guarantees that for large sets, the size of the edge boundary of a set grows linearly in its size.   Under {this condition}, we show that the proportion of nodes with large fan-in or fan-out is indeed local.  Here locality will be formalized by the notion of  local convergence \cite{Aldous2004, benjamini2001}, see Section~\ref{sec: LWC} below for the precise definition.

The question of locality of unoriented percolation on expanders has recently received much attention in the probability community. In \cite{alon2004},  it was shown that on bounded degree expanders, there exists at most one linear size component {(giant)}.  If in addition, one assumes the existence of a local weak limit, one obtains locality of the  threshold for the appearance of a giant \cite{benjamini2009critical,sarkar2018note}, in the sense that it  can {be} inferred from the limit.  Less is known for the relative size of the giant. Indeed, for bounded degree expanders, locality of the size of the giant is only known for high girth {regular expanders
\cite{Krivelevich20HighGirthExapnder}.  In this case,} the relative size of the giant is given by the survival probability of a percolated branching process.

To our knowledge, no results are known for oriented percolation on expanders with local limit.

{While somewhat {tangential} to the purpose in this paper, we would be amiss not to mention the vast literature on the percolation threshold and the size of the giant for random graphs and for percolation on random graphs, starting with the work of Erd\H{o}s and R\'enyi \cite{Erdos}.
Since then, various other random models have been studies, from the random digraph of Karp \cite{Karp90}, to so-called configuration models \cite{CM-Bollobas,molloy_reed_1998,CM-JansonLuczak} and their directed analogues \cite{cooper_frieze_2004},
to percolation on regular random graphs \cite{GOERDT2001,pittel2008} and configuration models \cite{Fount07,Janson09,BR2015}.}  {Note that in all these models, the size of the giant is again given in terms of the survival probability of a suitable branching process.}

To state our results formally, we need the notion of large-set expanders.
 Formally, it is defined as follows: Given a graph $G=(V,E)$ and a constant $\epsilon<1/2$, we define
\begin{equation}
\label{phi}\phi(G,\epsilon)=\min_{A\subset V: \epsilon |V|\leq |A|\leq |V|/2}
\frac{e(A,V\setminus A)}{|A|}
\end{equation}
where  $e(A,V\setminus A)$ is the number of edges joining $A$ to its complement.
Call a graph $G$ an $(\alpha,\epsilon,\bar d)$ large-set expander if the average degree of $G$ is at most $\bar d$ and $\phi(G,\epsilon)\geq \alpha$.  A sequence of possibly random graphs $\{G_n\}_{n\in\mathbb{N}}$ is  called a \emph{large-set expander sequence with bounded average degree},  if there exists $\bar d<\infty$ and $\alpha>0$ such that for {all} $\epsilon\in (0,1/2)$, the probability that 
$G_n$ is an  $(\alpha,\epsilon,\bar d)$ large-set expander goes to $1$ as $n\to\infty$. 

To simplify our notation, we will take $G_n$ to be a graph on $n$ vertices.  As usual, we use $G(p)$ to denote the random subgraph obtained from a graph $G$ by independently keeping each edge with probability $p$.  Given a probability measure $\mu$ on $\mathcal G_*$, we then define
\begin{equation}\label{zeta}
    \zeta(p)=\mathbb E_\mu\big[\mathbb P_{G(p)}(|C(o)|=\infty )\big]
\end{equation}
where $o$ is the root in $(G,o)\sim \mu$ and $C(o)$ is the connected component of $o$ in $G(p)$, and we define the percolation threshold $p_c(\mu)$ of $\mu$ as
\begin{equation}\label{pc-limit}
p_c(\mu)=\inf_p\{p\in[0,1]: \zeta(p)>0\}.
\end{equation}
\newproof{Finally, we use 
a quenched notion of local weak convergence, namely that of local convergence in probability, see Section~\ref{sec: LWC} for the precise definition}

\begin{theorem} \label{thm: size of giant in expander}
Let
 $\{G_n\}_{n\in\mathbb{N}}$ be a sequence of (possibly random) large-set expanders with bounded average degree converging locally in probability to $(G,o)\in\mathcal{G}_*$ with non-random distribution $\mu$. %Define  $\zeta(p)$
 %=\mu\big(\mathbb P_{G(p)}(|C(o)|=\infty )\big)$ 
 %as in \eqref{zeta}, and let  
 Let $C_{i}$ be the $i^{th}$ largest component of $G_n(p)$.
{If $p\neq p_c(\mu)$, then}
\[\frac{|C_1|}{n}\overset{\mathbb{P}}{\to}\zeta(p),\]
with $\overset{\mathbb{P}}{\to}$ denoting convergence in probability with respect to both $\mu$ and percolation. Moreover, for all $p\in[0,1]$,  $\frac{|C_2|}{n}\overset{\mathbb{P}}{\to}0, $
where the convergence is uniform \newproof{on any closed interval $I\subset (0,1)$ in that $\sup_{p\in I}\mathbb P(|C_2|\geq\epsilon n)\to 0$ for all $\epsilon>0$.}
\end{theorem}

\newproof{
\begin{remark} \label{rem:zeta-cont}
The restriction $p\neq p_c(\mu)$ can be removed for models where it is known that $\zeta$ is continuous at $p_c$. This includes many models where $\mu$ is supported on trees, including 
preferential attachment (where $p_c(\mu)=0$) and all $\mu$ supported on trees with more than $3$ ends, i.e., trees with at least $3$ disjoint path to $\infty$.

In fact, when proving the statement about the asymptotic size of the giant we first prove that it holds whenever $\zeta$ is continuous at $p$, and then prove that
under the assumption of the theorem, $\zeta$ is continuous except possibly at $p_c$.  To this end, 
 we generalize a result of Sarkar \cite{sarkar2018note}, and prove that a deterministic  measure $\mu$ on $\mathcal G_*$ is extremal in the set of unimodular measures on $\mathcal G_*$ when it is the local limit in probability of some (possibly random) sequence of large set expanders with  bounded average degree.  By a theorem of Aldous and Lyons \cite{aldous2007processes}, this in turn implies that $\zeta$ is continuous except possibly at $p_c(\mu)$\footnote{We thank the anonymous referee for pointing out this connection, and suggesting the generalization of Sarkar's results to our settings.}.
See Section~\ref{sec:zeta-cont}  for a more detailed discussion.
 \end{remark}
}

Our theorem generalizes previous  results  \cite{alon2004,benjamini2012, sarkar2018note,Krivelevich20HighGirthExapnder} in several directions, { allowing for applications to graph sequences sharing some of the features of more realistic network models.}
%motivate by applications, including our application to directed percolation in this paper.
First, we remove the condition of bounded degrees, and replace it by bounded average degree, a condition which allows for power law graphs which were not included before.  As an illustrative example,
we consider preferential attachment   and show that the critical threshold for a linear sized giant is $p_c=0$,
with $\zeta(p)=e^{-\theta(1/p)}$ as $p\to 0$, corresponding to an infinite order phase transition (Theorem~\ref{thm: pref-attachment}).
Second, we remove the assumption that the graph $G_n$ is locally tree-like, and give an explicit expression for the asymptotic size of the giant regardless of whether or not the limit is given by a birth process.
Third, we relax the condition of expansion, to include graphs which are not necessarily connected - as a side benefit, we obtain a condition which in many cases is easier to verify (see Appendix~\ref{sec: expansion pref attachment} for the case of preferential attachment).

\newproof{Despite these generalizations, the proof of Theorem~\ref{thm: size of giant in expander}
relies mainly on extensions of known methods, including those of Alon, Benjamini and Stacey \cite{alon2004} and Krivelevich, Lubetzky and Sudakov  \cite{Krivelevich20HighGirthExapnder}.  These were a major motivation for our proofs, even though the randomness of the sequence $G_n$ induces some subtleties which need to be taken into account to avoid trivial counter examples.  We discuss these in Section~\ref{sec: unoriented expander}, and relegate the more standard techniques to an appendix.  }

As a corollary of {Theorem~\ref{thm: size of giant in expander}, one obtains} a generalization of the results from \cite{benjamini2009critical,sarkar2018note}on the ``locality'' of $p_c$; indeed, our theorem implies that for  sequences of large-set expanders with bounded average degree converging locally in probability, the critical threshold for the appearance of a giant is equal to $p_c$
defined in \eqref{pc-limit}. %\todo{removed the verbal description of pc}
%which is the analog of the critical threshold for percolation on branching processes for the non-tree like limits considered in this paper. 
See Section~\ref{sec: defi percolation} for the precise definition of a critical threshold, and Corollary~\ref{cor: critical prob} below for a formal statement of this corollary.

The second {(and we believe technically more novel)} part of the paper concerns oriented percolation. To state our results, we need some additional notation: First, as usual, we say that a sequence of events, $(\mathcal E_n)$, holds with high probability if the probability of $\mathcal E_n$ goes to $1$ as $n\to\infty$.
Next, given a digraph, let $C^+(v)$ (and $C^-(v)$) be  the set of nodes $w$ that can be reached by an oriented path from $v$ to $w$ (from $w$ to $v$).  We refer to these sets as the  fan-out (and fan-in) of $v$.
As usual, the set $SCC(v)=C^+(v)\cap C^-(v)$ is called the strongly connected component of $v$.  For a strongly connected component $SCC$, we  use the symbol $SCC^+$ for the set of nodes $SSC^+=\bigcup_{v\in SCC}C^+(v)$ and  the symbol $SCC^-$ for the set of nodes $SSC^-=\bigcup_{v\in SCC}C^-(v)$.
{Finally, we use the symbol $D_G(p)$ to denote the random digraph obtained from a graph $G$ by first replacing each edge $\{u,v\}$ by two oriented edges $uv$ and $vu$ and then keeping each oriented edge independently with probability $p$.}
%{We also introduce the notation $$\zeta_-(p)=\lim_{p'\uparrow p} \zeta(p'),$$ and note that $p>p_c(\mu)$ if and only if $\zeta_-(p)>0$.}

\newproof{Our next theorem  establishes the structure of $D_G(p)$ for any sequence that has a unique giant whose relative size converges in probability after undirected percolation, i.e., the sequence satisfies the {\it conclusion} of Theorem~\ref{thm: size of giant in expander}.

\begin{theorem}\label{thm: main structure SI}
Let $p\in (0,1]$ and let
$\{G_n\}_{n\in\mathbb{N}}$ be a sequence of (possibly random) graphs such that
\begin{enumerate}[(i)]
\item there exists $q\in (0,p]$ and a function 
$\zeta:[p-q,p]\to [0,1]$ that is left-continuous at $p$ such that
 $\frac{|C_1|}n\overset{\mathbb{P}}{\to} \zeta(p')$  for all $p'\in [p-q,p]$;
\item 
$\frac{|C_2|}n\overset{\mathbb{P}}{\to} 0$ uniformly in $[p-q,q]$.
 \end{enumerate}
\noindent  Let $SCC_i$ be the $i^{th}$ largest strongly connected component in $D_{G_n}(p)$.  
{Then} 
\begin{enumerate}
    \item \label{part: thm2 SCC2} Uniformly for all $p'\in [p-q,q]$,
    $$\frac{|SCC_2|}{n}\overset{\mathbb{P}}{\to} 0.$$ 
 
  \item\label{part: thm2 subcritical} If $\zeta(p)=0$ and $v\in {V(G_n)}$  is chosen uniformly  at random, then
 $$
\frac{ |{C}^+(v)|}n\overset{\mathbb{P}}{\to} 0,\quad
\frac{ |{C}^-(v)|}n\overset{\mathbb{P}}{\to}  0\quad\text{and}\quad
\frac{|SCC_1|}{n}\overset{\mathbb{P}}{\to} 0.
$$
%where in the first two cases, convergence in probability includes the randomness of $v$.

\item\label{part: thm2 bowtie} If    $\zeta(p)>0$  then
$$
\frac{|SCC_1^+|}{n}\overset{\mathbb{P}}{\to} \zeta(p),
\quad
\frac{|SCC_1^-|}{n}\overset{\mathbb{P}}{\to} \zeta(p)
\quad\text{and}\quad
\frac{|SCC_1|}{{\mathbb E}[|SCC_1|]}\overset{\mathbb{P}}{\to} 1,
$$
with $$\liminf_{n\to\infty}\frac 1n{\mathbb E}[|SCC_1|]\geq \zeta^2(p).$$
Furthermore, if $v\in V(G_n)$ is chosen uniformly at random, then
with high probability, the following two statements  hold:
\begin{itemize}
\item either $v\notin SCC_1^-$ and $|{C}^+(v)|=o(n)$ or $v\in SCC_1^-$ and $|C^+(v)\Delta SCC_1^+|=o(n)$;
\item either
$v\notin SCC_1^+$ and $|{C}^-(v)|=o(n)$ or $v\in SCC_1^+$ and $|C^-(v)\Delta SCC_1^-|=o(n)$.
\end{itemize}
In particular, if $v\notin SSC_1^+\cup SCC_1^-$, both  $|{C}^+(v)|=o(n)$ and $|{C}^-(v)|=o(n)$.
\end{enumerate}

\end{theorem}

 }

\begin{figure}[htbp]\label{fig: bow-tie}
% \centering
        \resizebox{0.6\textwidth}{!}{
		\begin{tikzpicture}
			\node [rotate=270,trapezium, trapezium angle=60, minimum width=50mm, draw, thick] at (1, 0) {};
				 
			\node [rotate=90,trapezium, trapezium angle=60, minimum width=50mm, draw, thick] at (5, 0) {};

			\draw [line width=1,fill=white] (3, 0) ellipse (2 and 1.5);
			\node at (3, 0) {$SCC_1$};
			
			\node (A)[minimum size=4.2cm] at (3, 0) {};
	        \foreach \v in {-1, 0, 1}
					\draw[->,>=stealth] (A) -- +( 3 ,\v);
					
			\foreach \v in {-1, 0, 1}
					\draw[->,>=stealth] ( 0 ,\v) -- +(A);
					
			\draw[decorate, decoration={brace, amplitude=3pt}, line width=1] (-0.5, 2.7) -- (5, 2.7) node[pos=0.5, above=1pt] {$SCC_1^-\simeq L^+$};
						
			\draw[decorate, decoration={brace, amplitude=3pt,mirror}, line width=1] (1, -2.7) -- (6.5, -2.7) node[pos=0.5, below=1pt] {$SCC_1^+\simeq L^-$};

\coordinate (c) at (7,0);
  \draw[gray,rounded corners=.5mm] (c) \irregularcircle{.3cm}{1mm};
  \coordinate (e) at (9,0);
  \draw[gray,rounded corners=.5mm] (e) \irregularcircle{.3cm}{1mm};
    \coordinate (f) at (8,-1);
  \draw[gray,rounded corners=.5mm] (f) \irregularcircle{.3cm}{1mm};
  
			\node[gray] at (8, -1.5) {disconnected components};
		\end{tikzpicture}
		}
		\caption{The structure of oriented percolation on expanders implied by Theorem \ref{thm: main structure SI}. In the supercritical case, there exists a unique linear sized SCC, $SCC_1$, that almost all nodes with large fan-out reach that (are in $SCC_1^-$), and all nodes with large fan-in are reachable from it (are in $SCC_1^+$). }
		
\end{figure}
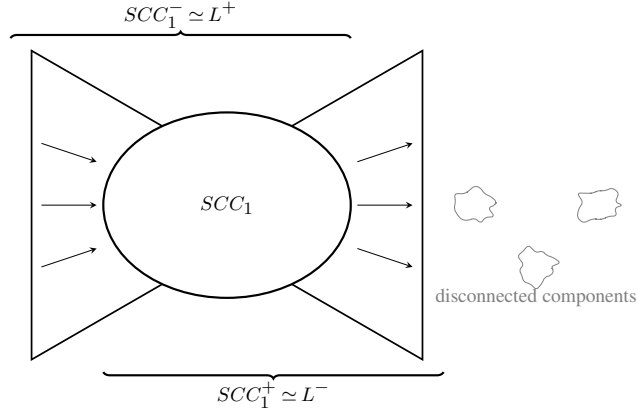

\begin{remark}
 The theorem implies the following ``bow-tie structure'' for $D_{G_n}(p)$ {when $\zeta(p)>0$:}
 Define the  bow-tie as the induced graph on $SCC_1^-\cup SCC_1^+$ with the left and right wing being given by the vertices in $SCC_1^-\setminus SCC_1$ and $SCC_1^+\setminus SCC_1$, respectively.  The theorem then implies that (up to $o(n)$ exceptions) with high probability all vertices not in the bow-tie will have fan-in and fan-out of size $o(n)$. In fact, all but at most $o(n)$ vertices  fall into one of {the following} four classes: (i) the giant strongly connected component (the center of the bow-tie), consisting of the vertices with large fan-in and large fan-out; (ii) the left (and  (iii) the right wing) consisting of the vertices with large fan-out and small fan-in (large fan-in and small fan-out), and (iv) the remaining ``dust'', consisting of vertices which have small fan-in and fan-out%
\footnote{The theorem is actually slightly stronger, since it says that this structure holds even if we define large fan-in and fan-out by requiring only that these sets contain order $n$ vertices - for almost all vertices, large fan-in or fan-out will then automatically give fan-in / fan-out of at least the size of the giant SCC.}.  See Figure~\ref{fig: bow-tie} for a demonstration.
%  Defining a vertex to be a vertex with large fan-in/fan-out to be a vertex with fan-in/fan-out of size at least $|SCC_1|$,  

The bow-tie structure was first described on an experimental analysis of the web graph \cite{Broder2000}. Later, Cooper and Frieze \cite{cooper_frieze_2004} {established}
the bow-tie structure of directed configuration model with maximum degree $o(n^{1/12})$. Later work, weakened the condition on maximum degree to $o(n^{1/4})$ in \cite{Graf}, and $o(\sqrt n)$ in \cite{cai2020giant}.
To the best of our knowledge, this is the first work showing the bow-tie structure {for oriented percolation on}
general expanders {(and more generally, oriented percolation on any model for which the conclusions of Theorem~\ref{thm: size of giant in expander} hold).}

%{Note that when $\zeta(p)=0$, the assumptions of the theorem reduce to the assumption that $|C_1|/n\to 0$ in probability (to see this, note that $|C_2|$ for $p'\in (p-q,p]$
%is bounded by $|C_1|$ in that interval, which by monotonicity is bounded by $|C_1|$ at $p$.  Thus absence of a giant at $p$ implies absence of a strongly connected giant as well as absence of vertices with large fan-in or large fan-out for all $p'\leq p$.}

% \ya{Finally, let us remark that our results extend to families of graphs beyond expanders, provided more is known about their un-oriented percolation.  More precisely, we can prove the results of Theorem~\ref{thm: main structure SI} if we assume local convergence together with the {\it conclusions} of Theorem~\ref{thm: size of giant in expander}.  This should make it possible to establish a bow-tie structure for other models, like percolation in high dimensions, where uniqueness and the asymptotic size of the giant for unoriented percolation can be established by other means.  We expand on this in Remark~\ref{remark: generalizing theorem 2} in Section~\ref{sec: oriented expander}.}
 \end{remark}
 
%   \begin{remark}We show in Lemma~\ref{lm: zeta right continuous} below that $\zeta$ is right continuous; since it is clearly monotone, one easily sees that the condition $\zeta_-(p)>0$ in the first statement is equivalent to  $p>p_c(\mu)$; therefore the first statement shows that the percolation threshold for the limit $(G,o)$ is also a threshold for the appearance of a giant SCC in $D_{G_n}(p)$, in addition to being a threshold for the appearance of a giant in $G_n(p)$.
%  \end{remark}

\begin{remark}\label{remark: pc for SCC}
The bow-tie structure established in  Theorem~\ref{thm: main structure SI} says in particular that the size of the giant $SCC$ is asymptotically equal to the number of vertices whose fan-in and fan-out is linear in $n$.  {Under the assumption of local convergence of $G_n$ in probability, }
one might therefore conjecture that $\frac 1n|SCC_1|$ converges in probability to
\begin{equation}\label{eq: zeta+-}
 \zeta^{+-}(p)=
 \mu\Big(\mathbb P_{D_{G}(p)}(|C^+(o)|=\infty\text{ and }|C^-(o)|=\infty)\Big).
\end{equation}
While our technology is not strong enough to prove this for general limits $\mu$,
%under the assumptions of Theorem~\ref{thm: size of giant in expander}, 
%While this is annoying from a structural point of view,  we note that for applications, the asymptotic number of vertices with large fan-out or large fan-in is actually the more relevant quantity: as pointed out before, they represent the probability of an outbreak (or a campaign targeting a random node being successful), as well as the number of people eventually infected in an outbreak (in a successful campaign).
we can prove that $\zeta_{+-}(p)$ is an asymptotic upper bound on $\frac 1n |SCC_1|$ whenever $G_n$ converges locally in probability to $\mu$, see Lemma~\ref{lem: upper bound SCC1} below.
%the critical probability for the occurrence of a giant SCC is the point where  $\zeta^{+-}(p)$ becomes positive; in addition, we can also establish that $ \zeta^{+-}(p)$ is an asymptotic upper bound on ; in fact, for this proof we don't even need expansion, 
\newproof{See also Remark~\ref{rem:SCC-tree} for a simple case where we can prove that $\frac 1n|SCC_1|$ converges to $ \zeta^{+-}(p)$ in probability.}
\end{remark}

%  \begin{remark}
% \ya{
% We conjecture that the continuity assumption of $\zeta$ can be lifted for expanders. In fact, the indistinguishability theorem of Lyons and Schramm \cite{lyons2011indistinguishability} implies continuity of $\zeta$ for quasi-transitive graphs. However, to our knowledge this is not known for expanders.}
% \end{remark}

Let us mention two applications of our results.  The first one is the  SIR (Susceptible-Infected-Recovered) infection model with fixed recovery time. In this model, each node can have three states: susceptible, infected or recovered. Each infected node  {infects each} of its neighbors independently according to a Poisson process with  rate of $\lambda$, and recovers after a fixed time (say one unit of time).  So, an infected vertex has an opportunity to infect any of its neighbors independently with probability $p=\frac{\lambda}{\lambda+1}$.   While in this model, an actual infection starting from a particular node gives an infection tree describing all nodes that get eventually infected, it is often useful to capture the structure of a possible infection independently of the initial node, by defining an \emph{infection digraph} which in our notation is nothing but the random digraph $D_G(p)$.  This gives a coupling of the infections starting at all possible seed vertices $v$, with the fan-out of $v$ being exactly the set of nodes getting sick eventually in an infection starting at $v$.  
{The structure of the bow-tie following from Theorem~\ref{thm: main structure SI} then implies that with high probability
\begin{itemize}
    \item an infection
starting in $SCC_1^-$ will infect all vertices in $SCC_1^+$, plus at most $o(n)$ extra vertices;
\item with the exception of up to $o(n)$ vertices, an infection starting in the complement of $SCC_1^-$ will only infect $o(n)$ other vertices.
\end{itemize} 
Together with the third statement of the theorem, we conclude that  if we infect {a} uniform random vertex in the network, the asymptotic probability and {the} size of an outbreak is $\zeta(p)$ and 
$n\zeta(p)$, respectively.}

Another application of  oriented percolation concerns information cascades. In this model, agents (nodes) are either informed or uninformed. Once an agent is informed they have only one chance to communicate the information to any of their contacts (neighbors in network), and the information will be shared successfully with probability $p$.
This is a special case of information cascade model considered in \cite{Kempe-Kleinberg} and {many follow ups}, where the success probability over all edges of the network is equal to $p$.  Similar to the infection digraph in SIR model, one can define an information digraph in which a directed edge from $u$ to $v$ represents the event that conditioned on $v$ being the first note to be informed, the information is shared  successfully along the edge $uv$.  As a result of our theorem on oriented percolation, one can estimate the {expected number} of  nodes that {will have been informed at the end of the cascade if the initial seed is chosen uniformly at random, or more generally, if a set of initial seeds {are} chosen uniformly at random, provided} the underlying network is an expander with bounded average degree that {converges locally in probability.}

We close this introduction with a final remark.
% The bow-tie structure established in  Theorem~\ref{thm: main structure SI} says in particular that the size of the giant $SCC$ is asymptotically equal to the number of vertices whose fan-in and fan-out is linear in $n$.  One might therefore conjecture that $\frac 1n|SCC_1|$ converges in probability to
% \begin{equation}\label{eq: zeta+-}
%  \zeta^{+-}(p)=
%  \mu\Big(\mathbb P_{D_{G}(p)}(|C^+(o)|=\infty\text{ and }|C^-(o)|=\infty)\Big).
% \end{equation}
% Unfortunately, our technology is not quite strong enough to prove this under the assumptions of Theorem~\ref{thm: main structure SI}.  While this is annoying from a structural point of view,  we note that for applications, the asymptotic number of vertices with large fan-out or large fan-in is actually the more relevant quantity: as pointed out above, they represent the probability of an outbreak (or a campaign targeting a random node being successful), as well as the number of people eventually infected in an outbreak (in a successful campaign).
% Note also that we can prove that the critical probability for the occurrence of a giant SCC is the point where  $\zeta^{+-}(p)$ becomes positive; in addition, we can also establish that $ \zeta^{+-}(p)$ is an asymptotic upper bound on $\frac 1n |SCC_1|$; in fact, for this proof we don't even need expansion.  Existence of the  local limit is enough, see
%  Lemma~\ref{lem: upper bound SCC1} below.

{
\begin{remark}
In parallel to our work, Remco van der Hofstad developed a different approach to the locality of the giant in unoriented random graphs \cite{Hofstad2021}.  He showed that if the random graph sequences has a local limit, and the property that it is unlikely that two random vertices lie in distinct, large components, the relative size of the giant is given by the probability that the origin in the limit lies in an infinite component.  This hints at a possible, alternative approach to proving the first statement in Theorem~\ref{thm: size of giant in expander}; as in our proof, one would first establish local  convergence of the percolated sequence in probability, but then use a sprinkling argument to prove that the \emph{assumptions} of \cite{Hofstad2021} are satisfied at the points of continuity of $\zeta(p)$, rather than directly proving our  Proposition~\ref{prop: lower bound giant} below. 
Unfortunately, an application of our {sprinkling} arguments only gives a condition which is weaker than required for an application of the results of \cite{Hofstad2021}.  So at the moment, the two methods seem to be complementary, establishing uniqueness of the giant for different sets of random graphs.
\end{remark}
}

\subsection{Overview of the Paper}
In Section \ref{sec: defs}, we set up notations and terminology, including the notion of local  convergence {in probability}
(Section \ref{sec: LWC}) and the formal definition of the threshold for the appearance of a giant and strongly connected giant (Section \ref{sec: defi percolation}). 
Finally, {in Section~\ref{sec:zeta-cont} we discuss the continuity of $\zeta$}, and in Section \ref{sec: Russo}, we  review a concentration bound that follows from a beautiful result of Falik and Samorodnitsky \cite{falik} and will be used later in the proof of Theorem~\ref{thm: main structure SI}.  The reduction
of the concentration bound to the results of \cite{falik}
is given in Appendix \ref{apx: Russo}.

Theorems \ref{thm: size of giant in expander} is proved in 
Section \ref{sec: unoriented expander}, {where on the way of proving it we also show that the relative size of the second largest component in expanders with bounded average degree converges uniformly to zero (see Lemma \ref{lem: uniform convergence second component}). The proof of this lemma {uses the techniques of Alon, Benjamini and Stacey \cite{alon2004}, {extending} their results for expanders with bounded degrees to  {large set} expanders with bounded average degrees and local limit in probability, and} {is given} in Appendix \ref{app:C2}}.

Section \ref{sec: coupling} is a stand-alone section
that explores the relation between oriented and unoriented percolation on general graphs via a natural coupling.
 Building upon this coupling {and our results for unoriented percolation}, Theorem \ref{thm: main structure SI} is then proved in Section~\ref{sec: oriented expander}. {One of the main technical difficulties in this section is the proof of concentration of the size of the strongly connected giant, without having an explicit formula for its expectation.  While a proof based on Russo's lemma and suitable bounds on influences might seem natural to the expert, it turns out be quite tricky for the oriented case, due to the fact that a single edge can join many small strongly connected components which without this edge were just sitting on a directed path, ``without a path back''.  This makes the size of strongly connected giant  much less ``local''  than the undirected analog.  }

In  Section \ref{sec: pref attachment}, we apply  Theorem \ref{thm: size of giant in expander} to   preferential attachment models.
{The details of many of the proofs are deferred to appendices:}
{The proof of expansion for
preferential attachment models
is
proven in Appendix \ref{sec: expansion pref attachment}; and finally,
upper and lower bounds for the survival probability 
of the limiting branching process after percolation are proven in
Appendix \ref{sec: appendix polya point}.
}

%%%%%%%%%%%%%%%%%%%%%%%%%%%%%%%%%%%%%%%%%%%%%%%%%%%%%%%%%%%%%
%% 			Unoriented Expanders. 		%%

\section{Notations, Definitions, and  Preliminaries}\label{sec: defs}
For a graph $G$, let $V(G)$ be the set of vertices and $E(G)$ be the set of edges. 
{As usual, a rooted graph is a graph with one particular node $v$ designated as the root; We will use the notation  $(G,v)$ to denote a rooted graph with root $v$.
A graph isomorphism between two graphs $G_1$ and $G_2$ is a bijection  $\phi:V(G_1)\to V(G_2)$ such that $\{v,w\}\in E(G_1)$ if and only if
$\{\phi(v),\phi(w)\}\in E(G_2)$. If the two graphs are rooted, we also require that $\phi$ maps the root of $G_1$ to that of $G_2$.  {We will use $\mathcal{G}_*$
to denote the space of equivalence classes of locally finite, rooted graphs under these isomormphisms.}

The $k$-neighborhood of a vertex $v$ in $G$ is defined as the induced subgraph on the set of nodes of graph distance at most $k$ from $v$, and will be denoted by $B_k(G,v)$.  If $G$ is clear from the context, we just write $B_k(v)$ instead of $B_k(G,v)$.  }

\subsection{Local  Convergence in Probability}\label{sec: LWC}

As usual,
local  convergence \cite{Aldous2004, benjamini2001} is defined in terms of a  metric $d_{loc}$ on  
$\mathcal{G}_*$:
given two rooted graphs $(G_1,o_1)$ and $(G_2,o_2)$, their ``local distance''
 is defined as
\[d_{loc}((G_1,o_1),(G_2,o_2))=\frac{1}{1+\inf_k\{B_k(G_1,o_1)\not\simeq B_k(G_2,o_2)\}},\]
where $\simeq$ denotes equivalence under isomorphisms which map the roots $o_1$ and $o_2$ into each other.
The function $d_{loc}$ defines a topology on the space of rooted graphs and local convergence of a sequence of graphs is defined with respect to that topology.     
  {Since the}
finite graph $G_n$   {we consider is typically not} 
%is not necessarily 
rooted,
% Instead, 
we choose a root uniformly at random,
\[\mathcal{P}_n=\frac{1}{n}\sum_{o_n\in V(G_n)}\delta_{(G_n,o_n)}.\]
For non-random sequence $G_n$, local weak convergence to a measure $\mu$ on $\mathcal{G}_*$ is defined by the requirement that $\mathbb{E}_{\mathcal{P}_n}[f]{\to}\mathbb{E}_\mu[f]$ for all bounded continuous functions $f$ on $\mathcal{G}_*$.  \newproof{If $G_n$ is random, there are three commonly considered notions of {local} convergence:  convergence in distribution, convergence in probability, and almost sure convergence, see Chapter 2 in \cite{RemcoVol2} for an overview.   For convergence in distribution (also called annealed), one requires that the expectations of $f$ with respect to both $P_n$ and the randomness of $G_n$ converge, while for the other two, the randomness of $G_n$ is fixed (quenched). 
 In this paper, we will use convergence in probability for the quenched version which we now define formally.}

 Consider thus a sequence of random graphs $\{G_n\}_{n\in\mathbb{N}}$, and a {(non-random)} probability $\mu$ on $\mathcal{G}_*$. We say  $G_n$ converges \emph{locally} in probability  to $\mu$ if for any bounded continuous function $f:\mathcal{G}_*\longrightarrow \mathbb{R}$,
\begin{equation}\label{eq: lwc definiiton}
\mathbb{E}_{\mathcal{P}_n}[f|G_n]\overset{\mathbb{P}}{\to}\mathbb{E}_\mu[f],
\end{equation}
where in $\mathbb{E}_{\mathcal{P}_n}[f|G_n]$ we only take expectations with respect to the random root in $G_n$. So, $\mathbb{E}_{\mathcal{P}_n}[f|G_n]$ can be random variable due to the conditional dependence on graph $G_n$ (in the case where the sequence $G_n$ is random). 
While in principle, convergence in probability allows for convergence to a random measure $\mu$, in which case $\mathbb{E}_\mu[f]$ would be random, in this paper, we will assume that the limiting measure $\mu$ on rooted graphs is non-random.

\newproof{Note that this restriction rules out certain random graph sequences: a sequence $\{G_n\}_{n\in \mathbb N}$ where $G_n$ is a random $3$-regular graph on $n$ nodes with probability $1/2$, and a random $4$-regular graph on $n$ nodes with probability $1/2$ will not have a deterministic limit $\mu$, while the union of two disconnected graphs of the same size where one is $3$-regular and one is $4$ regular has a deterministic limit  $\mu$ (with $\mu(G,o)$ being $1/2$ if $G$ is a $3$ or $4$ regular infinite tree, and $0$ otherwise).  Note that by contrast, both sequences converge to this deterministic measure $\mu$  if we consider convergence in distribution.}

\subsection{Thresholds for the Existence of a Giant for Unoriented and Oriented Percolation}\label{sec: defi percolation}

{The notion of a threshold for the appearance of various structures in finite, random graphs is a well known concept from random graph theory, 
%see, e.g., \cite{JansonBook}, 
with the question of the locality of the threshold for the appearance of a giant for percolation on bounded degree expanders being part of the literature on which this paper is building. For the convenience of the reader, and to define these concepts for the oriented case, we give a precise definition below.}

{We  use $\mathbb{P}_{G(p)}$ and $\mathbb{E}_{G(p)}$ to denote probabilities and expectations with respect to percolation on a graph $G$, and $\mathbb{P}_{D_G(p)}$ and $\mathbb{E}_{D_G(p)}$ for the oriented analogues. } Expectations with respect to the distribution $\mu$ of rooted random graphs describing the  local limit are denoted by $\mathbb{E}_\mu$. 
{Finally, we use the standard notation}  $a\vee b$ and $a\wedge b$ for the maximum and minimum of two real number $a$ and $b$.

We will say that a sequence $p_n$ is a
threshold sequence\footnote{When $p_n>\alpha$ for some $\alpha>0$ independent of $n$ {this
coincides with the notion of a sharp threshold, as  defined, e.g.,} in \cite{JansonBook}.} for the existence of a giant component (in short, the percolation threshold) if  for
  {all} $\epsilon>0$ 
and 
$c>0$
\[\mathbb{P}\big(G_n(   0\vee (p_n-\epsilon))\text{ contains a component of size at least } cn \big)\rightarrow 0 ,\]
and for   all $\epsilon>0$ 
there exists some $c>0$   {such} that
\[\mathbb{P}\big(G_n(1\wedge   (p_n+\epsilon))\text{ contains a component of size at least } cn \big)\rightarrow 1.\]
With a slight abuse of notation, we will write $p_n=p_c(G_n)$ to denote a threshold sequence\footnote{{Strictly speaking, the correct formal notation would be $(p_n)\in p_c(G_n)$ to stress the fact that $(p_n)$ remains a threshold sequence if we add a tern $\epsilon_n$ which goes to $0$ as $n\to\infty$.}}.
If $G_n$ is random, the above probabilities are with respect to percolation,
$\mathbb P=\mathbb P_{G_n(0\vee (p_n-\epsilon))}$ and 
$\mathbb P=\mathbb P_{G_n(1\wedge (p_n-\epsilon))}$, and convergence becomes convergence in probability (with respect to the randomness of $G_n$). 
We define the critical  threshold for the appearance of a giant SCC in $D_{G_n}(p)$ in the same way, with the only difference being that the word ``components'' is replaced by strongly connected components; we will use the notation $p_c^{SCC}(G_n)$ such a threshold sequence.

\begin{remark}
These two definitions immediately raise the question whether the two thresholds are related.  First, it turns out that asymptotically, the two must be the same if they both exist (see Corollary~\ref{cor: SCC threshold general graph} below).  But even without the assumption that both exist, we know quite a bit; in fact, without any prior assumptions on the existence of either threshold, we know that if the probability
that $G_n(p_n)$ contains a giant of size $c n$ or larger goes to $0$, then the probability that $D_{G_n}(p_n)$
contains a giant SCC of this size goes to zero as well (Corollary~\ref{cor:pc-lower-bd} below).  In the other direction we know that if the expectation of the giant component in
$G_n(p_n)$ is bounded below by $cn$ for some $c>0$, then
the expectation of the giant SCC in $D_{G_n}(p_n)$ is bounded by $c'n$ for some {$c'>0$}.  While this allows us to conclude that
once $G_n(p_n)$ has a giant with high probability, the expectation of the giant SCC in $D_{G_n}(p_n)$ is at least of order $n$, this does not imply existence of a giant SCC  with high probability; for arbitrary sequences of graphs, we just don't have enough control over the variance.
\end{remark}

Recall our definition \eqref{pc-limit}
of the percolation threshold for an 
infinite rooted graph $(G,o)$ with law $\mu$.
While 
it might seem natural to define $p_c^{SCC}({\mu})$ for the appearance of a strongly connected component similarly, the naive definition turns out not to be useful for identifying the threshold of a locally convergent sequence. {This is because} a giant strongly connected component for graphs with large girth might not correspond to a SSC in the limit graph.
Instead, we consider the event that both the fan-in and the fan-out of the root is infinite,
\[p_c^{+-}(\mu)=\inf\{p\in[0,1]:\zeta^{+-}(p)>0\},\]
with $\zeta^{+-}(p)$ given by \eqref{eq: zeta+-}.
With this definition, we will prove $p_c^{+-}(\mu)$ is the threshold 
for the appearance of a giant SCC in $D_{G_n}(p)$ if
$G_n$ is a sequence of large-set expanders with bounded average degrees that converges to $(G,o)$.
In fact, it is not hard to see
that % \todo{Removed the inequality between the zetas}
%$$
%\lim_{p'\uparrow p}
%\zeta(p)^2\leq \zeta^{+-}(p)\leq \zeta(p),
%$$
%%for arbitrary {distributions $\mu$ on random rooted graphs} 
%(see Lemma \ref{lm: zeta+- > zeta^2}).  This shows that
 for arbitrary rooted random graphs $$p_c^{+-}(\mu)=p_c(\mu)$$
 (see Lemma \ref{lm: zeta+- > zeta^2}), consistent with the fact  that if both the appearance of a giant and the appearance of a giant SCC have a threshold sequence, the two must be asymptotically equal (Corollary~\ref{cor: SCC threshold general graph}).

\newproof{
 \subsection{Continuity of $\zeta$}\label{sec:zeta-cont}
 In this section, we elaborate on Remark~\ref{rem:zeta-cont}. To this end, we
recall that when $\mu$ is a local limit of some (possibly random) sequence $G_n$, it obeys a symmetry relation known as  
 unimodularity \cite{benjamini2001}, see,   e.g., \cite{alon2004} for the definition of unimodularity.  
 
 Next, we point out that our definition of $p_c(\mu)$ in \eqref{pc-limit} differs from the standard definition of $p_c$ for random rooted graphs, as in, e.g., in \cite{lyons2011indistinguishability}.  In particular, in these papers, $p_c$ is a function of the random graph $(G,o)$ drawn from $\mu$, i.e., 
 $$p_c(G,o)=\inf\{p\colon\exists x\in V(G)\text{ s.t. }{{\mathbb P}_{G(p)}}(|C(x)|=\infty)\},
 $$
 with $C(x)$ denoting the connected cluster of $x$ in $G(p)$.
 If the limit is extremal in the set $\mathcal U$ of unimodular measures then $p_c(G,o)$ is almost surely a constant, and in that case, it will be equal to our definition of $p_c(\mu)$ in \eqref{pc-limit}.
 
 Finally, for extremal measures in $\mathcal U$, Aldous and Lyons showed that if $p_c(\mu)<p_1<p_2$, then $\mu$-almost surely, every
 infinite cluster in $G(p_2)$ contains an infinite cluster in $G(p_1)$ (Theorem 6.7 in \cite{aldous2007processes}). This in turn implies continuity of $\zeta$ for all $p\neq p_c(\mu)$ by the standard arguments (see, e.g., \cite{van1984continuity}).  Furthermore, 
 Theorem 8.11 from 
 \cite{aldous2007processes} gives continuity at $p_c$ for non-amenable extremal $\mu\in U$, which in particular holds for extremal trees with at least three ends, i.e., three disjoint infinite path in the tree under consideration.

 Extremality of $\mu$ was proven in \cite{sarkar2018note} when $\mu$ is the local weak limit of a non-random sequence of expanders of bounded degree.
 In Appendix~\ref{sec: contintuiy of zeta} we 
generalize this proof to measures  
 $\mu$ that arises as the local limit in probability of a 
 (possibly random) sequence of large-set expanders of uniformly bounded average degree.  As just explained, this immediately gives continuity of $\zeta$ for $p\neq p_c(\mu)$.

 \begin{corollary}\label{cor: zeta-cont}
Let $\{G_n\}$ be a sequence of (possibly random) large-set expanders of bounded average degree that converge locally in probability to $\mu$, and let $\zeta(p)$ and $p_c(\mu)$ be as in \eqref{zeta} and \eqref{pc-limit}. Then $\zeta(p)$ is continuous for all $p\neq p_c(\mu)$.
\end{corollary}
 }

 \subsection{Concentration Bounds}\label{sec: Russo}
 {One of the technically  difficult parts of this paper is the proof of concentration for the size of the giant strongly connected component without explicit control of its expectation.  To this end, we will use a
 concentration inequality going back to
the work of Falik and Samorodnitsky \cite{falik}.

 Given a positive integer $m$, let $x$ denote vectors $x=(x_e)_{e\in [m]}$, and for $0<p<1$, let $\mathbb P_p$
 be the independent product measure on $\{0,1\}^m$ with marginals $\mathbb P(x_e=1)=p$ (in our application, $m$ will be twice the number of edges in $G$, and $\mathbb P_p$ will be the oriented percolation measure $\mathbb P_{D_G(p)}$).}
  For $x\in \{0,1\}^m$ and $e\in [m]$, we use $x\oplus e$, $x\cup\{e\}$ and  $x\setminus \{e\}$ to denote the Boolean vector obtained from $x$ by flipping the bit $e$,  replacing it by $1$, or replacing it by $0$, respectively.
For an increasing function $f:\{0,1\}^m\to\mathbb R$, define the influence of an edge
\[
\Delta_e f(x)=f(x)-{\Big(p f(x\cup\{e\})+(1-p)f(x\setminus \{e\})\Big)},
\] 
and
\[\mathcal{E}_2(f)=\sum_{e\in[m]}\mathbb{E}_{p}
[|\Delta_e f|^2]\quad\text{and}\quad\mathcal{E}_1(f)=\sum_{e\in[m]}\mathbb{E}^2_{ p}
|\Delta_e f|.
\]
The results of Falik and Samorodnitsky \cite{falik} then imply the following bounds.
\begin{lemma} \label{lm: falik}
Let $f$,  $\mathcal{E}_1(f)$ and $\mathcal{E}_2(f)$ be  as above.
Then 
\[\frac{1-2p}{\log\frac{1-p}{p}}\operatorname{var}(f)\log\Big(\frac{\operatorname{var}(f)}{\mathcal{E}_1(f)}\Big)\leq \mathcal{E}_2(f).\]
\end{lemma}

For $p=1/2$ (with the prefactor replaced by its limit, $1/2$), the lemma
is essentially equivalent to Lemma 2.1 in \cite{Benjamini2006SubmeanVB}.  While apparently not realizing that its proof required the uniform measure, Lemma 2.1 from \cite{Benjamini2006SubmeanVB} was restated in \cite{benjamini2012} for general $p$, rendering several of the technical lemmas in that paper incorrect.  However, the needed changes do not invalidate the main results of \cite{benjamini2012}, since they all concern values of $p$ bounded away from zero, where the difference just amounts to a difference in an overall constant. We give the reduction of
Lemma~\ref{lm: falik} to the results of reference \cite{falik} in
Appendix \ref{apx: Russo}.

We close our preliminaries by recalling Russo's formula. 
Given an increasing event $A\subset \{0,1\}^m$ and $x\in \{0,1\}^m$, we define 
$e\in [m]$ to be pivotal if exactly one of $x$ and $x\oplus e$ is in $A$.  The Margulis--Russo formula  \cite{Russo1981OnTC} then says that
\begin{equation}\label{russo-1}
 \frac d{dp}\mathbb P_{p}(A)=\sum_{e\in [m]}\mathbb P_{p}(e\text{ is pivotal}).
\end{equation} 
We will also need the analog for an increasing function
$f:\{0,1\}^m\to\mathbb R$,  which states that
\begin{equation}\label{russo-2}
\sum_{e\in [m]}\mathbb E_{ p}\,[|\Delta_e f|]=2p(1-p) \frac d{dp}\mathbb E_{ p}\,[f].
\end{equation}

%%%%%%%%%%%%%%%%%%%%%%%%%%%%%%%%%%%%%%%%%%%%%%%%%%%%%%%%%%%%%
%%                  Section 3                       %%

\section{Locality of Unoriented Percolation on Expanders}\label{sec: unoriented expander}
In this section, we prove Theorem~\ref{thm: size of giant in expander}. 
To this end, 
we first establish a lemma stating that local  convergence  {in probability} allows for the control of both expectations and concentration of local quantities in the percolated graph $G_n(p)$ (Lemma~\ref{lm: local function - percolated graph}).  {Note that despite its apparent simplicity, the lemma is slightly subtle, and in particular does not hold if one only assumes local  convergence in distribution instead of local convergence in probability.}

Using this lemma, it will be straightforward to  upper bound the asymptotic size of the giant by $\zeta(p) n$.  For the lower bound, we will use a sprinkling argument, which in its simplest form goes back to Erd\H{o}s \cite{Erdos},
and is at the core of most previous work on locality in percolation including that of 
 Alon, Benjamini and Stacey \cite{alon2004}. 
%It does, however, require some continuity properties of the function $\zeta(p)$ and the local quantities approximating it.  The needed properties are stated in  Lemma~\ref{lm: zeta right continuous}.  
We recall that throughout this paper, we assume that all our sequences $(G_n)$ are growing; in particular, we assumed without loss of generality that $G_n$ is a graph on $n$ vertices, which may or may not be random.

\begin{lemma}\label{lm: local function - percolated graph}
Let $\mu$ be a {(non-random)} probability distribution on $\mathcal{G}_*$, and 
$\{G_n\}_{n\in\mathbb{N}}$ be a sequence of (possibly random) graphs % is it enough that in intro we said size of G_n grows?
that converge locally in probability  to $(G,o)\in\mathcal{G}_*$ with distribution $\mu$. For $p\in[0,1]$ and a positive integer $k$, let $f_k:\mathcal{G}_*\rightarrow\mathbb{R}$ be a bounded and continuous function defined on the $k$-neighborhood of a node. Then 
\[\mathbb{E}_{\mathcal{P}_n}[f_k(G_n(p),o_n){\mid G_n}]\overset{\mathbb{P}}{\to}\mathbb{E}_{\mu_p}[f_k(G(p),o)],\]
where convergence in probability is over the possible randomness of $G_n$ and percolation, and $\mu_p$ is {the deterministic measure on $ \mathcal G_*$} defined by first choosing  $(G,o)$ with respect to $\mu$, {then drawing a graph $G(p)$ via percolation, and finally replacing $G$ by the connected component of $o$ in $G(p)$.}
\end{lemma}
\begin{remark}
The lemma implies  that if 
$\{G_n\}_{n\in\mathbb{N}}$ is locally  convergent in probability to $(G,o)\sim \mu$, then 
$\{G_n(p)\}_{n\in\mathbb{N}}$ is locally  convergent in probability to $(G(p),o)\sim \mu_p$.  {To prove this, we need to extend the statement to all bounded, continuous functions, which in turn requires tightness.  But tightness is obvious here.  Indeed,}
all one needs to observe  is that if $\mathcal G_{N}$ is the set of all $(G,o)\in \mathcal G_*$ with at most $N$ nodes, then $\mu(\mathcal G_{ N})\geq 1- \epsilon$ for some $N$.  Local  convergence in probability then implies that the same statement  holds (with $\epsilon$ replaced by $2\epsilon$) for the probability distribution of $(G_n,x_n)$ and all large enough $n$, with probabilities with respect to both the randomness of $G_n$ and $x_n$.  But this property is inherited by the percolated graphs $G_n(p)$, which {gives the desired tightness.}
\end{remark}

\begin{proof}
{Recall that $\mathbb{E}_{G_n(p)}$ denotes expectation with respect to percolation only; if $G_n$ is random, these expectations are still conditioned on the random graph $G_n$.}
To prove lemma we use the second moment method.

\newproof{As a preliminary, we start with the observation that the distance of two vertices does not decrease after percolation i.e., if $dist_{G(p)}(x,y)\leq k$, then $dist_G(x,y)\leq k$ as well.  As a consequence, the values of $f_k$ evaluated on a percolated graph $(G(p),x)$ only depend on the induced subgraph on the vertices of distance at most $k$ from $x$ in $G$.  This has two consequences: (i) the expectation of $f_k$ with respect to percolation  depends only on the $k$-neighborhood of the root, and (ii) expectations of products factor if the roots are at least distance $2k+1$  apart.  Explicitly, if we define $f_{k,p}:\mathcal{G}_*\rightarrow\mathbb{R}$ by
$f_{k,p}(G,o)= \mathbb{E}_{G(p)}[f_k(G(p),o)] $, then (i)
$f_{k,p}(G_1,o_1)=f_{k,p}(G_2,o_2)$ whenever $d_{loc}\big((G_1,o_1),(G_2,o_2)\big)\leq\frac{1}{k+1}$, and (ii)
$\mathbb{E}_{G(p)}[f_k(G(p),o_1)f_k(G(p),o_2)] =f_{k,p}(G,o_1)f_{k,p}(G,o_2)$ whenever $dist_G(x,y)\geq 2k+1$.}

With these preparations, the proof now is relatively straightforward once we take into account a corollary from \cite{RemcoVol2} concerning the distance of two random vertices in sequences of locally convergent graphs, see below.

We start with the first moment, i.e., 
\[\mathbb{E}_{G_n(p)}\big[\mathbb{E}_{\mathcal{P}_n}[f_k(G_n(p),o_n)]\,\big|\, G_n\big]\overset{\mathbb{P}}{\to}\mathbb{E}_{\mu_p}[f_k(G(p),o)],\]
where convergence in probability is over possible randomness of $G_n$. 
By  the linearity of expectation, \[\mathbb{E}_{G_n(p)}\big[\mathbb{E}_{\mathcal{P}_n}[f_k(G_n(p),o_n)]{\,\mid G_n}\big]=\mathbb{E}_{\mathcal{P}_n}\big[\mathbb{E}_{G_n(p)}[f_k(G_n(p),o_n)]\,{\big| \,G_n}\big]
=\mathbb{E}_{\mathcal{P}_n}\big[f_{k,p}\mid G_n\big].\] 
%Since $f(G,o)$ is defined on the $k$-neighborhood of $o$ in $G$, the
%function $f_{k,p}= \mathbb{E}_{G(p)}[f_k(G(p),o)] $ is defined in a $k$-neighborhood as well. Indeed, $f_{k,p}(G_1,o_1)=f_{k,p}(G_2,o_2)$ whenever $d_{loc}\big((G_1,o_1),(G_2,o_2)\big)\leq\frac{1}{k+1}$.
{Our observation (i) above,} together with the assumption that $f$ is bounded, 
%this 
shows that  $f_{k,p}$ is a bounded continuous function on $\mathcal{G}_*$. 
% Let $f_{k,p}=\mathbb{E}_{G_n(p)}[f_k{\mid G_n}]$, then $f_{k,p}$ is a bounded continuous function on $\mathcal{G}_*$. {Indeed,} if $d_{loc}\big((G_1,o_1),(G_2,o_2)\big)\leq\frac{1}{k+1},$ 
% then $B_k(G_1,o_1)\simeq B_k(G_2,o)$, and as a result, $f_{k,p}(G_1,o_1)=f_{k,p}(G_2,o_2)$.
By the definition of local  convergence in probability, we have that \[\mathbb{E}_{\mathcal{P}_n}[f_{k,p}{\mid G_n}]\overset{\mathbb{P}}{\to}\mathbb{E}_\mu[f_{k,p}(G,o)]=\mathbb{E}_{\mu_p}[f_{k}(G(p),o)].\]
 
\newproof{For the second moment,  we compute
\[\mathbb{E}_{G_n(p)}\Big[\Big(\mathbb{E}_{\mathcal{P}_n}[f_k(G_n(p),{o_n})]\Big)^2\,\Big|\, G_n\Big]
=
%\mathbb{E}_{G_n(p)}\big[\big(\sum_{v\in [n]}f_k(G_n(p),v)\big)^2{\mid G_n}\big]=
\frac 1{n^2}\sum_{u,v\in[n]} \mathbb{E}_{G_n(p)}\big[f_k(G_n(p),v)f_k(G_n(p),u){\,\big|\, G_n}\big],\]
which we write as
\begin{align*}\label{eq: second momemnt fk}
\frac 1{n^2}\sum_{u,v\in[n]} &\mathbb{E}_{G_n(p)}\Big[f_k(G_n(p),v) f_k(G_n(p),u){\mid G_n}\Big]=\\
=&\frac1{n^2}\sum_{u,v\in[n]:\atop dist_{G_n}(u,v)\leq r}\mathbb{E}_{G_n(p)}\big[f_k(G_n(p),v) f_k(G_n(p),u){\mid G_n}\big]\\
&\qquad\qquad+\frac 1{n^2}\sum_{u,v\in[n]:\atop dist_{G_n}(u,v)> r}\mathbb{E}_{G_n(p)}\big[f_k(G_n(p),v) f_k(G_n(p),u){\mid G_n}\big]
\end{align*}
{where we choose $r=2k$.  By  our observation (ii) from the beginning of this proof, the second term can be rewritten as}
\begin{align*}
&\sum_{u,v\in[n]:\atop dist_{G_n}(u,v)> r}\mathbb{E}_{G_n(p)}\big[f_k(G_n(p),v) f_k(G_n(p),u){\mid G_n}\big]\\
&\qquad=\sum_{u,v\in[n]:\atop dist_{G_n}(u,v)> r}\mathbb{E}_{G_n(p)}\big[f_k(G_n(p),v){\mid G_n}\big]\mathbb{E}_{G_n(p)}\big[f_k(G_n(p),u){\mid G_n}\big],
\end{align*}
{implying that
\begin{align*}
\Big|\mathbb{E}_{G_n(p)}\Big[\Big(\mathbb{E}_{\mathcal{P}_n}[f_k(G_n(p),{o_n})]\Big)^2\Big| G_n\Big]&-
\Big(\mathbb{E}_{G_n(p)}\Big[\mathbb{E}_{\mathcal{P}_n}[f_k(G_n(p),{o_n})]\Big| G_n\Big]\Big)^2\Big|\leq\\
\leq 2\alpha^{(2)}_{\leq r}(G_n))\|f\|_\infty
\end{align*}
where $\alpha^{(2)}_{\leq r}(G_n)$ is the fraction of pairs $u,v\in V(G_n)$ such that $dist_{G_n}(u,v)\leq r$.}
By Corollary 2.20 in \cite{RemcoVol2}, for any two vertices $u$ and $v$ chosen independently and uniformly at random from $[n]$, the distance between them grows with $n$, i.e., $dist_{G_n}(u,v)\overset{\mathbb{P}}{\to}\infty$,
{where the probability is with respect to both the randomness of $G_n$ and the random choice of $u$ and $v$.  But this implies that $\alpha^{(2)}_{\leq r}(G_n))\to 0$ in probability (with respect to the random choice for $G_n$),}
and thus}
\[\frac{\operatorname{var}_{G_n(p)}\Big(\mathbb{E}_{\mathcal{P}_n}(f_k){\mid G_n}\Big)}{\mathbb{E}^2_{G_n(p)}\Big(\mathbb{E}_{\mathcal{P}_n}(f_k){\mid G_n}\Big)}\overset{\mathbb{P}}{\to}0,\]
where  convergence in probability is on random graphs $G_n$. So, by Chebyshev's inequality,
\[\frac{\mathbb{E}_{\mathcal{P}_n}(f_k{\mid G_n})}{\mathbb{E}_{G_n(p)}[\mathbb{E}_{\mathcal{P}_n}(f_k){\mid G_n}]}\overset{\mathbb{P}}{\to}1,\]
Then by convergence of the first moment,
\[\mathbb{E}_{\mathcal{P}_n}(f_k{\mid G_n})\overset{\mathbb{P}}{\to} \mathbb{E}_{\mu_p}(f_k).\]
\end{proof}

Next we state a lemma which will be used in our sprinkling argument.
 Its statement  (and its proof) are
similar to those used in the work of Alon Benjamini and Stacey \cite{alon2004} and later in \cite{benjamini2009critical,Krivelevich20HighGirthExapnder}, but we avoid the assumption of uniformly bounded maximal degree, and only uses large-set expansion instead of expansion.  
%This extension is also the main ingredient for extending the results of Sarkar concerning the extremality of $\mu$ \cite{sarkar2018note} to our setting, see 
%~\ref{sec: contintuiy of zeta} for details. 
We  recall \eqref{phi} and the definition of 
 $(\alpha,\epsilon,\bar d)$-large-set expanders as graphs $G$ with average  degree at most $\bar d$ and $\phi(G,\epsilon)\geq \alpha$. 
\begin{lemma}[Sprinkling Lemma]\label{lem: sprinkling}
Let  $G$ be an $(\alpha,\epsilon,\bar d)$-large-set expander on $n$ vertices, and let $H$ be an instance of $G(\beta)$ for some $\beta{\in (0,1]}$.
Given $R>0$,  let $\mathcal{S}$ be a family of disjoint  subsets of $V(G)$, each of size at least $R$.  For two sets $A$ and $B$ define an $A$-$B$ {path} as a path with one endpoint in $A$ and another in $B$. Then 
\[\mathbb{P}(\exists A,B\subset 
\mathcal{S}: |\bigcup_{V_i\in A} V_i|\geq \epsilon n,  |\bigcup_{V_i\in B} V_i|\geq \epsilon n, \text{no $A$-$B$ path in $H$})\leq e^{\frac{n}{R}-cn},\]
where $c$ is a constant that depends on $\beta$, $\epsilon$, $\alpha$, and $\bar d$, but it is independent of $R$ and $n$. 
\end{lemma}

\begin{proof}
Let $A$ and $B$ be two disjoint subsets of $\mathcal{S}$ that each contain at least $\epsilon n$ vertices of $G$.
By the large-set expansion of $G$, we know that we need to remove at least $\epsilon \alpha n$ edges to disconnect vertices of $A$ and $B$.
Therefore, by Menger's theorem  there are $\epsilon \alpha n $ edge-disjoint paths between $A$ and $B$  in $G$ \cite{MengerZurAK}. 
There are at most $\bar dn/2$ edges in the graph in total. Therefore, at least half of these paths has a length bounded by $l=\frac{\bar d}{\epsilon \alpha}$. Let $P$ be the set of paths between $A$ and $B$ of length at most $l$. The probability that none of these paths appear in $H$ is at most
$(1-\beta^l)^{\epsilon \alpha n/2}$. Hence,
\[\mathbb{P}( \text{no path between $A$ and $B$ in $H$})\leq \exp(-\beta^l\epsilon\alpha n/2).\]

Given $\mathcal{S}$, there are at most $2^{\frac{n}{R}}$ ways to choose disjoint subsets $A$ and $B$.
By a union bound over all possible partitions we find that the probability that such a partition exists is at most
\[2^{\frac{n}{R}}\exp(-\beta^l\epsilon\alpha n/2)
\leq 2^{\frac{n}{R}}\exp(-\beta^{\bar d/\epsilon\alpha}\epsilon\alpha n/2)
,\]
which gives the result for $c=\beta^{\bar d/(\epsilon \alpha)} \epsilon \alpha/2$. 
\end{proof}

\newproof{The proof of Theorem  \ref{thm: size of giant in expander}
follows from the following proposition which generalizes a recent result of Krivelevich, Lubetzky and Sudakov  \cite{Krivelevich20HighGirthExapnder}, and Lemma~\ref{lem: uniform convergence second component} below which is 
a straightforward generalization of a results of Alon, Bejamini and Stacy \cite{alon2004}. }

\begin{prop}\label{prop: lower bound giant}
Let $\{G_n\}_{n\in\mathbb{N}}$ be a sequence of graphs satisfying the assumptions of Theorem \ref{thm: size of giant in expander}, {and let
$p$ be a continuity point of $\zeta$.}
Then for any $\epsilon>0$
\[\mathbb{P}\left(\zeta(p)-\epsilon\leq\frac{|C_1(G_n(p))|}{n}\leq \zeta(p)+\epsilon\right)\rightarrow 1,\]
For the upper bound, neither  the
{continuity assumption at $p$, nor the}
assumption of expansion, nor that of bounded average degrees is needed.
\end{prop}

{Before giving the proof of the proposition, we remark that an analogue of this statement for the case where the limit is a regular tree 
%(and $\zeta$ is continuous) 
was established in 
\cite{Krivelevich20HighGirthExapnder}, using again a sprinkling argument, combined with branching process techniques (which do not apply here). }

\begin{proof}
We begin by proving  
%\eqref{eq: lim1} 
{the upper bound on $|C_1|$.  Note that this part of the proof will hold} for any sequence of
locally convergent graphs (without the 
assumption of expansion {or bounded average degrees}). 
Fix $k\geq 1$. Define $f_k(G_n(p),v)=\mathbbm {1}\big(|B_k(G_n(p),v)|\geq k\big),$ as the indicator that $v$ is in a component of size at least $k$ in $G_n(p)$. Let  $Z_{\geq k}(G_n(p))=\mathbb{E}_{\mathcal{P}_n}[f_k(G_n(p),v)]$  be the fraction of vertices in $G_n(p)$ that are in a component of size at least $k$.  Then by 
Lemma~\ref{lm: local function - percolated graph}
\begin{equation}\label{Zk-conv}
Z_{\geq k}(G_n(p))\overset{\mathbb{P}}{\to}\zeta_k(p),
\end{equation}
where $\zeta_k(p)=\mu_p(|C(o)|\geq k)$.
Note that $\zeta(p)=\lim_{k\rightarrow \infty}\zeta_{k}(p)$. Suppose that $k$ is large enough that $|\zeta_k(p)-\zeta(p)|\leq \epsilon/2$.  
{Then the desired upper bound}
%\eqref{eq: lim1} 
will be proved once we prove the following,
\[\mathbb{P}(\frac{|C_1|}{n}\leq \zeta_k(p)+\epsilon/2)\rightarrow 1.
\] 
We consider two cases: $\zeta_k(p)>0$ and $\zeta_k(p)=0$. If $\zeta_k(p)>0$,  then
$n\zeta_k(p)\geq k$ for large enough $n$. { But then
$\frac{|C_1|}{n}>\zeta_k(p)+\epsilon/2$ implies $|C_1|\geq k$ which in turn implies
$nZ_{\geq  k}\geq |C_1|$.  Therefore,}
\[\mathbb{P}\Big(\frac{|C_1|}{n}>\zeta_k(p)+\epsilon/2\Big)
\leq \mathbb{P}({Z_{\geq k}(G_n(p))}>\zeta_k(p)+\epsilon/2),\]
{for $n$ large enough} which implies
 \[\mathbb{P}(\frac{|C_1|}{n}\geq \zeta_k(p)+\epsilon/2)\rightarrow 0, \,\, \text{ as }n\rightarrow \infty.
\] 
If $\zeta_k(p)=0$ then  $Z_{\geq k}\rightarrow 0$ {in probability}. For the event that $Z_{\geq k}=0$, we have $|C_1|\leq k<\epsilon n/2$ if $n$ is large enough.  Therefore, if $\zeta_k(p)=0$ we get 
\[\mathbb{P}(\frac{|C_1|}{n}\geq \epsilon/2)
\leq \mathbb{P}(Z_{\geq k}>0)
\rightarrow 0, \,\, \text{ as }n\rightarrow \infty.
\] 
The two cases $\zeta_k(p)=0$ and $\zeta_k(p)>0$  give
{the desired upper bound on $|C_1|$.}
Note that this part also implies that if $\zeta(p)=0$ then $\frac{|C_1|}{n}\rightarrow 0$ {in probability}.

{To prove the lower bound, we may assume without loss of generality that $\zeta(p)>0$, i.e.,
we may assume that 
$p>p_c=\inf_p\{\zeta(p)>0\}$. 
Recalling the definition of large-set expanders, choose $\alpha>0$ such that for all
$0<\epsilon<1/2$  $G_n$ is an $(\alpha,\epsilon/8,{\bar d})$-large-set expander {with probability tending to $1$ as $n\to\infty$.}  Choose
$\epsilon$ small enough to make sure that
$\zeta(p)\geq 3\epsilon/2$, and} let $\delta>0$ be such that for $p-\delta<p'<p$, $\zeta(p)-\zeta(p')\leq \epsilon/2$ (implying in particular that $\zeta(p')\geq\epsilon$).
Choose $p'\in((p-\delta\lor p_c),p)$ and $\epsilon'>0$ such that $1-p\geq(1-p')(1-\epsilon')$.

{We will  use Lemma \ref{lem: sprinkling}
to show that, {with high probability,}
after raising $p'$ to $p$, most of the vertices in
large components in $G_n(p')$ will merge into
one giant component in $G_n(p)$.
}

Let $\mathcal{S}_k$ be the set of 
components of size greater than $k$ in $G_n(p')$. 
{Since $\zeta_k(p')\geq \zeta(p')\geq \zeta(p)-\epsilon/2$
we may use \eqref{Zk-conv}
to concluded that  
for all $k$, with high probability the total number of vertices in the sets in $\mathcal{S}_k$ is at least $\epsilon n$.  By our choice of $\epsilon'$, ``sprinkling'' edges with probability $\epsilon'$ on top of percolation
with probability $p'$ will give a percolated graph which is stochastically bounded by $G_n(p)$. Thus, by}
%Now, we can apply 
Lemma \ref{lem: sprinkling},  
{we see} that if $k$ is large enough,  with high probability all but $\epsilon n/8$ vertices that are in a component of size at least $k$ in $G_n(p')$ are in $C_1(G_n(p))$. So, there exist $K_1$ such that for all $k>K_1$ and  large enough $n$,
\[ 
\mathbb{P}\left(
\frac{|C_1(G_n(p))|}{n}\geq {Z_{\geq k}(G_n(p'))}-\frac{\epsilon}{8}\right)\rightarrow 1.\]
{Combined with \eqref{Zk-conv}, this shows that
there exists a constant $K_2$ such that for $k\geq K_2$,}
\[ \mathbb{P}\left(
%\zeta_k(p)+\frac{\epsilon}{4} \geq
\frac{|C_1(G_n(p))|}{n}\geq \zeta_k(p')-\frac{\epsilon}{4}\right)\rightarrow 1.\]
Since $\lim_{k\rightarrow \infty}\zeta_k(p'){=}\zeta(p')$, 
{we get that there exists a constant}
%for some 
$K_3$ such that for and $k\geq K_3$
%\max(K_1,K_2,K_3)$ 
%and large enough $n$
\[ \mathbb{P}\left(
%\zeta(p)+\frac{\epsilon}{2} \geq
\frac{|C_1(G_n(p))|}{n}\geq \zeta(p')-\frac{\epsilon}{2}\right)\rightarrow 1.\]
By the choice of $\delta$ and $p'$, $\zeta(p')\geq \zeta(p)-\epsilon/2$. Hence, we get 
%the result
desired lower bound on $|C_1|$:
\[\mathbb{P}\left(
%\zeta(p)+\epsilon \geq
\frac{|C_1(G_n(p))|}{n}\geq \zeta(p)-\epsilon\right)\rightarrow 1.\]
\end{proof}

%ya{The final step before proving Theorem~\ref{thm: size of giant in expander} is to bound the relative size of the second largest component. For this purpose, we do not need the local convergence of the sequence, and a weaker condition is enough. }
 \newproof{
 Our  next lemma generalizes Theorem 2.1 in \cite{alon2004}, replacing an assumption of bounded degree expanders by the assumption of large set expansion plus 
  a tightness bound on the largest degree, $\Delta_R(G_n,o_n)$,
 in a ball of radius  $R$
 around a random root $o_n$ in $G_n$.  Specifically, we will assume that  \begin{equation}\label{Delta-tightness}
\forall R<\infty\quad
\limsup_{n\to\infty}\mathbb P(\Delta_R(G_n,o_n)\geq \Delta)\to 0\quad\text{as}\quad\Delta\to\infty,
 \end{equation}
where the probability is with respect to a random root $o_n$ and the randomness of $G_n$.}
 
\begin{lemma}\label{lem: uniform convergence second component}
Let $\{G_n\}$ be a sequence 
\newproof{of bounded average degree large-set expanders obeying the tightness condition
\eqref{Delta-tightness}},
%satisfying the assumptions of Theorem \ref{thm: size of giant in expander}, 
let  $0<q<1/2$ and $c>0$ be arbitrary.
Then for any $\epsilon>0$ there exists $N_{\epsilon,q,c}$  such that for all $n>N_{ \epsilon,q,c}$ and all $p\in[q,1-q]$
\[\mathbb{P}(\frac{|C_2|}{n}\geq c )\leq \epsilon\]
{where $\mathbb P$ denotes probabilities with respect to both the randomness of $G_n$ and percolation.}
\end{lemma}
The proof of the lemma closely follows that of \cite{alon2004}, and is given in Appendix~\ref{app:C2}.

\begin{proof}[Proof of Theorem \ref{thm: size of giant in expander}]
The theorem follows immediately from 
Proposition \ref{prop: lower bound giant},
Corollary~\ref{cor: zeta-cont}, and Lemma~\ref{lem: uniform convergence second component} \newproof{and the fact that local convergence in probability implies tightness, which by
Theorem A.16 of \cite{RemcoVol2} implies \eqref{Delta-tightness}.}
\end{proof}

Next, we observe that
by Theorem \ref{thm: size of giant in expander}, the critical threshold is local for expanders with bounded average degree. 
\begin{corollary}\label{cor: critical prob} 
Let $G_n$ be a sequence of expanders with bounded average degree that converges locally in probability to $(G,o)\in \mathcal{G}_*$ with the law $\mu$. Then
 \[p_c(G_n)\overset{\mathbb{P}}{\to} p_c(\mu).\]
\end{corollary}
\begin{proof}
Given any $\epsilon>0$ and $p>p_c(\mu)+\epsilon$, since $\zeta(p)$ \newproof{is continuous}, we can apply Theorem \ref{thm: size of giant in expander} to get that $\frac{|C_1|}{n}>0$ in $G_n(p)$ for large enough $n$, and hence, \[\lim_{n\rightarrow \infty}\mathbb{P}(p_c(G_n)\geq p_c(\mu)+\epsilon)=0,\] for all $\epsilon>0$.

If  $p=(p_c(\mu)-\epsilon\vee 0)$,  we know that $\zeta(p)=0$. 
So, by  Theorem \ref{thm: size of giant in expander}, $\frac{C_1}{n}\overset{\mathbb P}{\to} 0$. 
Therefore, if $p_c(\mu)>0$
 \[\lim_{n\rightarrow \infty}\mathbb{P}(p_c(G_n)\leq p_c(\mu)-\epsilon)=0,\,\, \text{ for all }\epsilon>0.\] 
 Note that if $p_c(\mu)=0$, we already know that $p_c(G_n)\geq p_c(\mu)$ for all $n$. As a result of the above limits,
  \[\lim_{n\rightarrow \infty}\mathbb{P}(|p_c(G_n)- p_c(\mu)|\geq \epsilon)=0,\,\, \text{ for all }\epsilon>0.\] 
\end{proof}
% Note that we {we were able to prove the above statement in spite of the fact that $\zeta$ might not be continuous. }
%In Section \ref{sec: pref attachment}, we {apply the corollary to bond percolation on} preferential attachment graphs, proving that $p_c$ is zero. 

%%%%%%%%%%%%%%%%%%%%%%%%%%%%%%%%%%%%%%%%%%%%%%%%%%%%%%%%%%%%%
%%                  Section 5                       %%

\section{Coupling  Oriented and Unoriented Percolations}\label{sec: coupling}

In this section, we relate oriented and unoriented percolation for general graphs. In particular, we will compare the thresholds $p_c(G_n)$ for the appearance of a giant component in $G_n(p)$ to that of the appearance of a giant SCC in $D_{G_n}(p)$ introduced in Section \ref{sec: defs}, and show that they are asymptotically the same if both exist (Section~\ref{sec:threshold-coupling}).
Next, in Section~\ref{sec: coupling unique giant}, we analyze oriented percolation when  the  giant component in the unoriented case is 
unique, and show that under this assumption, the linear-sized SCC  is unique if it exists.

{Throughout Section~\ref{sec: coupling},}
we {make} no assumptions on graph expansion or the existence of local  limit, and the result %can 
carry over to general graphs.

\subsection{Comparing unoriented and oriented thresholds}\label{sec:threshold-coupling}

The following lemma introduces a coupling between oriented and unoriented percolation. The second part couples non-overlapping fan-outs of two vertices to the undirected components of those vertices.
To state the lemma, we need the following notation: Given a node $u$ in $G$ and a subset $S\subset V(G)$, we define $C_{\setminus S}^+(u)$ as the fan-out of $u$ in the induced {digraph}  on the complement of $S$.  If $u\in S$, then  $C_{\setminus S}^+(u)=\emptyset$.

\begin{lemma} \label{lem: coupling}
Let $G$ be a graph on $n$ nodes, and let $p\in(0,1)$. Then,
\begin{enumerate}
    \item For $v$ in $G$ and $\alpha>0$,
\[\mathbb{P}_{D_G(p)}(|C^+(v)|\geq \alpha n)=\mathbb{P}_{G(p)}(|C(v)|\geq \alpha n)=\mathbb{P}_{D_G(p)}(|C^-(v)|\geq \alpha n).\]
\item Let $u$ and $v$ be vertices in {$G$},
and let $k_1$ and $k_2$ be positive integers,
\begin{align*}
\mathbb{P}_{D_G(p)}\big(|C^+(v)|\geq k_1, |C^+_{ \setminus C^+(v)}(u)|\geq k_2\big)
&= \mathbb{P}_{G(p)}\big(|C(v)|\geq k_1, |C(u)| \geq k_2, C(v)\neq C(u)\big)\\
&{=\mathbb{P}_{D_G(p)}\big(|C^+(v)|\geq k_1, |C^-_{ \setminus C^+(v)}(u)|\geq k_2\big).}
\end{align*}
\end{enumerate}
\end{lemma}
\begin{proof}
\begin{enumerate}
\newproof{    \item  This part is a special case of part 2. To see this, one can add a dummy isolated node $u$ and let $k_2=0$. }
 \item 
We will prove the first equality and during the proof we will point out how it can be extended to prove the second inequality.
For any vertex $w$, define $T(w)$ and $T^+(w)$ as the tree rooted at $w$ obtained by breadth-first  exploration of $C(w)$ and $C^+(w)$, respectively.
Also, define $T^+(u\setminus v)$ as the breadth-first exploration of $C^+_{\setminus C^+(v)}(u)$. Finally, given a tree $T\subset G$ and a root $v$ in $T$, define
the corresponding oriented graph $T^+$  by directing edges away from the root.  

Consider now two arbitrary trees $T_1$ and $T_2$ rooted at $v$ and $u$, respectively.
If they intersect, the probability that $T^+(v)=T_1^+$ and $T^+(u\setminus v)=T_2^+$ 
is zero, and so is the probability that $T(v)=T_1$, $T(u)=T_2$ and $C(v)\neq C(u)$.  It they are disjoint, we will  define a coupling of
$G(p)$ and $D_G(p)$ which shows that
\[\mathbb{P}_{D_G(p)}\big(T^+(v)=T_1^+,T^+(u\setminus v)=T_2^+\big)=\mathbb{P}_{G(p)}\big(T(v)=T_1,T(u)=T_2\big).\]
Define $l_1(w)$ and $l_2(w)$ to be the distance of a vertex $w$ from the root in $T_1$ and $T_2$, respectively (with the root having level $0$ and nodes that are not in the tree having level $\infty$). We express the instances of $D_G(p)$ by choosing, for each edge $\{x,y\}\in E(G)$, two Bernoulli random variables $X_{x,y}$ and $X_{y,x}$, so that $X_{x,y}$ is $1$ if and only if the directed edge from $x$ to $y$ exists in $D_G(p)$ and is $0$ otherwise.  Similarly, let $Y_{x,y}$ be a Bernoulli random variable corresponding to the existence of an undirected edge between $x$ and $y$ in $G(p)$. 

To define the coupling, first consider the case that  $\{x,y\}\in E(G)$ and $l_1(x)<l_1(y)$: if the edge $(x,y)$ does not exists in $T$, let $X_{x,y}=0=Y_{x,y}$.
If $x$ is the successor of $y$ in $T_1$ let $X_{y,x}=1=Y_{y,x}$.  Note in particular that
the events $T^+(v)=T_1^+$ and $T(v)=T_1$ happen only if we have
set $X_{x,y}=Y_{x,y}=0$ whenever $x$ is a vertex in $T_1$ and $y$ is a vertex in $T_2$.

Next we couple the binary random variables for edges $\{x,y\}\in E(G)$ such that $l_2(x)<l_2(y)$
and $x,y\notin V(T_1)$.  Since the event $T^+(u\setminus v)=T_2$ involves only edges with
both endpoints in $V(G)\setminus V(T_1)$, the edges coupled in the second step determine whether
$T^+(u\setminus v)=T_2$ or not.  On the other hand, the event $T(u)=T_2$ \emph{does} involve edges between
the vertices in $T_1$ and $T_2$, namely, it requires that $Y_{x,y}=0$ if $\{x,y\}$ is an edge pointing from
$T_1$ to $T_2$.  But as remarked before, these edges have already been set in our first coupling step,
and have been set in such a way that if $T(v)=T_1$ {then} all these edges are absent in $G(p)$, as required.
Setting finally all remaining edges independently, we obtain a coupling such that
the events {in $D_G(p)$} happen if and only the corresponding events {in $G(p)$} happen, and with the same probability.

This completes the proof of the first {identity} in the lemma.  The second one is essentially the same,
except that in the second step, we orient all edges in $D_G(p)$ in the opposite direction.
\end{enumerate}
\end{proof}

Lemma \ref{lem: coupling} immediately gives the following corollary, which in particular shows that the existence of a giant SCC in $D_G(p)$ implies the existence of a giant component in $G(p)$.
 
 \begin{corollary}\label{cor:pc-lower-bd}
 Given a (possibly random) graph $G$  
 on $n$ nodes, $p\in [0,1]$ and $\alpha>0$, we have that
 \begin{align*}
    \mathbb P_{D_G(p)}&(|SCC_1|\geq \alpha n)\\
     &\leq  \mathbb P_{D_G(p)}(\text{there exists $\geq \alpha n$ vertices $v$ with }|C^+(v)|\geq \alpha n)\\
    & \qquad\qquad \leq\frac 1\alpha   \mathbb P_{G(p)}(|C_1|\geq \alpha n),
 \end{align*}
 where the probability $\mathbb P$ is first over the randomness of $G$ and then oriented/unoriented percolation.
 \end{corollary}
 
 \begin{proof}
We {first prove} the result for a non-random graph $G$; {it can} then %it can 
be generalized to  random graphs $G$ by conditioning on 
%each 
the random graph {instance},
%configuration 
and then taking a weighted average over all possible 
{instances}.
%graphs.

 Let $k=\lceil \alpha n\rceil$. In $D_G(p)$,
 if $|SCC_1|\geq k$, then there are at least
 $k$ vertices $v$ with $|SSC(v)|\geq k$, which implies that there are at least $k$ vertices with fan-out at least $k$, proving the first inequality.   Next define
$Z_{\geq k}$ as the number of vertices such that $|C(v)|\geq k$, and define $Z^+_{\geq k}$ to be the number of vertices such that $|C^+(v)|\geq k$.
Using first Lemma \ref{lem: coupling}, and then the fact that either $Z_{k}=0$ or
$k\leq Z_{\geq k}\leq n$, we  have,
\begin{align*}
   \mathbb P&(\text{there exists $\geq \alpha$ vertices $v$ with }|C^+(v)|\geq \alpha n)
=\mathbb P(Z_{\geq k}^+\geq k)
\leq \frac 1{k}\mathbb E[Z_{\geq k}^+]=\\
&=
\frac 1{k}\mathbb E[Z_{\geq k}]
=\frac 1{k}\mathbb E[Z_{\geq k} 1_{Z_{\geq k}\geq k}]
\leq \frac nk \mathbb P(Z_{\geq k}\geq k)=\frac nk \mathbb P(|C_1|\geq k).
\end{align*}
Since $n/k\leq 1/\alpha$, this proves the corollary.
\end{proof}

The next lemma  gives a bound in the opposite direction.
Recall that the strongly connected component of a vertex $v$,
$SSC(v)$, is the  intersection of $C_n^+(v)$ and $C_n^-(v)$.
Also, recall that $C_i$ and $SCC_i$ denote the $i^\text{th}$ largest component/strongly connected component in $G(p)$
and $D_G(p)$, respectively.

\begin{lemma}\label{lm: expected SCC}
Given a graph $G$ and a constant $p\in[0,1]$,
\[
\frac 1n \mathbb E_{D_G(p)}[|SCC_1|]\geq 
\frac 1{n^2}\sum_i\mathbb E_{D_G(p)}[|SCC_i|^2]
\geq \Big(\frac 1n\mathbb E_{G(p)}[|C_1|]\Big)^4.
\]
\end{lemma}

\begin{proof}
First we note that for all vertices $u,v$ we have that
\begin{equation}\label{eq: fkg}
\mathbb{P}_{D_G(p)}(u\in C^+(v)\text{ and }v\in C^+(u))\geq\mathbb{P}_{D_G(p)}(u\in C^+(v)) \mathbb{P}_{D_G(p)}(v\in C^+(u)) \end{equation}
by the standard FKG inequality.  Indeed, let $D$ be the digraph obtained from $G$ by replacing every edge in $G$ by two oriented edges,
and let $\{0,1\}^D$ be the set of subgraphs of $D$, equipped with the natural partial order (with $D_1\leq D_2$ if each edge in $D_1$ is an edge in $D_2$).  Then the functions 
$\mathbf{1}(u\in C^+(v))$ and $\mathbf{1}(v\in C^+(u))$ are both increasing functions
on $\Omega$, so \eqref{eq: fkg} follows from the Harris inequality \cite{harris_1960}.

For any two vertices $u$ and $v$ define $q_{uv}$ to be the probability that $v\in C^+(u)$ in $D_G(p)$. 
 By the coupling {from the proof of} Lemma \ref{lem: coupling}, $q_{uv}$ is equal to the probability that $v\in C(u)$ in $G(p)$, which is also the same probability that $u\in C(v)$. Therefore, $q_{uv}=q_{vu}$.
Using that 
$
\mathbb{E}_{D_G(p)}|SCC(v)|=\sum_u
 \mathbb{P}_{D_G(p)}(u\in C^+(v)\text{ and }v\in C^+(u))
$, we therefore get 
\begin{align*}
 \frac 1{n^2}\sum_v\mathbb{E}_{D_G(p)}|SCC(v)|&\geq \frac 1{n^2}\sum_{u,v\in V(G)} q_{uv}^2
 \\
 &\geq \Big(\frac{1}{n^2}\sum_{u,v\in V(G)} q_{uv}\Big)^2
 =\Big(\frac{1}{n^2}\sum_v\mathbb{E}_{G(p)}|C(v)| \Big)^2
 \end{align*}
where the second inequality follows from  Cauchy--Schwarz.
As a consequence,
\begin{align*}\sum_i\mathbb{E}_{D_G(p)}&[{|SCC_i|^2}]
= \sum_v\mathbb{E}_{D_G(p)}[{|SCC(v)|}]
 \geq\frac 1{n^2}\Big(\sum_v\mathbb{E}_{G(p)}[|C(v)|]\Big)^2
 \\
 &=\frac 1{n^2}\Big(\sum_i\mathbb{E}_{G(p)}[|C_i|^2]\Big)^2
 \geq \frac 1{n^2}\big(\mathbb{E}_{G(p)}[|C_1|^2]\big)^2
 \geq \frac 1{n^2}\big(\mathbb{E}_{G(p)}[|C_1|]\big)^4.
\end{align*}
To
complete the proof, we note that $\sum_i|SCC_i|^2\leq |SCC_1|\sum_i|SCC_i|=n|SCC_1|$.
\end{proof}

\begin{corollary}\label{cor: SCC threshold general graph} 
Fix a (a possibly random) sequence of graphs  $\{G_n\}_{n\in \mathbb{N}}$. If  
 $p_n$ and $p_n^{SCC}$ are threshold {sequences} for the existence of a giant in $G_n(p)$ and the existence of a giant SCC in
 $D_{G_n}(p)$, respectively, then $|p_n-p_n^{SCC}|\to 0$ as $n\to 0$.
\end{corollary}

\begin{proof}
Fix $\epsilon>0$.
By Corollary~\ref{cor:pc-lower-bd} and the definition of a threshold sequence, we know that there exists a $c=c(\epsilon)>0$ and an $N<\infty$ such that 
\[
\mathbb{P}\Big(|SCC_1|\geq cn \text{ in }D_{G_n}(1\wedge (p_n^{SCC}+\epsilon))\Big)\geq \frac 34
\]
and
\[
\mathbb{P}\Big(|SCC_1|\geq cn \text{ in }D_{G_n}(0\vee (p_n-\epsilon))\Big)
\leq \frac 1c\mathbb{P}\Big(|C_1|\geq cn \text{ in }
{G_n}(0\vee (p_n-\epsilon))\Big)
\leq \frac 1{3}
\]
for all $n\geq N$, where the probabilities are first over the possible randomness of $G_n$ and then  percolation.
Since the size of $SCC_1$ is increasing in $p$, this immediately implies that 
$$
1\wedge (p_n^{SCC}+\epsilon)\geq 0\vee (p_n-\epsilon)
$$ for all $n\geq N$. 

To prove a matching bound in the other direction, let $c>0$
and $\tilde N$ be such that for $n\geq \tilde N$
\[
\mathbb{P}\Big(|C_1|\geq cn \text{ in }{G_n}(1\wedge (p_n+\epsilon))\Big)\geq \frac 34.
\]
Then $\mathbb E|C_1|\geq \frac{3cn}4$, where the expectation is over the possible randomness of $G_n$ and  percolation.
So by using Lemma~\ref{lm: expected SCC} for all possible instances of $G$ and Jensen's inequality,
$$
\frac 1n {\mathbb{E}_G}\mathbb E_{D_G(1\wedge (p_n+\epsilon))}[|SCC_1|]\geq 2C,
$$
for some $C>0$ that depends on $c$.  Since $\mathbb E[|SCC_1|]\leq Cn+n\mathbb P(|SCC_1|\geq Cn)$, we conclude that
$$
\mathbb P_{D_G(1\wedge (p_n+\epsilon))}\big(|SCC_1|\geq Cn\big)\geq C.
$$
Using this fact, we now can proceed as in the derivation of the lower bound on $p_n^{SCC}-p_n$ to show that for $n$ large enough,
$$
1\wedge (p_n+\epsilon)\geq 0\vee (p_n^{SCC}-\epsilon).
$$
Since $\epsilon>0$ was arbitrary, this bound together with the matching bound above implies that $|p_n-p_n^{SCC}|\to 0$ as $n\to\infty$.
\end{proof}

Recall the definition of $\zeta^{+-}(p)$ from \eqref{eq: zeta+-}. Similar to proof of Lemma~\ref{lm: expected SCC} we can give bounds on $\zeta^{+-}(p)$.
\begin{lemma}\label{lm: zeta+- > zeta^2}%\todo{Should we remove this Lemma?  Not sure - its proof is pretty and the lemma holds without any continuity assumptions.  Why do we need the limit $p'\uparrow p$ - I don't see it used in the proof.}
Let $\mu$ be a probability distribution on $\mathcal{G_*}$,
{and let}
%Given 
$p\in[0,1]$.
%, let $\zeta^{+-}(p)$ be defined as in  \eqref{eq: zeta+-}, and let $\zeta_-(p)=\lim_{p'\uparrow p} \zeta(p')$.
Then
\[\zeta(p)\geq \zeta^{+-}(p)\geq
%\zeta^2_-(p
{\zeta^2(p).}
\]
{As a consequence, $p_c(\mu)=p_c^{+-}(\mu).$}
\end{lemma}
\begin{proof}
To prove the lower bound,
consider a graph $(G,o)$ drawn from the distribution $\mu$. 
Since $\mathbf{1}(|C^+(o)|=\infty)$ and $\mathbf{1}(|C^-(o)|=\infty)$ are increasing functions of edges, by the FKG inequality 
\[\mathbb P_{D_{G}(p)}\Big(|C^+(o)|=\infty,|C^-(o)|=\infty\Big)\geq \mathbb P_{D_{G}(p)}\Big(|C^+(o)|=\infty\Big)\mathbb P_{D_{G}(p)}\Big(|C^-(o)|=\infty\Big).\]
Similar to the coupling of Lemma~\ref{lem: coupling} on  the infinite graph $G$, we get 
\[\mathbb P_{D_{G}(p)}\Big(|C^+(o)|=\infty\Big)=\mathbb P_{{G}(p)}\Big(|C(o)|=\infty\Big)=\mathbb P_{D_{G}(p)}\Big(|C^-(o)|=\infty\Big).\]
As a result,
\begin{align*}
    \mathbb \mu\Bigg(\mathbb P_{D_{G}(p)}\Big(|C^+(o)|=\infty,|C^-(o)|=\infty\Big)\Bigg)&\geq \mu\Bigg(\mathbb P^2_{G(p)}\Big(|C(o)|=\infty\Big)\Bigg)\\
    &\geq \mu\Bigg(\mathbb P_{G(p)}\Big(|C(o)|=\infty\Big)\Bigg)^2,
\end{align*}
where the second inequality is by Cauchy--Schwarz.

The upper bound immediately follows by applying the coupling in Lemma \ref{lem: coupling} to infinite graphs $G$. In fact,
\begin{align*}
    \zeta^{+-}(p)&=\mu\Bigg(\mathbb P_{D_{G}(p)}\Big(|C^+(o)|=\infty,|C^-(o)|=\infty\Big)\Bigg)
    \\
&\leq \mu\Bigg(\mathbb P_{D_{G}(p)}\Big(|C^+(o)|=\infty\Big)\Bigg)
=\mu\Bigg(\mathbb P_{{G}(p)}\Big(|C(o)|=\infty\Big)\Bigg)=\zeta(p).
\end{align*}
{The statement about $p_c$ follows trivially.}
\end{proof}

\subsection{Graphs with a Unique Giant Component}\label{sec: coupling unique giant}
We proceed by considering graphs with a unique giant for unoriented percolation and we analyze the implications for oriented percolation. 
We will use the following definition.
\begin{definition}
Fix $p\in [0,1]$ and $\epsilon>0$.  A  sequence $\{G_n\}_{n\in\mathbb N}$ of
{(possibly random)} graphs is called a sequence  of graphs with {an} $\epsilon$-unique giant component
 if the probability that $|C_2|\geq \epsilon n$ in $G_n(p)$ goes to $0$ as $n\to\infty$. We say that $G_n$ has a \textit{uniformly} $\epsilon$-unique giant component 
 \newproof{in an interval $I\subset [0,1]$}
 if %for any $q\in (0,1/2)$,
 \[\lim_{n\rightarrow \infty}
 %\sup_{p\in[q,1-q]}
 \sup_{p\in I}
 \mathbb P_{G_n(p)}(|C_2|\geq \epsilon n)= 0,\]
{where the probability $P_{G_n(p)}$ is over both the randomness of $G_n$ and the randomness of percolation.}
\end{definition}
 Note that by Lemma~\ref{lem: uniform convergence second component}, a sequence $G_n$ satisfying the  assumptions of Theorem \ref{thm: size of giant in expander} has a 
uniformly $\epsilon$-unique giant component \newproof{in $I$ for all $\epsilon$ and all closed intervals $I\subset(0,1)$.}
As a first consequence of this fact and Lemma~\ref{lem: coupling}, we show that if in $D_G(p)$ two vertices have large non-overlapping fan-ins/fan-outs, then there must be two giant components {in $G(p)$}. As a result, we prove that if a giant SCC exists it must be unique and almost all of the vertices with a large fan-out must reach to the giant SCC {before} exploring many nodes outside the SCC.
The next corollary shows that  on graphs with {an} $\epsilon${-unique giant}, all but $\epsilon n$ of the large fan-out must have equal size.

\begin{corollary}\label{cor: random nodes out-component}
{Fix $p\in [0,1]$ and} let  $\{G_n\}_{n\in\mathbb{N}}$ be a  sequence of {(possibly random)} graphs   with $\epsilon$-unique giant component.  Then for any two fixed vertices $u$ and $v$,
\[\mathbb{P}_{D_{G_n}(p)}(|C^+(v)|\geq \epsilon n, |C_{\setminus C^+(v)}^+(u)|\geq \epsilon n)\rightarrow 0,\]
where the convergence is uniform in $u$ and $v$, and the randomness is  over  oriented percolation, {and the possible randomness of $G_n$}. 
If the sequence $\{G_n\}$ has {a} uniformly $\epsilon$-unique giant component {in the interval $I$} then the convergence is uniform in {$p\in I$}. 
\end{corollary}
\begin{proof}\newproof{
By applying Lemma~\ref{lem: coupling} Part 2, we see that the statement is equivalent to
$$
\mathbb{P}_{G_n(p)}\left(|C(v)|\geq \epsilon n, |C(u)|\geq \epsilon n
\text{ and } C(u)\neq C(v)\right)\to 0,
$$
where the probability {now} %only 
goes over the randomness in $G_n(p)$ (including the possible randomness of $G_n$). But if
$|C(v)|\geq \epsilon n$, $|C(u)|\geq \epsilon n$
and $C(u)\neq C(v)$ there  exist at least two clusters of size $\geq \epsilon n$, implying
that $|C_2|\geq\epsilon n$. Therefore the left hand side is bounded by $\mathbb{P}_{G_n(p)}\left(|C_2|\geq\epsilon n\right)$, which goes to zero by $\epsilon$-uniqueness of the giant. Note that the convergence is uniform in $p$ if the sequence has a uniformly $\epsilon$-unique giant {in I}.}
\end{proof}

{The corollary clearly implies that}
the result  holds for uniform random choice of $u$ and/or $v$. 
Furthermore, one can also bound the size of the second largest SCC in $D_G(p)$.
\begin{lemma}\label{lm: unique SCC}
Let $p\in[0,1]$ and let $\{G_n\}_{n\in\mathbb{N}}$ {be} a (possibly random) sequence of graphs with {an}  $\epsilon$-unique giant component.
Then with  probability \newproof{tending to $1$} the second largest SCC in $D_{G_n}(p)$ {contains less than}  $\epsilon n$ vertices. If the sequence $\{G_n\}_{n\in\mathbb N}$ has {a} uniformly $\epsilon$-unique giant component {in $I$} then the convergence is uniform %in 
\newproof{for all $p\in I$}. 
\end{lemma}
\begin{proof}
Assume to the contrary that for some $\delta$ with probability at least $\delta$, $D_{G_n}(p)$ has two SCCs larger than $\epsilon n$. In an instance of $D_{G_n}(p)$, let $A$ and $B$ be two disjoint SCCs. Then without loss of generality assume there is no directed path from any vertex of $A$ to any vertex of $B$. 
So, if we pick a random pair of nodes $(u,v)$ of this instance,   with probability at least $\epsilon^2$ we have that  $|C^+(v)|\geq \epsilon n$ and $|C_{\setminus C^+(v)}^+(u)|\geq \epsilon n$. 
Now, by considering all instances of $D_{G_n}(p)$, we have that 
\[\mathbb{P}_{u,v\in [n],D_{G_n}(p)}(|C^+(v)|\geq \epsilon n, |C_{\setminus C^+(v)}^+(u)|\geq \epsilon n)\geq \epsilon^2 \delta.\]
Therefore, by Corollary \ref{cor: random nodes out-component} we get a contradiction. 
\newproof{As in  Corollary~\ref{cor: random nodes out-component}, the convergence is uniform in $I$ if the giant is uniformly $\epsilon$-unique in $I$.}
\end{proof}

Now that we know if the giant SCC exists, it is unique, 
%{and}
{with}  a very similar argument we can prove that 
(i) all {but $o(n)$} nodes with a large fan-out reach into the giant SCC (if {it} exists) and (ii)
there are only {$o(n)$} nodes in their fan-out before it reaches the set $SCC^+_1$.

\begin{lemma}\label{lm: L+ to SCC} 
Given $\epsilon>0$ and $p\in[0,1]$, let $\{G_n\}_{n\in\mathbb{N}}$  be a sequence of (possibly random) graphs with {an}  $\epsilon$-unique giant component. 
Let $SCC_1$ be the largest SCC of $D_{G_n}(p)$ and let 
$O_{\epsilon}$ be the set of vertices $u$ {such}
that $|C_{\setminus SCC_1^+}^+(u)|\geq \epsilon n$, i.e., the set of vertices that have a large fan-out
{before reaching}
%but have $\epsilon n$ nodes in their fan-out outside of 
$SCC_1^+$. 
Then
\[{\lim_{n\to\infty}}\mathbb{P}_{D_{G_n}(p)}(|SCC_1|\geq \epsilon n, |O_{\epsilon}|\geq \epsilon' n)=0\quad\text{ for all }\epsilon'>0,\]
{and
\[{\lim_{n\to\infty}}\mathbb{P}_{D_{G_n}(p)}(|SCC_1|\geq \epsilon n, |\tilde O_{\epsilon}|\geq \epsilon' n)=0\quad\text{ for all }\epsilon'>0,\]
where $\tilde O_\epsilon$ is the set of nodes $u$ such that $|C^+(u)|\geq \epsilon n$
and $C^+(u)\cap SCC_1=\emptyset$.}
\end{lemma}
\begin{proof}
Assume to the contrary that there exists $\epsilon'>0$  and $\delta>0$ that 
\[\mathbb{P}_{D_{G_n}(p)}(|SCC_1|\geq \epsilon n, |O_\epsilon|\geq \epsilon' n)\geq \delta\]
{for infinitely many $n$.}
Then given an instance of $D_{G_n}(p)$ 
{such that $|SCC_1|\geq \epsilon n$ and $ |O_\epsilon|\geq \epsilon' n$,}
if we pick two random nodes $u$ and $v$, with probability at least $\epsilon\epsilon'$, $u\in O_\epsilon$ and  $v\in SCC_1$, which in turn implies that $|SCC(v)|\geq \epsilon n$ and $|C_{\setminus SCC(v)^+}^+(u)|\geq \epsilon n$.  But $SSC(v)^+=C^+(v)$, and $|SCC(v)|\geq \epsilon n$ implies
$|C^+(v)|\geq \epsilon n$,
so we have that with probability at least $\epsilon\epsilon'$,
$|C^+(v)|\geq \epsilon n$ and $|C_{\setminus C^+(v)}^+(u)|\geq \epsilon n$, which contradicts 
Corollary~\ref{cor: random nodes out-component}.

The last statement follows by similar arguments and the observation that
$C^+(u)\cap SCC_1(v)=\emptyset$ implies that $C^-(v)\cap C^+(u)=\emptyset$, which in turn gives that
$C^{-}_{\setminus C^+(u)}(v)=C^-(v)$.
\end{proof}

Together with Lemma~\ref{lm: expected SCC} 
Lemma~\ref{lm: unique SCC} also allows us to bound the expected size of the square of the largest SCC from below.

\begin{lemma}\label{lem:E0fSSC1^2}
Given $\epsilon>0$ and $p\in[0,1]$, 
let $\{G_n\}_{n\in\mathbb{N}}$  be a sequence of (possibly random) graphs with
 $\epsilon$-unique giant component.  Then for all $\epsilon'\in (0,\epsilon)$
there exists $N$ such that for all $n\geq N$
\[
\newproof{\mathbb{E}_{D_{G_n}(p)}\Big(\frac {|SCC_1|^2}{n^2}\Big)\geq 
\Big(\mathbb{E}_{{G_n}(p)}\Big[\frac{|C_1|}n\Big]\Big)^4-\epsilon'.}
\]
%sufficiently large $n$.
\end{lemma}
\begin{proof}
Fix an instance $G_n$.
If $G_n$ has an $\epsilon$ unique giant, then by
Lemma~\ref{lm: unique SCC}, $|SCC_2|\leq\epsilon n$
with  probability {tending to $1$}, and therefore,
$$
\sum_{i\geq 2}\mathbb E_{D_{G_n}(p)}[|SCC_i|^2]\leq 
n\mathbb E_{D_{G_n}(p)}[|SCC_2|]\leq n^2\epsilon+o(n^2).
$$
Combined with Lemma~\ref{lm: expected SCC} and Jensen's inequality, this implies the statement of the lemma.
\end{proof}

%%%%%%%%%%%%%%%%%%%%%%%%%%%%%%%%%%%%%%%%%%%%%%%%%%%%%%%%%%%%%
%%                 Section 5                      %%

\section{From Unoriented  to Oriented Percolation }\label{sec: oriented expander}
This section analyzes the structure of the oriented percolation. The main goal is to prove Theorem \ref{thm: main structure SI}. For that purpose in Lemma \ref{lm: variance}, we show that in the supercritical case a linear strongly connected component exists with high probability. 
But first,  we give an upper bound on the size of the largest SCC \newproof{for any sequence of graphs with a local limit in probability.}
\begin{lemma}\label{lem: upper bound SCC1}
Let $\{G_n\}_{n\in\mathbb{N}}$ be a {(possibly random)}
sequence of graphs that converges locally in probability to $(G{,o)}\sim\mu$. Recall the definition of $\zeta^{+-}(p)$ in \eqref{eq: zeta+-}. Then for any $\epsilon>0$ and $p\in[0,1]$,
\[\mathbb P_{D_{G_n}(p)}\big(\frac{|SCC_1|}{n}\geq \zeta^{+-}(p)+\epsilon\big)\rightarrow 0,\]
\end{lemma}
\begin{proof}
The proof is similar to the unoriented case in Proposition~\ref{prop: lower bound giant}.
For $k\geq 1$ and a vertex $v\in G_n$, define \[f_k^{+-}(D_{G_n}(p),v)=\mathbf{1}\big(|C^+(v)\cap B_k(D_{G_n}(p),v)|\geq k,  |C^-(v)\cap B_k(D_{G_n}(p),v)|{\geq k}\big),\]
as the indicator that $v$ has fan-out and fan-in larger than $k$.
Define the fraction of vertices with fan-in and fan-out larger than $k$ as $Z^{+-}_{\geq k}(D_{G_n}(p))=\mathbb E_{\mathcal P_n}[f_k^{+-}(D_{G_n}(p),v)]$. It is easy to check that Lemma~\ref{lm: local function - percolated graph} also holds for the percolation on digraphs. Therefore,
\[Z^{+-}_{\geq k}\overset{\mathbb P}{\to} \zeta^{+-}_k(p),\]
where $\zeta^{+-}_k(p)=\mu\Big(\mathbb P_{D_{G}(p)}(|C^+(o)|\geq k, |C^-(o)|\geq k)\Big)$. 

Note that if $|SCC_1|\geq k$ then $Z^{+-}_{\geq k}\geq |SCC_1|$, and if $Z^{+-}_{\geq k}\leq k$ then $|SCC_1|\leq k$. Then by considering two cases $\zeta_k^{+-}(p)=0$ and $\zeta_k^{+-}(p)>0$, with a similar argument as in Proposition~\ref{prop: lower bound giant} we get that
\[\mathbb P_{D_{G_n}(p)}\big(\frac{|SCC_1|}{n}\geq \zeta_k^{+-}(p)+\epsilon/2\big)\rightarrow 0.\]
Since $\zeta^{+-}(p)=\lim_{k\rightarrow \infty}\zeta^{+-}_k(p)$, one can find $K$ such that for $k>K$ we have $|\zeta_k(p)-\zeta(p)|\leq \epsilon/2$. Therefore,
\[\mathbb P_{D_{G_n}(p)}\big(\frac{|SCC_1|}{n}\geq \zeta^{+-}(p)+\epsilon\big)\rightarrow 0.\]
\end{proof}
\begin{remark}\label{rem:SCC-tree}
\newproof{When the limit $\mu$ is a non-random tree, $|C^+(o)|$ and $|C^-(o)|$ become independent. Thus  $\zeta_{+-}(p)=\zeta^2(p)$, which matches the lower bound from Theorem~\ref{thm: main structure SI}.  In other words, under the assumptions of Theorem~\ref{thm: size of giant in expander}, and the additional assumption that the limit $(G,o)$ is a non-random tree, 
$$\frac 1n|SCC_1|\overset{\mathbb P}{\to}\zeta_{+-}(p)=\zeta^2(p).$$
A simple example is a sequence of $d$-regular expanders of large girth.  In general, the asymptotic size of $\frac 1n|SCC_1|$ will not be given by $\zeta^2(p)$, even if the limit is a random tree.}
\end{remark}

The next lemma gives tail bounds on the number of nodes with a large fan-in/fan-out.

\begin{lemma}\label{lem: structure lemma}
Let $L^+_c$ ($L^-_c$) be the set of vertices with fan-out (fan-in) larger than $cn$.  Fix $p\in (0,1]$ and \newproof{an interval $I\subset [0,1]$ containing $p$}. Assume that $\frac{|C_1|}{n}\overset{\mathbb P}{\to}\zeta(p)$ and that 
for all $\epsilon>0$,
the sequence $\{G_n\}_{n\in\mathbb N}$ has {a} uniformly $\epsilon$-unique giant component in {$I$}.
 Then
\begin{enumerate}
\item \label{lm5: part1}
For all
$\epsilon>0$, $\mathbb
    P_{D_{G_n}(p)}(|L^+_{\zeta(p)+\epsilon}|\geq \epsilon n)\rightarrow 0$.
\item \label{lm5: part2} 
  % \newproof{ Let $\delta>0$ be as in the {conditions} of Theorem \ref{thm: main structure SI}.}
    %Given $c,\alpha\in(0,1]$, 
 \newproof{For all $a,c\in (0,1]$ there exists $\epsilon_0>0$ such that for all $\delta>0$ and   $0<\epsilon<\epsilon_0$} 
    there exists $N$ such that for all $n>N$, 
%\todo{if we ever need it, we can choose $\epsilon_0=\epsilon_0(c)$, and $N=N(\epsilon,\delta,\alpha)$}
    \newproof{and all $p'\in I$}, 
    \[\mathbb P_{D_{G_n}({p'})}\big(|L^+_c|\geq \alpha n\big)-{\delta}\leq \mathbb P_{D_{G_n}({p'})}(|L^-_{\alpha-\epsilon}|\geq (c-\epsilon) n).\]
    
\item \label{lm5: part3} 
    If $\zeta(p)>0$,  and  $\alpha\in (0,\zeta(p))$, then
    $$\frac{|L^+_\alpha|}{n}\rightarrow \zeta(p),
    $$
    in {expectation and in} probability.
    {If $\zeta(p)=0$, then for any $\alpha>0$, $\frac{|L^+_\alpha|}{n}\rightarrow 0
    $  in expectation, and hence in probability.}
\end{enumerate}
\end{lemma}

\begin{proof}

\begin{enumerate}
    \item Assume to the contrary that for infinitely many $n$, \[\mathbb
    P_{D_{G_n}(p)}(|L^+_{\zeta(p)+\epsilon}|\geq \epsilon n)\geq \delta.\]
    Then {$\mathbb E|L^+_{\zeta(p)+\epsilon}|\geq \delta\epsilon n$, implying that} there exists a node $v$ such that
    \[\mathbb P_{D_{G_n}(p)}(|C^+(v)|\geq \zeta(p)n+\epsilon n)\geq \delta \epsilon.\]
   On the other hand, by Lemma \ref{lem: coupling}, part 1 and \newproof{the convergence of the giant in unoriented percolation $\frac{|C_1|}{n}\overset{\mathbb P}{\to}\zeta(p)$,}
 % the upper bound in Proposition~\ref{prop: lower bound giant},
     \[\mathbb P_{D_{G_n}(p)}(|C^+(v)|\geq \zeta(p)n+\epsilon n)=\mathbb P_{{G_n}(p)}(|C(v)|\geq \zeta(p)n+\epsilon n)\rightarrow 0,\]
     a contradiction.
     
      \item 
         \newproof{We prove {that with probability at least $1-\delta$, the following statement holds: for all
     digraphs $D_{G_n}(p')$  with $|L_c^+|\geq\alpha n$
     %. We claim that 
     there are at least $(c-\epsilon)n$ nodes in $D_{G_n}(p')$ such that their fan-in is larger than $(\alpha-\epsilon)n$, i.e., $|L_{\alpha-\epsilon}^-|\geq(c-\epsilon)n$. We prove this in two steps, 1)  We show that {with probability at least $\delta$} there is a single node $u\in L^+_c$ such that its fan-out $C^+(u)$ covers (almost) all of the fan-outs of the nodes in $L^+_c$. 2) 
     this, we prove
    that the fan-in of most of the nodes $w\in C^+(u)$ is large, and in fact, $|C^-(w)|\geq |L^+_c|-\epsilon n$.

     To formally prove the first step, we need the following definition. Call a pair $(x,y)$ bad if $x\in L^+_c$  and $|C^+_{\setminus C^+(x)}(y)|\geq\epsilon^2 n/2$. The choice of $\epsilon^2/2$ may seem arbitrary at first, but it will be useful later on. Next, we will bound the number of bad-pairs.}}
     %Note that  uniform convergence of $\frac{|C_2|}{n}$ to $0$  implies that       $G_n$ is a sequence of graphs with a $\epsilon'$-unique giant for all \newproof{$p'\in(p-q,p]$} and all $\epsilon'>0$.
Let $\epsilon'=\epsilon^2/2$ and choose $\epsilon$ small enough so that $\epsilon'\leq c$.
By   Corollary \ref{cor: random nodes out-component}, \newproof{and the assumption that $G_n$ has a uniform $\epsilon'$-unique giant in 
${I}$}, we conclude that there exists $N=N(\epsilon,\delta)$  such that for all $n\geq N$ \newproof{and all $p'\in{I}$},
     $$
\frac 1{n^2}\sum_{u,v\in V(G_n)}\mathbb{P}_{D_{G_n}(p')}(|C^+({u})|\geq 
cn,|C_{\setminus C^+({u})}^+({v})|\geq \epsilon^2n/2 )\leq\delta^2.
$$
Hence, the expected number of bad pairs is at most $\delta^2n^2$, and
by Markov inequality with  probability at least $1-\delta$, the number of bad pairs is less than $\delta n^2$.
{To complete the first step, assume without loss of generality that}
 $\delta<\epsilon^2 \alpha/2$. 
%  \todo{ $\delta<\epsilon^2 \alpha/2$.}
 Given an instance of $D_{G_n}(p')$ that has at most $\delta n^2$ bad pairs and $|L^+_c|\geq \alpha n$, there   exists a vertex $u\in L^+_c$  that appears in at most $\epsilon^2 n/2$ bad pairs; \newproof{as a consequence, for at least $n(1-\epsilon^2/2)$ nodes $y$ we have $|C^+(y)\setminus C^+(u)|=|C^+_{\setminus C^+(u)}(y)|<\epsilon^2 n/2$, completing the first step.}

 \newproof{ Now, we proceed with the second step: We claim that there are $(c-\epsilon) n$ nodes in $C^+(u)$ {such that} their fan-in contains at least $(\alpha-\epsilon) n$ nodes of $L^+_c$. Let $X_u$ be the set of nodes in $C^+(u)$ that have less than $|L^+_c|-\epsilon n $ nodes in their fan-in. We will use the fact that $u$ appears in at most $\epsilon^2n/2$ bad pairs to prove $|X_u|\leq \epsilon n$. 
  Before proceeding with its proof,  note that if $|X_u|\leq \epsilon n$,  then there are $|C^+(u)|-\epsilon n$ nodes with fan-in of size {at least} ${|L^+_c|-\epsilon n \geq}(\alpha-\epsilon)n$. Thus, $|L^-_{\alpha-\epsilon}|\geq (c-\epsilon)n$.  As a result, for any $\delta<(\alpha\epsilon)^2/2$ 
%   \todo{$\delta<(\alpha\epsilon)^2/2$ }
  there exists $N$ such that for $n\geq N$
  and all $p'\in I$,}
 \begin{align*}
 \mathbb P_{D_{G_n}(p')}&(|L^+_c|\geq \alpha n)-\delta \\
 &\leq\mathbb P_{D_{G_n}({p'})}(|L^+_c|\geq \alpha n,\text{ and $\exists$ at most } \delta n^2 \text{ bad pairs})\\
 &\leq \mathbb P_{D_{G_n}({p'})}(|L^-_{\alpha-\epsilon}|\geq (c-\epsilon) n).
 \end{align*}
Thus the proof of part 2 follows once we prove $|X_u|\leq\epsilon n$.
  
Construct a bipartite graph
$B$ on $C^+(u)\times L_c^+$ with an edge between $w\in C^+(u)$ and $v\in L_c^+$ whenever $w\in C^+(v)$.
  To bound $|X_u|$, we find a lower bound and an upper bound for the number of edges of $B$. First, by definition of $X_u$,  the number of edges coming out of the side of $C^+(u)$ in $B$ are at most $|X_u|(|L^+_c|-\epsilon n) +(|C^+(u)|-|X_u|)|L^+_c|$. Now, to find a lower bound on the number of edges, note that $u$ appeared in at most $\epsilon^2n/2$ bad pairs  $(u,v)$, and if $(u,v)$ is not bad, then
 $|C^+(u)\cap C^+(v)|\geq |C^+(u)|-\epsilon^2n/2$. So,  there are at least $|L_c^+|-\epsilon^2n/2$ nodes of $L_c^+$ that have all but $\epsilon^2n/2$ nodes of $C^+(u)$ in their fan-out. Combining these two bounds, we get
   \[|X_u|(|L^+_c|-\epsilon n) +(|C^+(u)|-|X_u|)|L^+_c|\geq (|C^+(u)|-
\frac{\epsilon^2}{2} n)(|L^+_c| -\frac{\epsilon^2}{2} n).\]
  As a result, $|X_u|\leq \epsilon n$.

     \item Using Lemma \ref{lem: coupling} part 1  and \newproof{the convergence of $\frac{|C_1|}{n}\to \zeta(p)$}, one can compute the first moment,
     \begin{align*}
   \mathbb E [\frac{|L^+_\alpha|}{n}]&=\mathbb E[\frac{1}{n}\sum_{v\in[n]}\mathbf 1_{|C^+(v)|\geq \alpha n}] =\frac{1}{n}\sum_{v\in[n]} \mathbb P_{D_{G_n(p)}}( {|C^+(v)|\geq \alpha n})\\
   &=\frac{1}{n}\sum_{v\in[n]} \mathbb P_{{G_n(p)}} ( {|C(v)|\geq \alpha n})
   =\zeta(p)+o(1),
   \end{align*}
 proving convergence in expectation.   { Note that the same argument also gives that $ \mathbb E [\frac{|L^+_\alpha|}{n}]\to 0$ for all $\alpha>0$ {if} $\zeta(p)=0$.}
  Next, given an arbitrary small  $\epsilon>0$,  assume that there exists some $\delta>0$,
    \[\mathbb P_{D_{G_n}(p)}(|L^+_\alpha|\geq \zeta(p) n+2\epsilon n)\geq \delta.\]
    Then by Part \ref{lm5: part2} and the symmetry of changing the directions of all edges
%    \[\mathbb P_{D_{G_n}(p)}(|L^+_c|\geq \alpha n)-\delta\leq \mathbb P_{D_{G_n}(p)}(|L^-_{\alpha-\epsilon}|\geq (c-\epsilon) n).\]
            we have for large enough $n$
     \[{P_{D_{G_n}(p)}\Big(|L^+_{\zeta(p)+\epsilon}|\geq (\alpha-\epsilon) n\Big)=}
     P_{D_{G_n}(p)}\Big(|L^-_{\zeta(p)+\epsilon}|\geq (\alpha-\epsilon) n\Big)\geq \frac{\delta}{2}.\]
  This is a contradiction with Part \ref{lm5: part1}. As a result, for $\epsilon >0$ small enough
%   \todo{once you make $\epsilon$ smaller, you now have to make $\delta$ smaller.  That seems to be a problem.  Y: Isn't $\delta$ in part 2 different from this? Here' we're arguing at $p$ and not for an interval.}
        \[\operatorname{var}_{D_{G_n}(p)}\Big(\frac{|L^+_\alpha|}{n}\Big)=\mathbb E [(\frac{|L^+_\alpha|}{n})^2]-\zeta(p)^2{+o(1)}\leq (\epsilon+ \zeta(p))^2-\zeta(p)^2{+o(1)}.\]
        Since $\epsilon$ was arbitrary, we get that the variance goes to $0$, which proves
        the result.
\end{enumerate}     
   \end{proof}
%\ya{We would like to point out that continuity of $\zeta$ is not used in the previous proof, and $\epsilon$-uniqueness in an interval, along with convergent of the largest component is enough for the statement to hold. The continuity of $\zeta$ will be useful next, when we want to bound the variance of the relative size of $SCC_1$.}

In  Lemma \ref{lem:E0fSSC1^2} we saw that in the supercritical case $\mathbb{E}|SCC_1|^2\geq \alpha n^2$ for some $\alpha >0$. To prove $SCC_1$ is linear-sized with high probability, i.e., to prove that
$p_c(\mu)$ is a threshold for the existence of a giant SCC, we will want to show
that $\operatorname{var}(\frac{|SCC_1|}{n})\rightarrow 0$.  We will do this by invoking Lemma~\ref{lm: falik} from Section \ref{sec: Russo}, 
a bound on
how much adding an edge $e$ to $D_{G_n}(p)$ will change the size of 
$SCC_1$, Lemma~\ref{lm: bounded influence} below, and Russo's formula in {Equation} \eqref{russo-2} for the expectation of this influence.
Recall the definition of $\Delta_e f$   for a Boolean function $f$ from Section \ref{sec: Russo}. Russo's formula then immediately gives the following lemma.

 \begin{lemma}\label{lm: derivative}
 For any graph $G$, and $p\in(0,1)$
 \[\sum_e\mathbb{E}_{S\sim D_G(p)}|\Delta_e SCC_1(S)|= {p(1-p)} \frac{d}{dp}\mathbb{E}_{S\sim D_G(p)}|SCC_1(S)|,\]
 where $SCC_1$ is equal to the size of the largest SCC with edges in $S$ and the sum goes over all oriented edges in $E(G)$.
 \end{lemma}
 \begin{proof}
 This follows from \eqref{russo-2}.
 \end{proof}

  The next result bounds the influence of an edge to later bound the variance of $|SCC_1|$.

 \begin{lemma}\label{lm: bounded influence}
  \newproof{Let $I=[q,p]$ with $0\leq q\leq p\leq 1$, and assume that in $G_n(q)$,
  $\frac{|C_1|}{n}\overset{\mathbb P}{\to}\zeta(q)>0$.  Furthermore, assume that for all $\epsilon>0$, 
$\{G_n\}_{n\in\mathbb N}$ has {a} uniformly $\epsilon$-unique giant component in $I$.}
 Given    $\epsilon>0$ there then exists $N<\infty$ such that for $n\geq N$,
 \[\sup_{p'\in I}\mathbb E_{D_{G_n}(p')}|\Delta_eSCC_1|\leq \epsilon n\]
 for all directed edges $e$ in $\in E(G_n)$.

 \end{lemma}

 \begin{proof}
{Consider a digraph $D_{G_n}(p')$ that if one adds the directed edge
$e=(y,x)$ to  $D_{G_n}(p')$, then the size of $SCC_1$ increases by $\epsilon n$. Let $H_e$ be the corresponding event.}
% Define $H_e$ to be the event that if one adds the directed edge $e=(y,x)$ to  $D_{G_n}(p)$, then the size of $SCC_1$ changes by $\epsilon n$. 
We will prove that $\sup_{p'\in(p-q,p]}\mathbb P(H_e)\leq \epsilon$. 
% Consider a digraph $D_{G_n}(p)$ that contributes to $H_e$ i.e., $e$ is not in $D_{G_n}(p)$ and adding the edge $e$ increases the size of $SSC_1$ by at least $\epsilon n$. 

{For the proof we need the following notation.
Let $D'$ be a subgraph of $D_{G_n}(p')$ such that each vertex in $D'$ appears in a path from $x$ to $y$.}
Let $\mathcal S$ be the set of of maximal strongly connected components in $D'$, 
and let DAG be the directed acyclic graph obtained by contracting all SCCs in $\mathcal S$.
Since adding $(y,x)$ changes the size of $SCC_1$, we know that $SCC(x)\neq SCC(y)$.  
Choose a vertex $v_i$ for each strongly connected component in $ \mathcal S$, giving a set of vertices
$\{v_0, v_1,\dots, v_k\}$ (where we choose $v_0=x$ and $v_k=y$). Order the vertices consistent with the partial order given by DAG. \newproof{Note that $SCC(v_1), SCC(v_2),\ldots, SCC(v_k)$ do not necessarily form a path, however, for all $i>j$, there are no edges  from $SCC(v_i)$ to  $SCC(v_j)$ in the  DAG.}

Define $S_1^i=\cup_{j={0}}^i SCC(v_j)$ and $S_2^i= \cup_{j={i+1}}^k SCC(v_j)$. We claim that there exist some index $s$ such that both $|S_1^s|\geq\epsilon n/2$ and $|S_2^s|\geq \epsilon n/2$. 
We know $|S_1^{k}|\geq \epsilon n$, so let $s$ be the smallest index such that $|S_1^s|\geq \epsilon n/2$ (note that in particular $|S_1^{s-1}|<\epsilon n/2$ if $s\geq 1$). Then, we show that $|S_2^s|\geq \epsilon n/2$.
 The reason is that adding the edge $e$ changed the size of $SCC_1$ and any $SCC(v_i)$ by at least $\epsilon n$, and as a result,
 \[|SCC(v_s)|+\epsilon n\leq \sum_{\ell=0}^k |SCC(v_\ell)| = |S_1^{s-1}|+|SCC(v_s)|+|S_2^s|\leq \epsilon n/2+|SCC(v_s)|+|S_2^s|\]
for $s\geq 1$.  If $s=0$, then
 \[|SCC(v_s)|+\epsilon n\leq \sum_{\ell=0}^k |SCC(v_\ell)| = |SCC(v_s)|+|S_2^s|.\]
In both cases,
$|S_2^s|\geq \epsilon n/2$, which proves our claim.

\newproof{The rest of the proof follows similar ideas as in the proof of  Lemma \ref{lem: structure lemma}, part \ref{lm5: part2}. We will show that there are at least $\epsilon n/4$ nodes in $S_1^s$ with large fan-ins such that their fan-ins does not contain $\epsilon n/2$ nodes in the fan-in of $y$. In fact, all the nodes of $S_2^s$ are in the fan-in of $y$, while none of them appear in the fan-ins of any node in $S_1^2$. We will bound the probability of this event ($\mathbb P (H_e)$) by using Corollary \ref{cor: SCC threshold general graph}.}

 To formalize the proof, let
$c\in (0,{\zeta(q)})$.  
% Note that by our assumptions,
% Decreasing $q$ if needed, we may assume
%that the 
%$|C_1|/n{\to \zeta(p)}$ 
%converges at $p-q$.
%{at $p$.}
%  the giant converges  that $\zeta$ is continuous \todo{replace with assumption that theorem 1 holds at infinitely many points}
%  at $p-\delta$,
{By}
 part \ref{lm5: part3} of Lemma \ref{lem: structure lemma}, 
 there exists some $N_1$ such that for all $n>N_1$, with probability at least $\epsilon^3/2$, 
 $|L^+_c|\geq cn$ in $D_{G_n}({q})$ .  \newproof{Since $|L^+_c|$ is  increasing in $p$ and $p'\geq q$ for all $p'\in I$, we
 conclude that the lower bound on $|L^+_c|$ holds in ${D_{G_n}(p')}$ for all $p'\in I$.}
 Similar to the proof of part \ref{lm5: part2} in Lemma \ref{lem: structure lemma}, we will prove that there exists some $N_2$ such that  for all $p'\in {I}$ and all $n>N_2$, all but $\epsilon n$ nodes of $L^+_c$ appear in the fan-ins of at least half of the nodes in $S_1^s$.

\newproof{Recall that $x$ and $y$ are the two endpoints of the edge $e$.} Call $(x,v)$ an $x$-bad pair if { $|C^+(x)|\geq \epsilon n$ and} $|C^+_{\setminus C^+(x)}(v)|\geq \epsilon^3n/2$, and call $(u,y)$ a $y$-bad pair if  { $|C^-(u)|\geq (c-\epsilon)n$ and} $|C^-_{\setminus C^-(u)}(y)|\geq \epsilon n/2$.
We may use Corollary \ref{cor: random nodes out-component} for uniformly $\epsilon$-unique giants, to conclude that  for any $\delta>0$ and all large enough $n$ the expected number of $x$-bad pairs is at most $\delta^2 n$ for all {$p'\in I$.} By Markov inequality with probability at least $1-\delta$ the number of $x$-bad pairs is less than $\delta n$. Choosing $\delta=\epsilon^3/2$, we therefore get
$$
\mathbb P_{D_{G_n}(p')} (H_e) \leq \mathbb P_{D_{G_n}(p')}(H_e\text{,  ${|L_c^+|\geq cn}$ and $\exists$ at most $\frac{\epsilon^3}{2} n$ $x$-bad pairs})+ {\epsilon^3}.
   $$
Consider now the event that $H_e$ holds, that there are at most $\frac{\epsilon^3}{2} n$ $x$-bad pairs {and that  $|L^+_c|\geq cn$.} Let  $B$ be the bipartite graph on $C^+(x)\times L^+_c$ where there is an edge between $u\in C^+(x)$ and $v\in L^+_c$ if $u\in C^+(v)$. Let $X$ be the number of nodes in $C^+(x)$ that have at most $|L^+_c|-\epsilon n$ nodes in their fan-ins. We  then can proceed
as in part \ref{lm5: part2} of Lemma \ref{lem: structure lemma} to conclude that %\todo{here we use part 2 of Lemma 5.2.}
\[|X|(|L^+_c|-\epsilon n)+(|C^+(u)|-|X|)|L^+_c|\geq (|C^+(u)|-\frac{\epsilon^3 n}{2})(|L^+_c|-\frac{\epsilon^3 n}{2}).
\]
Therefore, $|X|\leq\epsilon^2n$.  
 Since $S_1^s\subseteq C^+(x)$ and $|S_1^s|\geq\epsilon n/2$, there are at least $\epsilon n/4$ nodes $u\in S_1^s$  such that $|C^-(u)|\geq|L^+_c|-\epsilon n\geq (c-\epsilon)n$. {But since} the fan-in of a node  $u\in S_1^s$ does not contain any node from $S_2^s\subseteq C^-(y)$, {we have that  $|C^-_{\setminus C^-(u)}(y)|\geq |S_2^s|\geq
 \epsilon n/2$.  Thus, for all these nodes $u$, the pair $(u,v)$ is a $y$-bad pair.}
 As a result,
 \begin{align*}
    \mathbb P_{D_{G_n}(p')} (H_e)&\leq
\mathbb P_{D_{G_n}(p')}(H_e\text{,  ${|L_c^+|\geq cn}$ and $\exists$ at most $\frac{\epsilon^3}{2} n$ $x$-bad pairs})+{\epsilon^3} \\
   &\leq \mathbb P_{D_{G_n}(p')}(\exists \text{ at least $\frac{\epsilon}{4} n$ $y$-bad pairs})+\epsilon^3\leq \epsilon
 \end{align*}
 where the last inequality is obtained again by Corollary \ref{cor: random nodes out-component} for large enough $n$.
 
 In the above arguments, the choice of $n$   {is} independent  of $x$ and $y$  due to the fact that  the convergence in Corollary \ref{cor: random nodes out-component} is uniform in the fixed vertex $v$. Thus, the event
 $H_e$ takes place with probability at most $\epsilon$ for any edge $e$.  As a result, for all $\epsilon>0$, all large enough $n$, all $e\in E(G_n)$,
 \newproof{and all $p'\in I$}
 \[\mathbb E|\Delta_e SCC_1|\leq \epsilon n+ n\mathbb{P}_{D_{G_n}(p')}(H_e)\leq 2 \epsilon n, \]
as desired.
 \end{proof}
 
 The following is the  key lemma used in the proof of Theorem \ref{thm: main structure SI}. It shows that the strongly connected component exists with high probability in the super critical regime. The main ingredients of the proof  are the concentration bounds given in Section \ref{sec: Russo} and the bounds on the influence of an edge on the size of largest SCC. Combined with %Lemma~\ref{cor:pc-lower-bd}
Corollary~\ref{cor: SCC threshold general graph}
 and since $p_c(\mu)$ is a threshold for the giant in $G_n(p)$,
 the lemma establishes that for a graph sequence that satisfying the assumptions of Theorem \ref{thm: size of giant in expander},
 $p_c^{SCC}(G_n)\to p_c(\mu)$. 
 
\begin{lemma}\label{lm: variance}
Let $\{G_n\}_{n\in\mathbb N}$ be a sequence of 
{graphs} satisfying the assumptions of Theorem \ref{thm: main structure SI}
{in an interval $[p-q,p]$. If $\zeta(p)>0$ and}
%Define $\zeta_-(p)=\lim_{p'{\uparrow p} }\zeta(p')$. 
%Then 
$\alpha\in(0,\zeta^2(p))$, {then}
$$\mathbb{P}_{D_G(p)}(|SCC_1|\geq \alpha n)\rightarrow 1.$$
\end{lemma}

\begin{proof}
\newproof{Fix $\epsilon'>0$ such that  $0<\epsilon'<\zeta(p)$.  To prove the lemma, we need to show that
\begin{equation}\label{SCC1-lower-bd}
    \mathbb{P}_{D_G(p)}\Big(|SCC_1|\geq (\zeta(p)-\epsilon')^2 n\Big)\rightarrow 1.
\end{equation} 
By the continuity of $\zeta$ at $p$ one can find  $q'<q$  such that $|\zeta(p-q')- \zeta(p)|\leq \frac{\epsilon'}4$, implying in particular that $\zeta(p-q')>0$. 

Let
$I=[p-q',p]$, and}
{let $m$ be twice the number of edges in $G_n$, i.e., let $m$ be the number of possible,
oriented edges in $D_{G_n}(p)$.}
We first use Lemma \ref{lm: falik} to prove that given  $\epsilon>0$, {there exists $N>0$ such that} the following holds for 
all$n>N$  and all \newproof{$p'\in I$}
\begin{equation}\label{var-SCC-bd}
    \operatorname{var}_{D_{G_n}(p')}{(|SCC_1|)}\leq \epsilon n\frac{d}{dp'}\mathbb{E}_{D_{G_n}(p')}|SCC_1|.
\end{equation}
To prove \eqref{var-SCC-bd}, we consider two cases based on whether $\frac{\mathcal{E}_2(SCC_1)}{\mathcal{E}_1(SCC_1)}$ is larger or smaller than 
$M=e^{1/\epsilon}$.

\noindent\textit{Case 1}: For $M$ defined as above $\mathcal{E}_2(SCC_1)\geq M \mathcal{E}_1(SCC_1)$.
We claim that this assumption implies that
\begin{equation}
\label{var-bd-case1}
\operatorname{var}_{D_G(p')}(|SCC_1|)\leq\frac{\log\frac{1-p'}{p'}}{(1-2p')} \frac{\mathcal{E}_2(SCC_1)}{\log(\frac{M}{\log M})}.\end{equation}
To see this, assume that $\operatorname{var}_{D_G(p')}(|SCC_1|)\geq \frac{\mathcal{E}_2(SCC_1)}{\log(M)}$ (otherwise \eqref{var-bd-case1} holds 
{by the fact that $1\leq 2\leq \frac{\log\frac{1-p'}{p'}}{1-2p'}$} and $\frac 1{\log M}\leq \frac 1{\log(M/\log M)}$). Then 
$$
\frac{\operatorname{var}(|SCC_1|)}{\mathcal{E}_1(SCC_1)}\geq \frac{M \operatorname{var}(|SCC_1|)}{\mathcal{E}_2(SCC_1)}\geq\frac{M}{\log M}
$$
and \eqref{var-bd-case1} follows by Lemma  \ref{lm: falik}.

Next note that $|SCC_1|\leq n$ implying that   $\mathcal{E}_2(f)\leq n \sum_e\mathbb{E}_S|\Delta_eSCC_1|$.  Combined with \eqref{var-bd-case1} and Lemma \ref{lm: derivative} we conclude that
\[\operatorname{var}_{D_G(p')}(|SCC_1|)\leq 
\frac{p'(1-p')\log\frac{1-p'}{p'}}{(1-2p')}\frac{n}{\log(\frac{M}{\log M})}\frac{d}{dp'} \mathbb{E}_{D_G(p')}|SCC_1|.\]
It is easy to see that the first quotient is bounded by $1/2$  (e.g., expanding both the numerator and denominator around $p'=1/2$ and comparing the derivatives).  Therefore
%Since 
\[\frac{p'(1-p')\log\frac{1-p'}{p'}}{(1-2p')\log(\frac{M}{\log M})}\leq \frac 1{2\log(\frac{M}{\log M})}\leq \epsilon,\]
 proving \eqref{var-SCC-bd} for Case 1.

\noindent \textit{Case 2}: For $M$ defined as above, $\mathcal{E}_2(SCC_1)\leq M\mathcal{E}_1(SCC_1)$. 
In this case, we will use the 
the Efron-Stein inequality to bound the variance by
 \[\operatorname{var}_{D_{G_n}(p')}(|SCC_1|)\leq\frac{1}{2}\mathcal{E}_2(SCC_1)
 \leq \frac M2 \mathcal{E}_1(SCC_1).\]
But in this case, the bound $\mathcal{E}_2(f)\leq n \sum_e\mathbb{E}_S|\Delta_eSCC_1|$ is not strong enough to
complete the proof of \eqref{var-SCC-bd}.  To overcome this, {we use }  Lemma \ref{lm: bounded influence}, which implies that for any constant $\epsilon>0$, and {large enough $n$}, $\mathbb{E}{|}\Delta_e SCC_1{|}\leq \epsilon n/{M}$ for all edges $e$ and all \newproof{$p'\in {I}$}.
As a result,
\begin{align*}
\mathcal{E}_1(SCC_1)
&=\sum_{e\in [m]}\left( \mathbb{E}_{D_{G_n}(p')}[\Delta_e SCC_1]\right)^2\\
&\leq \frac{\epsilon n}M\sum_{e\in[m]} \mathbb E_{ D_{G_n}(p')} |\Delta_e SCC_1|=\frac{\epsilon}M p'(1-p')n \frac{d}{dp'}\mathbb E_{ D_{G_n}(p')}|SCC_1|,
 \end{align*}
 resulting in 
  \[\operatorname{var}_{D_{G_n}(p')}(|SCC_1|)\leq \epsilon {\frac{p'(1-p')}2 }n\frac{d}{dp'}\mathbb E_{ D_{G_n}(p')}|SCC_1|, \]
{Since $p'(1-p')\leq 2$,}  we get {\eqref{var-SCC-bd}} in the second case as well.

We are now ready to prove \eqref{SCC1-lower-bd}.
Given $\epsilon>0$, let $ N$ 
% \todo{What is N}
be such that for $n\geq  N$ the bound
\eqref{var-SCC-bd} holds for all $p'\in {[p-q',p]}$. 
We claim that given $n\geq  N$ there exists  $p'=p'(n)\in {[p-q',p]}$
%(p-q,p]$,
\[ \frac{d}{dp'} \mathbb{E}_{D_{G_n}(p')}|SCC_1|\leq 2n/q'. \]
 Indeed, assume this is not the case, then by the fundamental theorem of calculus,%\todo{Does this hold with countable discontinuity points?}
$$
\mathbb{E}_{D_{G_n}(p)}|SCC_1|\geq 2n +\mathbb{E}_{D_{G_n}(p-{q'})}|SCC_1|,
$$
which is a contradiction, since $|SCC_1|\leq n$ with probability $1$.
Then by \eqref{var-SCC-bd} for $n\geq  N$ and 
%\newproof{$p'\in(p-q,p]$},
{$p'=p'(n)\in [p-q',p]$}
  \[\operatorname{var}_{D_{G_n}(p')}(|SCC_1|)\leq \frac{2\epsilon}{{q'}} n^2.\]
\newproof{  Next we use Lemma~\ref{lem:E0fSSC1^2} 
  together with convergence of $|C_1|/n$ to conclude that for all
  $\epsilon''>0$ with
   $\epsilon''<\zeta(p-q)$ there exists an $N'<\infty$ such that for $n\geq N'$,
$$
\mathbb{E}_{D_{G_n}(p')}(|SCC_1|^2)
\geq \mathbb{E}_{D_{G_n}(p-q')}(|SCC_1|^2)
\geq \left(\zeta(p-q')-\epsilon''\right)^4n^2.
$$
Choosing $\epsilon''=\frac{\epsilon'}4$ and using that 
$|\zeta(p)-\zeta(p-q')|\leq \epsilon'/4$, we thus have}
$$
\mathbb{E}_{D_{G_n}(p')}(|SCC_1|^2)
\geq\bigg(\zeta(p) - \frac{\epsilon'}2\bigg)^4 n^2,
$$
for all $n\geq N'$.
With this lower bound on the expectation of $|SCC_1|^2$ and the variance bound,
\begin{equation}\label{eq: low bnd E[SCC1]}
\bigg(\mathbb{E}_{D_{G_n}(p')}(|SCC_1|)\bigg)^2\geq
\bigg(\zeta(p) - \frac{\epsilon'}2\bigg)^4 n^2
-\frac{2\epsilon}{{q'}} n^2
\end{equation}
and, by 
Chebyshev's inequality,
\[\mathbb{P}_{D_{G_n}(p')}\Big(\Big|\frac{|SCC_1|}{n}-\mathbb{E}[\frac{|SCC_1|}{n}]\Big|\geq \frac 14\epsilon'  \Big)\leq\frac{32\epsilon}{{q'}\epsilon'}.\]
We therefore have shown that given $\epsilon>0$ small enough (depending on ${q'}$, $\epsilon'$ and $\zeta(p)$) there exists $\tilde N<\infty$ such that for all $n\geq \tilde N$ there exists $p'=p'(n)\in {[p-q',p]}$ such that
\[
\mathbb{P}_{D_{G_n}(p')}\Big(\frac{|SCC_1|}{n}\geq \Big(\zeta(p) - \epsilon'\Big)^2 \Big)\leq\sqrt\epsilon.
\]
Since $|SCC_1|$ is increasing in $p$ this implies that for all sufficiently small
$\epsilon>0$
there exists an $\tilde N'<\infty$ such that
\[
\mathbb{P}_{D_{G_n}(p)}\Big(\frac{|SCC_1|}{n}\geq \Big(\zeta(p) - \epsilon'\Big)^2 \Big)\leq\sqrt\epsilon,
\]
for all $n\geq \tilde N'$.  This proves \eqref{SCC1-lower-bd}.
\end{proof}
Now, we are ready to proceed with the proof  of  Theorem \ref{thm: main structure SI}.

\begin{proof}[\textbf{Proof of Theorem \ref{thm: main structure SI}}]

% First, note that {that $\zeta(p)=0$ implies that $p$ is a continuity point of $\zeta$ and $p\leq p_c(\mu)$.  Keeping this in mind, }
% Lemma \ref{lm: variance},   {Corollary}~\ref{cor:pc-lower-bd} and Theorem \ref{thm: size of giant in expander} imply part \ref{part: thm2 SCC1}.

Part \ref{part: thm2 SCC2} follows by the assumption of the theorem on uniquness of the second largest component: we know that for all $\epsilon>0$, $G_n$ has a
uniformly $\epsilon$-unique giant, and  Part \ref{part: thm2 SCC2} follows by  Lemma \ref{lm: unique SCC}. 

Next, to prove Part \ref{part: thm2 subcritical}, Corollary \ref{cor:pc-lower-bd} implies that for all  $\epsilon>0$,
\begin{align*}
 \mathbb P(\text{there exists $\geq \epsilon n$ vertices $v$ with }|C^+(v)|\geq \epsilon n)\to 0.
\end{align*}\newproof{
As a result  for a uniform random vertex $v$, $\frac{|C^+(v)|}{n}\overset{\mathbb P}{\to}0$. A similar argument implies the statement  for fan-ins. Furthermore, in any instance of $D_{G_n}(p)$ {with} $|SCC_1|\geq\epsilon n$ the probability that a uniform random node has fan-out larger than $\epsilon n$ is at least $\epsilon$. So, $\frac{|SCC_1|}{n}$ must also converge to $0$ in probability.}

\newproof{Next we prove Part~\ref{part: thm2 bowtie}. The statement $$\liminf_{n\to\infty}\frac 1n{\mathbb E}[|SCC_1|]\geq \zeta^2(p)$$ follows {from}  Equation \eqref{eq: low bnd E[SCC1]}. Further, in Lemma~\ref{lm: variance} we showed that for any fixed $\epsilon>0$, $$\limsup_{n\to\infty} \mathbb{P}_{D_{G_n}(p)}\Big(\Big|\frac{|SCC_1|}{n}-\mathbb{E}[\frac{|SCC_1|}{n}]\Big|\geq  \epsilon  \Big)=0,$$
which implies $
\frac{|SCC_1|}{{\mathbb E}[|SCC_1|]}\overset{\mathbb{P}}{\to} 1$.

Next, {choose} $v\in V(G_n)$ %be chosen 
uniformly at random. {We will prove that}
%, and we prove 
either $v\notin SCC_1^-$ and $|{C}^+(v)|=o(n)$ or $v\in SCC_1^-$ and $|C^+(v)\Delta SCC_1^+|=o(n)$.} {Recalling the definition of 
the sets $\tilde O_\epsilon$ and $O_\epsilon$ from 
Lemma \ref{lm: L+ to SCC}, the statements of the lemma then imply that 
 \[\frac{|O_\epsilon|}{n}\rightarrow 0
 \qquad\text{and}\qquad \frac{|\tilde O_\epsilon|}{n}\rightarrow 0\]
in probability.
The second statement implies that if a random vertex $v$
does not fall into $SCC_1^-$ (which is equivalent to
$C^+(v)\cap SCC_1^+=\emptyset$), with high probability
its fan-out has $o(n)$ vertices, proving $|C^+(v)|=o(n)$ for this case. If a random vertex $v$ falls into $SCC_1^-$, by the first statement of Lemma \ref{lm: L+ to SCC}, we know that when considering the induced  subgraph on the complement of $SCC_1^+$, the fan-out of $v$ is of size at most $o(n)$.
But the fan-out of $v$ in this induced subgraph is nothing 
but $C^+(v)\setminus SCC_1^+=C^+(v)\Delta SCC_1^+$,
proving $|C^+(v)\Delta SCC_1^+|=o(n)$ for the fan-out of $v$.
}
The same argument works by symmetry for fan-ins.

The rest of this proof is dedicated to convergence of the relative size of $SCC_1^+$ and $SCC_1^-$.
Since 
$\zeta(p)>0$,  there exists some $\alpha>0$ such that with high probability $|SCC_1|\geq \alpha n$. 
Since any node in $SCC_1$ has $SCC_1^+$ in their fan-out, for small enough $\epsilon>0$,
\[\mathbb P(\frac{|SCC_1^+|}{n}\geq\zeta(p)+\epsilon)\leq \mathbb P(|L^+_{\zeta(p)+\epsilon}|\geq \alpha n).\]
 Then by  part \ref{lm5: part1} of Lemma \ref{lem: structure lemma},
 \[\mathbb P(\frac{|SCC_1^+|}{n}\geq\zeta(p)+\epsilon)\rightarrow 0.\]
 
To prove the lower bound assume to the contrary that there exists $\delta>0$ such that  for infinitely many $n$,
\[\mathbb P(\frac{|SCC_1^+|}{n}\leq\zeta(p)-\epsilon)\geq \delta.\]
 For any vertices $u,w\in SCC_1$ note that  $C^+(w)=C^+(u)=SCC_1^+$. Since  $|SCC_1|\geq \alpha n$, with probability greater than $\alpha$ a random node lies in $SCC_1$. Therefore for a random node $v$
 \[\mathbb P_{D_{G_n}(p)}(\alpha \leq \frac{|C^+(v)|}{n}\leq \zeta(p)-\epsilon )\geq \alpha \delta.\]
By Lemma \ref{lem: coupling},
\[
\mathbb P_{G_n(p)}(\alpha \leq \frac{|C(v)|}{n}\leq \zeta(p)-\epsilon )\geq \alpha\delta,
\]  contradicting the bounds in Proposition~\ref{prop: lower bound giant}.
 Therefore, we must have
 \[\mathbb P(\frac{|SCC_1^+|}{n}\leq\zeta(p)-\epsilon)\rightarrow 0.\]
 Since $\epsilon>0$ was arbitrary we get the result. By symmetry the same holds for $SCC_1^-$.
 
\end{proof}

\section{Applications to Preferential Attachment Graphs}\label{sec: pref attachment}
 As an application of  our method to power law graphs, we consider percolation on preferential attachment graphs. 
% In this section, we analyze the limit of preferential attachment graphs which was shown in \cite{berger2014} to be a multi-type branching process of the form considered in the last section.
Here we consider the following version of preferential attachment, which closely follows the original formulation by Barab\'asi and Albert \cite{Barabasi}.  The model has a parameter $m\in \mathbb N$, and is defined as follows.  Starting from a connected  graph  
$G_{t_0}$
on at least
$m$ vertices, a random graph $G_t$ is defined inductively: given $G_{t-1}$ and its degree sequence $d_i(t-1)$, we form a 
new graph by adding one more vertex, $v_t$, and connect it
to $m$ distinct  vertices $w_1,\dots,w_m\in V(G_{t-1})$ by first choosing $w_1,\dots,w_m\in V(G_{t-1})$ {i.i.d with distribution $\mathbb P(w_s=i)=\frac {d_i(t-1)}{2|E(G_{t-1})|}$, $s=1,\dots m$,} and then conditioning on all vertices being distinct (thus avoiding  multiple edges).  {While all our results hold for arbitrary connected starting graphs on at least $m$ vertices, it will be  notationally convenient to choose $G_{t_0}$ is such a way that at time $t\geq t_0$, the graph has {$t$} vertices and {$mt$}
 edges.  For concreteness, we choose 
$G_{t_0}$ to be the graph $K_{2m+1}$, the complete graph on $t_0=2m+1$ vertices.  We denote the resulting random graph sequence by $({P{\hskip-.2em}A}_{m,n})_{n\geq 2m+1}$, and following \cite {berger2014}, we call the version of preferential attachment we defined above the \emph{conditional model}, while the model where the conditioning step is left off will be called the independent model.}

There are several papers establishing that the percolation threshold is $0$ for variants of this problem, see, e.g., \cite{bollobas2003} for site percolation on a different preferential attachment model that allows multiple edges and self-loops, and \cite{Dereich11,Dereich2013RandomNW} for bond percolation on  what is called Bernoulli preferential attachment\footnote{In this model, the number of new edges is not specified, but instead is a random variable which is sum of $n$ Bernoulli random variables skewed towards higher degrees.} in \cite{hofstadVol1}.  {Note that 
the results of 
 \cite{bollobas2003} give{s} an easy proof that $p_c=0$ for 
 bond percolation on the Bollobas-Riordan version of  preferential attachment models as well.  All one needs to observe is that bond percolation with probability $p$ gives a stochastic upper bound on site percolation with probability  $p'=p^m$ (take a bond percolation configuration, and delete all vertices for which at least one of the $m$ initial edges is absent).  This does not quite give a proof for the conditional model  $({P{\hskip-.2em}A}_{m,n})_{n\geq 2m+1}$ considered here since the two models differ in minor technical details, but more importantly, we  (a) want to demonstrate the power of the methods developed in this paper, establishing this result from scratch, and (b) we will be able to obtain sharper bound on the relative size of the largest cluster.}

%But to our knowledge, our proof below is the first proof establishing the existence and asymptotic size of the giant for all $p>0$ for bond percolation for the original preferential attachment model.

Before {stating} the theorem, we  point out that the sequence  $\{{P{\hskip-.2em}A}_{m,n}\}_{{n\geq m}}$ converges locally in probability  to a P\'olya-point process \cite{berger2014, RemcoVol2}.  As we will see, the robustness of ${P{\hskip-.2em}A}_{m,n}$ then reduces to the robustness of  P\'olya-point processes to bond-percolation. Specifically, the relative size $\zeta(p)$ appearing in the next theorem is the survival probability of the P\'olya-point processes after bond-percolation, and the statement that $p_c=0$ for preferential attachment reduces to the statement that $\zeta(p)>0$ for all $p>0$.

\begin{theorem}\label{thm: pref-attachment}
Let $m\geq 2$,   for a positive integer {$n\geq 2m+1$} let ${P{\hskip-.2em}A}_{m,n}$ be the conditional preferential attachment graph defined  above, let $p\in [0,1]$, 
and let $C_1$ and $C_2$ be the
the largest and second largest connected component in ${P{\hskip-.2em}A}_{m,n}(p)$, respectively.  Then the following limits exist
\[\frac{|C_1|}{n}\overset{\mathbb{P}}{\to}\zeta(p)
\quad\text{and}\quad \frac{|C_2|}{n}\overset{\mathbb{P}}{\to}0,
\]
where  $\zeta:[0,1]\to [0,1]$ is a continuous function with {$\zeta(p)=e^{-\Theta(1/p)}$ as $p\to 0$ and $\zeta(1)=1$.} 
%\zeta(p)\geq {(1-p^m)}e^{-1/p}$ for all $p\in (0,1)$.  
So in particular, the largest component in ${P{\hskip-.2em}A}_{m,n}(p)$ has linear size for all $p>0$, showing that $p_c=0$.
\end{theorem}
{The theorem will follow from Theorem \ref{thm: size of giant in expander} once we establish}
(1) large-set expansion of ${P{\hskip-.2em}A}_{m,n}$, 
(2) {continuity of the survival probability $\zeta(p)$ of the P\'olya-point processes after bond-percolation}, and (3)
{the bounds $\zeta(p)\to 1$ as $p\to 1$ and
$\zeta(p)=e^{-\Theta(1/p)}$ as $p\to 0$.}

Theorem 1 in \cite{MIHAIL2006239}, shows {positive edge expansion for a different version} of preferential attachment. Following a similar argument, we prove in Appendix \ref{sec: expansion pref attachment} that ${P{\hskip-.2em}A}_{m,n}$ has positive large set expansion, {which is weaker than the expansion established in \cite{MIHAIL2006239}, but sufficient for our purpose.}
\begin{lemma}\label{lem: pref-attach expansion}
Let $m\geq 2$ and $n\geq m$, and let ${P{\hskip-.2em}A}_{m,n}$ be defined as above. Then there exists some $\alpha>0$ such that for any $\epsilon\in(0,1/2)$ 
and all large enough $n$, 
$\{{P{\hskip-.2em}A}_{m,n}\}_{n\geq m}$ is an $(\alpha,\epsilon, 2m)$ large-set expander with probability $1-\epsilon$.
\end{lemma}

{To continue, we will use the
explicit construction of the local  limit of preferential attachment in \cite{berger2014}, which gives what the authors call a 
 P\'olya-point graph or process.}
This graph is a random rooted tree, where vertices 
have types $(S,x)\in \{\emptyset,R,L\}\times [0,1]$, where the discrete label is $S=\emptyset$ for the root, and right (R) or left (L) for all other vertices.  We will refer to the continuous label $x$ as the ``position'' of a point in the P\'olya-point graph.
The root, with type $(\emptyset,x)$, has a random position $x=\sqrt{y}$
where $y$ is drawn uniformly at random from $[0,1]$. For a vertex of type $(S,x)$ define \[   m(S)= 
\begin{cases}
   m,& \text{if } S= L \text{ or } S=\emptyset\\
 m-1,              & \text{if } S=R.
\end{cases}
\]
For a vertex of type $(S,x)$ the off-springs are generated as follows.
\begin{itemize}
\item Each such vertex has a deterministic number $m(S)$ of children  of type $(L,x_i)$, $i=1,\dots,m(S)$, {where} $x_1,\ldots, x_m(S)$ are chosen i.i.d uniformly at random from $[0,x]$.

\item In addition, it has   $N\sim \operatorname{Poi}(\gamma \frac {1-x}x)$  right children, where  
$\gamma\sim\Gamma(m+1,1)$ if $S=L$,
and $\gamma\sim\Gamma(m,1)$ if $S\in\{R,\emptyset\}$. Given $N$, the right children have type
$(R,y_1),\ldots, (R,y_N)$,
where 
 $y_1,y_2,\ldots,y_N$ are chosen i.i.d uniformly at random from $[x,1]$.
\end{itemize}

\label{back}
\noindent

Next we discuss how to compute $\zeta(p)$ for the  P\'olya-point graph.  To this end,
we 
derive the implicit formula for the survival probability $\rho(S,x)$ of the tree under a node of type $(S,x)$ after percolation, with $(S,x)\in \{L,R\}\times [0,1]$.  Let $d_R^p$ and $d_L^p$ be the random number of right and left children of such a node after percolation.  
Intuitively, the extinction probability $1-\rho(S,x)$ is equal to the probability that all of its children do not appear in an infinite cluster.    Taking first the expectation over the positions of these children and then over the number of left and right children after percolation
will  give an  implicit equation for $\rho(S,x)$, showing that $\rho(S,x)$ is a solution of
\begin{equation}\label{eq: fixed point}
    (\Phi f)(S,x)=f(S,x),
\end{equation}
where {$S\in\{L,R\}$, and}
\begin{equation}
\label{phi-def}(\Phi f)(S,x)=1-\frac{x\Big(x-p\int_0^xf(L,y)dy\Big)^{m(S)}}{\Big(x+p\int_x^1f(R,y)dy\Big)^{m(S)+1}},\end{equation}
see Appendix~\ref{sec: appendix polya point} for the derivation of \eqref{eq: fixed point}.  As is typical for implicit equations for survival probabilities in branching processes, the above equation has a trivial solution $f(x,S)\equiv 0$, raising the question of whether there exist other solutions, and if so, which one is the survival probability $\rho(S,x)$.  As also typical, $\rho(S,x)$ will be the maximal solution, which here means the point-wise maximum over all solutions.  The exact statement is given in Proposition~\ref{prop: fixed point Phi} below, whose proof is also given in Appendix~\ref{sec: appendix polya point}.  

Before stating the proposition, we note that
once we know  $\rho(S,x)$ for all vertices of discrete type $L$ or $R$, we can calculate the survival probability for the root in exactly the same way, except that we now also need to integrate over the position $x$ of the root, which we recall is equal to $\sqrt y$ where $y$ is uniform in $[0,1]$.  This leads to the equation

\begin{equation}\label{eq: zeta PA}
\quad\zeta(p)=\int_0^1  ({\Phi} \rho)(\emptyset,\sqrt y) dy,\quad\text{where}\quad
%\\&
({\Phi} f)(\emptyset,x)=
1-\Big(\frac{x-p\int_{0}^{x}f(L,z)dz}{x+p\int_{x}^{1}f(R,z)dz}\Big)^m,
\end{equation}
see again Appendix~\ref{sec: appendix polya point} for the proof. Note that we extended the domain of $\Phi$ in \eqref{phi-def} to $(S,x)\in\{\emptyset,L,R\}\times[0,1]$.
To formulate Proposition~\ref{prop: fixed point Phi}, we introduce one more quantity, 
the probability that a node of type $(S,x)$ reaches level $k$ after percolation with probability $p$, a quantity we denote by 
$\rho_k(S,x)$. 
\begin{prop}\label{prop: fixed point Phi}
Let $p>0$ and let $S\in \{L,R\}$.
Then the following holds.
\begin{enumerate} 
    \item Let $\rho_k(S,x)$ be the probability that a node of type $(S,x)$ reaches level $k$ after percolation with probability $p$. Then $\rho_k(S,x)=(\Phi^k1)(S,x)$ for all $k\geq 0$ and all $x\in [0,1]$.
   \item The survival probability $\rho(S,x)$ is the maximum solution of \eqref{eq: fixed point}, i.e., for any other solution $f$ we have that $\rho(S,x)\geq f(S,x)$ for all $x\in [0,1]$. 
\end{enumerate}
\end{prop}

We will use this proposition together with \eqref{eq: fixed point}, \eqref{phi-def} and \eqref{eq: zeta PA} to establish the following bounds on the survival probabilities $\rho(S,x)$ and $\zeta(p)$. 

\begin{prop}\label{cor: lower bnd zeta pref-attach}
Let $p>0$.  Then
\begin{equation}\label{zeta-mathcing-bds}
{e^{-\frac{1}{p(m-1)}}}
\leq
\zeta(p)
\leq{2m}e^{-\frac{1-2p}{(m+1)p}}. 
\end{equation}
\end{prop}

Note that by part 1 of Proposition \ref{prop: fixed point Phi}, $\rho_{k+1}=\Phi\rho_k$. The main idea of the proof is to 
{establish upper and lower bounds of the form}
\[f^-_k(x)\leq \rho_k(S,x)\leq f^+_k(x),\text{ where } f_k^\pm(x)=1-\Big(\frac{1}{1+\frac{\epsilon_k^\pm}{x}}\Big)^{m\pm 1},\]
and {where} $\epsilon^\pm_k$ are defined recursively. Then we get the result by showing that the limit $\lim_{k\rightarrow\infty}\epsilon^\pm_k$ exists and is of order $e^{\Theta(\frac{1}{p})}$.
{See Appendix \ref{sec: appendix polya point} for the complete proof.}

To finish the Theorem of \ref{thm: pref-attachment},
{we} note that \newproof{ the continuity of $\zeta$ is already known from Corollary~\ref{cor: zeta-cont}.}

After these preparations, the proof of Theorem \ref{thm: pref-attachment} is now almost obvious.

\begin{proof}[Proof of Theorem \ref{thm: pref-attachment}]\newproof{
First,  we note that by the lower bound in Proposition~\ref{cor: lower bnd zeta pref-attach} $p_c=0$. So by large-set expansion of PA the continuity of $\zeta$ for $p>0$ follows from Corollary~\ref{cor: zeta-cont}. So, we need to show  $\zeta(p)\to 1$ as $p\to1$. This follows from  the fact that in the P\'olya{-}point graph, the root has $m$ left children, each of these left children have again $m$ left children, etc., to bound $\zeta(p)$ from below by the survival probability for percolation on a tree where the root has degree $m$, and all other vertices have degree $m + 1$. 
}

Further, by Lemma \ref{lem: pref-attach expansion}, the sequence $\{{P{\hskip-.2em}A}_{n,m}\}_{n\geq m}$ are large-set expanders.
Also, by \cite{berger2014} and Theorem 5.8 in \cite{RemcoVol2}, this sequence converges locally in probability  to the P\'olya-point graph. Thus we can use 
 Theorem \ref{thm: pref-attachment} to get that  $|C_2|/n$ converges to zero in probability for all $p\in [0,1]$, and that $|C_1|/n$ converges to $\zeta(p)$ for all continuity
 points of $\zeta$, which is all $p\in [0,1]$ as well.
\end{proof}

\begin{remark}\label{rmk: scc1 PA}
{Using the results from the previous sections,} it is easy to see that the relative size of $SCC_1$ {for directed percolation on preferential attachment graphs} is of order $e^{-\Theta(p^{-1})}$ as well. To see this, {we first} note that by Lemma \ref{lem: upper bound SCC1} and Lemma~\ref{lm: zeta+- > zeta^2}, 
\begin{align*}
 \mathbb P_{D_{G_n}(p)}\big(\frac{|SCC_1|}{n}\geq \zeta(p)+\epsilon\big)\leq
    \mathbb P_{D_{G_n}(p)}\big(\frac{|SCC_1|}{n}\geq \zeta^{+-}(p)+\epsilon\big)\rightarrow 0.
\end{align*}
Applying the upper bound in Proposition \ref{cor: lower bnd zeta pref-attach}, {this gives}
\begin{align*}
 \mathbb P_{D_{G_n}(p)}\big(\frac{|SCC_1|}{n}\geq (2m+1)e^{-\frac{1-2p}{(m+1)p}}+\epsilon\big)\rightarrow 0.
\end{align*}
{For a lower bound, we use}  Lemma~\ref{lm: variance} and again Proposition \ref{cor: lower bnd zeta pref-attach}
{to get}
\begin{align*}
 \mathbb P_{D_{G_n}(p)}\big(\frac{|SCC_1|}{n}\geq (1-p^m)e^{-\frac{1}{p}}-\epsilon\big)\rightarrow 1.
\end{align*}

\end{remark}

 \section*{Acknowledgements}
The authors  thank Remco van der Hofstad, for  insightful communications on  local limits for random graph sequences,  Jennifer Chayes for discussions concerning percolation, and Persi Diaconis for feedback on an earlier version of this paper. \newproof{Finally, we  would like to thank our anonymous reviewers for their insightful comments and suggestions which greatly improved our paper. 

\noindent Yeganeh Alimohnammadi and Amin Saberi are supported by NSF grant CCF1812919.
}

%%%%%%%%%%%%%%%%%%%%%%%%%%%%%%%%%%%%%%%%%%%%%%%%%%%%%%%%%%%%%
%%                  The Bibliography                       %%=
%%%%%%%%%%%%%%%%%%%%%%%%%%%%%%%%%%%%%%%%%%%%%%%%%%%%%%%%%%%%%

%% if your bibliography is in bibtex format, uncomment commands:
\bibliographystyle{imsart-number} % Style BST file (imsart-number.bst or imsart-nameyear.bst)
\bibliography{ref}       % Bibliography file (usually '*.bib')
\newpage

%%%%%%%%%%%%%%%%%%%%%%%%%%%%%%%%%%%%%%%%%%%%%%
%%                  Appendix                %%
%%%%%%%%%%%%%%%%%%%%%%%%%%%%%%%%%%%%%%%%%%%%%%
\begin{appendix}
\newproof{
\section{Continuity of $\zeta$ for random sequence of expanders}\label{sec: contintuiy of zeta}
In this Section, we will prove Corollary~\ref{cor: zeta-cont} and show that the percolation function $\zeta$ is continuous for large-set expanders that converge locally in probability. As discussed in Section~\ref{sec:zeta-cont}, it is enough to prove that the limit of a sequence of large-set expanders is an ergodic (extremal) unimodular random graph. Very recently, Sarkar \cite{sarkar2018note} proved ergodicity of the limit for deterministic expanders with bounded degree. We will show that their proof is extendable to possibly random sequence of large-set expanders with bounded average degree.

To state the lemma we need the following definition. A measurable function $f:\mathcal G_*\to\mathbb R$ is rerooting-invariant if its value stays invariant under changes in  the position of the root.
\begin{lemma}
 Let $\{G_n\}$ be a sequence of possibly random $(\alpha,\bar{d})$ large-set expanders obeying the assumptions of Theorem~\ref{thm: size of giant in expander}, and let $\mu$ be the limit. Then $(G,o)\sim\mu$ is ergodic. That is, if $f$ is any rerooting-invariant function, then $f(G)$ is constant almost surely.
\end{lemma}

\begin{proof}
Following the notation of Sarkar \cite{sarkar2018note}, for two rational numbers $0\leq a<b\leq 1$ let 
\[\Gamma_1=\{(G_1,o)\in \mathcal G_*: f(G)\leq a\},\qquad\qquad\Gamma_2=\{(G_1,o)\in \mathcal G_*: f(G)\geq b\}.\]
It is enough to show that $\mu(\Gamma_1)$ and $\mu(\Gamma_2)$ cannot be both positive, which is sufficient to prove the statement of the lemma. To prove this by contradiction assume that there exists $p_0>0$ such that $\mu(\Gamma_1)>p_0$ and  $\mu(\Gamma_2)>p_0$. 

Now, by Theorem A.7 of \cite{RemcoVol2} we know that $\mathcal G_*$ is a Polish metric space. So $\mu$ is tight and regular since it is a probability measure on a Polish space (see e.g., Theorem 1.3. \cite{billingsley2013convergence}, and Chapter II, Theorem 1.2 of \cite{parthasarathy2005probability}). So, there exists compact sets $H_i\subseteq\Gamma_i$ such that $\mu(H_i)\geq p_0/2$.

Fix $\epsilon\leq p_0/8$. Recall that $\bar d$ is the average degree and $\alpha$ is the expansion. Fix $K=\frac{\bar d}{\epsilon\alpha}$.
Then following the proof of Sarkar \cite{sarkar2018note}, there exists $R<\infty$ such that for all $(G_i,o_i)\in H_i$,
\begin{equation}\label{eq: R for continuity}
    B_R(G_1,o'_1)\neq B_R(G_2,o'_2)\qquad \forall o'_i\in B_K(G_i,o_i) \text{ for } i=1,2.
\end{equation}

Further, by tightness of $\mu$ and compactness of $H_i$, Theorem A.16 of \cite{RemcoVol2} implies that there exists some $\Delta_{R,K}$ such that the maximum degree of $R+K$-neighborhoods of any rooted graphs in $H_i$ is bounded by $\Delta=\Delta_{R,k}$, i.e., 
\[\sup_{(G,o)\sim \mu}\max\{deg(o'):\quad\forall o'\in B_{K+R}(G_i,o_i)\quad \forall (G_i,o_i)\in H_i\}\leq \Delta.\]
 Let $\mathcal H_{R+K,\Delta}$ be the set of connected rooted graphs of radius $K+R$ whose  vertices all have degree at most $\Delta$.   Note that the number of such graphs is  by {$1+\Delta+\dots+\Delta^{R+K}\leq \Delta^{R+K+1}$,  
  $|\mathcal H_{R+K,\Delta}|\leq \Delta^{R+K+1}$}. Define a local function $h_{1, R+K}$ as 
 %the indicator of isomorphism to  $B_{R+K}(G',o')$ for some $(G',o')\in H_1$, i.e., for the rooted graph $(G,o)$,
follows:
\[h_{1, R+K}(G,o)=\max_{(G',o')\in \mathcal H_{R+K,\Delta}\cap H_{1}}\mathbbm 1 \Big(B_{R+K}(G,o)\simeq B_{R+K}(G',o')\Big).\]
 Then by local convergence in probability,
 \[\mathbb E_{\mathcal P_n}[h_{1, R+K}|G_n]\overset{\mathbb P}{\to} \mathbb E_\mu[h_{1, R+K}].\]
 Note that $h_{1, R+K}(G,o)=1$  for all $(G,o)\in H_1$ and hence, $\mathbb E_\mu[h_{1, R+K}]\geq p_0/2$.
 Given $G_n$, let $A_{i,n}=\{v\in V(G_n): h_{i, R+K}(G_n,v)=1\}$. Then by local convergence in probability, for large enough $n$,
 \[\mathbb P(|A_i,n|\geq p_0n/4)\geq 1-\epsilon.\]
 Now, combining this with large-set expansion,  for large enough $n$,
 \[\mathbb P(G_n\text{ is }(\alpha,\epsilon,\bar d)  \text{ large-set expander, and }|A_i,n|\geq p_0n/4)\geq 1-2\epsilon.\]
Let $G_n$ be an instance such that the above conditions hold. Then  by Menger's Theorem {applied to bounded average degree large set expanders (as in the proof of} Lemma \ref{lem: sprinkling}), there exists a path of length at most $K=\frac{\bar d}{\epsilon\alpha}$ between $A_{1,n}$ and $A_{2,n}$ in $G_n$. Let $v_1$ and $v_2$ be the two ends of this path, and $v_1\in A_{1,n}$, $v_2\in A_{2,n}$. Then 
\[B_R(G_n,v_2)\subseteq B_{R+K}(G_n,v_1)\simeq B_{R+K}(G_1,r_1),\]
for some $(G_1,r_1)\in H_1$. Also, $B_R(G_n,v_2)\simeq B_{R+K}(G_2,r_2)$ for some $(G_2,r_2)\in H_2$. Therefore,
\[\mathbb P\Big(B_R(G_1,o'_1)\simeq B_R(G_2,o'_2)\qquad \text{ for some } o'_i\in B_K(G_i,o_i) \text{ for } i=1,2\Big)\geq 1-2\epsilon,\]
which is a contradiction with \eqref{eq: R for continuity}. So, $\mu$ must be extremal.
\end{proof}

\section{Proof of Lemma \ref{lm: falik}}\label{apx: Russo}

We start by stating the relevant result of \cite{falik}
in the general setting considered there.  Given an arbitrary probability  measure $\mu$ on $\{0,1\}^m$, we denote expectations with respect to $\mu$ by $\mathbb E$.  
We use $x$ to denote elements of $\{0,1\}^m$, $x_1,\ldots,x_m$  for the coordinates of $x$, and the notation $x\sim y$ to denote elements of $\{0,1\}^m$ that differ in exactly one coordinate.
For $i=0,1,\dots, m$, define $f_i=\mathbb{E}[f|x_1,\ldots x_i]$ (so in particular $f_m=f$ and $f_0=\mathbb E [f]$), and for $i=1,\dots, m$, define  $d_i=f_i-f_{i-1}$.
Finally, we use $c(\mu)$ to denote the log-Sobolev constant, 
$$
c(\mu)=\sup_f\frac{\mathbb E_{x\sim \mu}[\sum_{y\sim x} (f(x)-f(y))^2]}{Ent(f^2)}
$$
where the sup goes over all boolean functions $f$. 
The theorem we use to prove
Lemma \ref{lm: falik} is
Theorem 2.2 in \cite{falik}, which states that
\begin{equation}\label{eq: falik}
{c({\mathbb P_p})} \operatorname{var}(f)\log\Big(\frac{\operatorname{var}(f)}{\sum_{e=1}^m\mathbb E^2|d_e|}\Big)
\leq \mathbb E_{x}\Big[\sum_{y\sim x} (f(x)-f(y))^2\Big].
\end{equation}
If $\mu$ is the product measure for independent $Be(p)$ variables, $\mu=\mathbb P_p$,  the log-Sobolev constant is explicitly known, and is equal to
$c(\mathbb P_p)=\frac{1-2p}{p(1-p)\log\frac{1-p}{p}}$,
{see, e.g.,} Theorem 5.2 in \cite{Boucheron2013ConcentrationI}. 
Lemma \ref{lm: falik} therefor follows once we establish that
\begin{equation}\label{E2}
p(1-p)\mathbb E_{x}[\sum_{y\sim x} (f(x)-f(y))^2]=\mathcal{E}_2(f),\end{equation}
and
\begin{equation}\label{E1-bd}
{\sum_e}\mathbb E^2|d_e|{\leq} \mathcal{E}_1(f).
\end{equation}

{We start with the observation that 
$\Delta_e f(x)=(1-p)(f(x)-f(x\oplus e))$
if $x_e=1$ 
and 
$
\Delta_e f(x)=p(f(x)-f(x\oplus e))$
if $x_e=0$. Using the fact that the first event happens with probability $p$, and the second with probability $1-p$, one easily sees that
\[
\mathbb E_{x}\,[|\Delta_e f|]=2p(1-p)\mathbb E_{ x}\,[|f(x)-f(x\oplus e)|]
%=2p_e(1-p_e) \frac d{dp_e}\mathbb E_{\vec p}\,[f].
\]
and 
$$
\mathbb E_x[|\Delta_ef(x)|^2]=(p(1-p)^2+(1-p)p^2)\mathbb E_x[(f(x)-f(x\oplus e))^2].
$$
The first identity,}
together
with  Lemma {3.1} in \cite{falik}, which states that
\[\mathbb E |d_e|\leq 2p(1-p)\mathbb E _x|f(x)-f(x\oplus e)|,\]
then implies
\eqref{E1-bd}. 

{To prove \eqref{E2}, we 
use the second identity and the fact that} $p(1-p)^2+(1-p)p^2=p(1-p)$ to get
\begin{align*}
   \mathcal{E}_2(f)&=\sum_e \mathbb E_x[({\Delta_e f(x)})^2]
       = p(1-p)\sum_e\mathbb E_x[
     (f(x)-f(x\oplus e))^2].
     %\\&=p(1-p)\mathbb E_x[ (f(x)-f(x\oplus e))^2],
   \end{align*}
%where we used that $p(1-p)^2+(1-p)p^2=p(1-p)$ {and $\mathbb P_p^e$ is defined as the measure $\mathbb P_p$  conditioned on the state of $x_e$.} 
{On the other hand,}
\begin{align*}
    \mathbb E_{x}[\sum_{y\sim x} (f(x)&-f(y))^2]= \sum_x\sum_{y\sim x}\mathbb P_p(x)(f(x)-f(y))^2\\
    &=  \sum_x\sum_{e\in[m]}\mathbb P_p(x)(f(x)-f(x\oplus e))^2
    =\sum_e\mathbb E_x[ (f(x)-f(x\oplus e))^2],
  \end{align*}
which completes the proof of \eqref{E2}  and hence of the lemma.

\section{Uniform Bounds on the Size of $C_2$}
\label{app:C2}

In this appendix we  prove Lemma \ref{lem: uniform convergence second component}. 
We follow the strategy of \cite{alon2004} where  a similar result for expanders with bounded maximum degree is proved. Given  a graph $G_n$ on $n$ vertices, a positive number $c>0$ and an edge $e\in E(G_n)$, let $S(e,c,n)$ be the event that $e$ connects two components of size larger than $cn$ in $G_n(p)$. Let $S(c,n)$ be the event that $S(e,c,n)$ occurs for an edge $e$ chosen uniformly at random
from all edges in $E(G_n)$.
The following bounds the probability that the event $S(e,c,n)$ holds. 

\begin{lemma}\label{lm: pivotal edge bound}
Given $q>0$ there exist a constant $\beta$ such that for all $p\in[q,1-q]$ and all finite graphs $G_n$,
\[\mathbb{P}_{G_n(p)}(S(c,n))\leq (\lfloor \frac{1}{c}\rfloor -1)\frac{\beta}{\sqrt{|E(G_n)|}}.\]
\end{lemma}
\begin{proof}
% {For a fixed graph $G_n$,}
This follows from Lemma 2.3 and equation (6) in \cite{alon2004}; note that while equation (6) in \cite{alon2004} appears in the proof of a corollary  which assumes expansion and bounded degrees in its statement, neither of these assumptions enter their proof of the bound (6).  In fact, $\beta$ is nothing but the constant from Lemma 2.3 in \cite{alon2004} (where it is called $\alpha$), and it just depends on $q$. 
\end{proof}

To state the next lemma, we use the notation $B_r(A,G)$ for the $r$-neighborhood of a set of vertices $A$ in a graph $G$.

\begin{prop}\label{prop: expansion property}
Let $\alpha>0$< $\bar d<\infty$, and $0<c<1$, and set
\begin{equation}\label{r-B2}
r=\lceil 2\bar d/\alpha\rceil +\lceil \bar d/(2c\alpha)\rceil.
\end{equation}
Let $\epsilon\leq \min(c,1-c)$, and let $G$ be a graph with $n$ vertices and average degree at most $\bar d$ such that $\phi(G,\epsilon)\geq \alpha$.  Then $|B_r(G,A)|\geq \frac{3}{4}n$ 
for all $A\subseteq V(G)$ with $|A|\geq cn$.
\end{prop}
\begin{proof} 
The  proof is adapted from Lemma 2.6 in 
\cite{alon2004}.
Assume by contradiction that
$|B_r(G,A)|< 3n/4$.  Setting $C=V(G)\setminus B_r(G,A)$ we then have $|C|> n/4$.
Let $E(W)$ be the set of edges joining two points in $W$.
By the expansion property, if $|B_{k}(G,C)|\leq n/2$ then 
\[|E(B_{k+1}(G,C))|\geq |E(B_{k}(G,C))|+\alpha n/4,\]
and by induction $|E(B_{k+1}(G,C))|\geq \alpha (k+1) n/4$.  Since the total number of edges is at most $\bar dn/2$, we conclude that $|B_k(G,C)|>n/2$ if
$k\geq 2\bar d/\alpha$, and similarly,  $|B_{k'}(G,A)|> n/2$ if $k'\geq \bar d/(2c\alpha)$.
Therefore, $B_{k}(G,C)\cap B_{k'}(G,A)\neq \emptyset$, showing that the distance between $C$ and $A$ is at most $r$, which is a  contradiction.
\end{proof}
The following is adapted from the proof of Lemma 2.7. in \cite{alon2004}. 
The main difference is that we will replace the bounded degree condition used there by the tightness condition { \eqref{Delta-tightness}.}

\begin{proof}[Proof of  Lemma \ref{lem: uniform convergence second component}]

{By {large-set expansion}, there exist}
$\alpha>0$ and $\bar d$ be such that  {for all $\epsilon>0$, with probability tending to $1$}, $\{G_n\}$ is an $(\alpha,{\epsilon,}\bar d)$-large-set expander. Let $r$ be as in \eqref{r-B2}.
Define ${V}_{k,\Delta}$ as the set of set of vertices such that all  vertices in their $k$ neighborhood have degree at most $\Delta$.
{Given the tightness condition \eqref{Delta-tightness}},  for all
$\epsilon>0$, there exists $\Delta<\infty$ and $N_\epsilon<\infty$ such that for $n\geq N_\epsilon$, with probability $1-\frac{\epsilon}{4}$ 
we have $\frac{|V_{r,\Delta}|}{n}\geq 1-\frac{\epsilon}{4}$. Let $A_\epsilon$ be the event that the following conditions hold:
$\frac{|V_{r,\Delta}|}{n}\geq 1-\frac{\epsilon}{4}$,
 $\phi(G_n,\epsilon)\geq \alpha$, and  $G_n$
has average degree at most $\bar d$.  Increasing $N_\epsilon$
if needed, then for $n\geq N_\epsilon$, $A_\epsilon$ has probability at least $1-\epsilon/2$.

{Fix a $G_n$ such that $\phi(G_n,\epsilon)\geq \alpha $ and $A_\epsilon$ holds.}
For a vertex $v\in V(G_n)$ let $S'(v,c,n,r)$ be the event that there exists an edge $e$ in the ball $B_r(v)$ such that $S(e,c,n)$ holds. Let $D(v,r)$ be the event that $B_r(v)$ intersects with at least two different connected components of size greater than $cn$. We will use the bound (10) in \cite{alon2004}, which states that
\begin{equation}\label{eq: S(v,c,n)}
\mathbb{P}_{G_n(p)}\big(S'(v,c,n,r)\big)\geq q^{2r}\Delta^{-2r^2}\mathbb{P}_{G_n(p)}\big(D(v,r)\big),
\end{equation}
holds as long as the degree of every vertex in 
$B_r(v)$ has degree at most $\Delta$, i.e., as  long as $v\in V_{r,\Delta}$.
On the other hand, for graphs $G_n$ whose average degree is bounded by $\bar d$,
\begin{align*}
   \frac{1}{n}\sum_{v\in V_{r,\Delta}}& \mathbb{P}_{G_n(p)}\big(S'(v,c,n,r)\big)
    \leq \Delta^{r} \frac{1}{n}\sum_{e\in E(G_n)}\mathbb{P}_{G_n(p)}\big(S(e,c,n)\big)\\
    &\leq 2\bar d \Delta^{r}\mathbb{P}_{G_n(p)}\big(S(c,n)\big)
    %\\    &
    \leq 2 \bar d \Delta^{r}(\lfloor\frac{1}{c}\rfloor-1)\frac{\beta}{\sqrt{|E(G_n)|}},
\end{align*}
where the first inequality follows  by a union bound on the edges and the observation that each edge can appear in $r$-neighborhood of at most $\Delta^{r}$ vertices of $V_{r,\Delta}$. Combining this inequality with \eqref{eq: S(v,c,n)} we have
\[ \frac{1}{n}\sum_{v\in V_{r,\Delta}} \mathbb{P}_{G_n(p)}\big(D(v,r)\big)\leq 2 \bar d \Delta^{r+2r^2}(\lfloor\frac{1}{c}\rfloor-1)\frac{\beta}{q^{2r}\sqrt{|E(G_n)|}}.\]

 Given an instance of $G_n(p)$ with two or more components of size larger than $cn$, by $A_\epsilon$, the choice of $r$, and Proposition \ref{prop: expansion property}, the $r$-neighborhood of each of them contains 
at least $3n/4$ vertices,  implying that there are at least $n/2$ vertices with distance $r$ or less from two large components. Thus  the event $D(v,r)$ takes place for at least $n/2-\epsilon n/4\geq n/4$   nodes  $v\in V_{r,\Delta}$. 
By Markov's inequality applied to the sum of the indicator functions of $D(v,r)$ over $v\in  V_{r,\Delta}$, we therefore get
that
\begin{equation}\label{eq: bnd second component} 
% Yeganeh added labeling for equation
    \mathbb{P}_{G_n(p)}\left(\frac{|C_2|}{n}\geq cn\right)
    \leq \frac 4{n}\sum_{v\in V_{2r,\Delta}} \mathbb{P}_{G_n(p)}\big(D(v,r)\big)
    \leq \frac C{\sqrt{|E(G_n)|}}
\end{equation}
where $C$ is a constant which depends on $\alpha$, $\bar d$,
$q$ and $\Delta$, but not on $n$ or $p$ (as long as $q\leq p\leq 1-q$). 

Choose $n$ large enough that $A_\epsilon$ holds with probability $\epsilon/2$ and
$C/{\sqrt{|E(G_n)|}}\leq \epsilon /2$. Then we get the lemma by conditioning over all instances of $G_n$ and applying \eqref{eq: bnd second component} for any instance satisfying $A_\epsilon$ and $\phi(G_n,\epsilon)\geq \alpha$.
\end{proof}

%%%%%%%%%%%%%%%%%%%%%%%%%%%%%%%%%%%%%%%%%%
\section{Expansion of Preferential Attachment Models}\label{sec: expansion pref attachment}
In this appendix, we extend the proof of Theorem 1 in \cite{MIHAIL2006239} to show that conditional preferential attachment models are good expanders (Lemma~\ref{lem: pref-attach expansion}).  

The first step is to bound the maximum degree, {a step which was needed for the model considered in \cite{MIHAIL2006239} for reasons that will become clear in the course of our proof.} There are stronger bounds on the maximum degree  of a vertex in other variations of preferential attachment (see for example Section 4.3 in \cite{durrett_2006} and Theorem 4.18 in \cite{hofstadVol1}). But {Proposition~\ref{prop: max degree PA} below}  is sufficient for our purposes.  Before stating the proposition, we 
state and prove a simple lemma, which will be used in its proof.

\begin{lemma}\label{prop: weighte Lehmer inequality}
For any sequences $0<d_1\leq d_2\leq \ldots \leq d_n$ and $w_1\geq w_2\geq\ldots \geq w_n>0$, then
\[\frac{\sum_{i=1}^n w_id_i^2}{\sum_{i=1}^n w_id_i}\leq \frac{\sum_{i=1}^n d_i^2}{\sum_{i=1}^n d_i}.\]
\end{lemma}
\begin{proof}
The statement is equivalent to
\[\sum_{i=1}^n\sum_{j=1}^n w_id_i^2d_j\leq \sum_{i=1}^n\sum_{j=1}^nw_jd_i^2d_j.\]
By reordering terms this is equivalent to
\[0\leq \sum_{i,j}{d_id_j} \big(w_jd_i-w_id_i+w_id_j-w_jd_j\big),\]
which holds since $w_j-w_i$ and $d_i-d_j$ have the same sign and
\[\big(w_jd_i-w_id_i+w_id_j-w_jd_j\big)=(w_j-w_i)(d_i-d_j)\geq 0.\]
\end{proof}

\begin{prop}\label{prop: max degree PA}
Given {$m\geq 2$} and $n\geq 2m+1$ let ${P{\hskip-.2em}A}_{m,n}$ 
be 
the conditional preferential attachment model defined
in Section \ref{sec: pref attachment}. Let $M$ be the maximum degree  in ${P{\hskip-.2em}A}_{m,n}$. Then there exists some $C>0$ such that 
\[\mathbb P(M\geq n^{7/8})\leq e^{-C\sqrt n}.\]
\end{prop}
\begin{proof}
Recall that we start the sequence $({P{\hskip-.2em}A}_{m,n})_{n\geq t_0}$ with
$K_{2m+1}$ at time $t_0=2m+1$.  It has $mn$
edges at time $n\geq 2m+1$, giving $\sum_i d_i(n)=2mn$.  {Let $S_q(n)$ denote the set of sequences
$(k_1,\dots,k_q)$ of pairwise different integers in $[n]$.}
The probability that the first, second, $\dots$ of the $m$ edges created at time $n+1$ attaches to 
$m$ different nodes $k_1,\ldots ,k_m\in [n]$ is
\begin{equation}\label{P-cond}
P^{cond}_n(k_1,\dots k_m)=\frac{\prod_{i=1}^md_{k_i}(n)}{\sum_{(k'_1,\ldots,k'_m)\in S_m(n)}\prod_{i=1}^md_{k'_i}(n)}.
\end{equation}
Define $$w_i(n)=\sum_{\substack{(k_2,\ldots ,k_{m})\in S_{m-1}(n) \\  k_j\neq i, \forall j}}
\prod_{j=2}^{m}d_{k_j}(n).$$
{For $s=1,\dots, m$, the marginal the $s^\text{th}$} new edge connecting to 
vertex $i$ is then equal to 
$$p_n(i)=\frac{d_i(n)w_{i}(n)}{\sum_{j=1}^nd_j(n)w_j(n)}.$$

Let $Q_n=\sum_i d_i^2(n)$. Note that $d_{\max}^2(n)\leq Q_n$. So, it is enough to bound $Q_n$. For that purpose, we use concentration for super-martingales. First, note that
\begin{align*}
    \mathbb E[Q_{n+1}\mid {P{\hskip-.2em}A}_{m,n}]&=Q_n+2\sum_{k_1\ldots k_m} P_t^{cond}(k_1,\ldots, k_m)\big(d_{k_1}(n)+\cdots+d_{k_m}(n)\big)+m+m^2\\
    &= Q_n+2m\sum_{i}\frac{w_i(n)d_i^2(n)}{\sum_j w_j(n)d_j(n)}+m(m+1).
\end{align*}
Note that if $d_i(n)\geq d_j(n)$ then $w_i(n)\leq w_j(n)$. Then by Lemma \ref{prop: weighte Lehmer inequality},
\begin{align*}
    \mathbb E[Q_{n+1}\mid {P{\hskip-.2em}A}_{m,n}]&\leq Q_n+2m\sum_{i}\frac{d^2_i(n)}{2mn}+m(m+1)
    %\\    &
    =  Q_n(1+\frac{1}{n})+m(m+1).
\end{align*}
Let $S_n=\frac{Q_n}{n}-m(m+1)(\sum^n_{i=2m+2}\frac{1}{i})$, then $S_n$ is a supermartingale.
Also, since all $d_i$ start at $2m$ and can grow by at most one in each step, we have $ d_i(n)\leq n-1$.  Using this fact, we easily see that {$Q_{n+1}\leq Q_n+2m(n-1)+m(m+1)$ and hence}
$$
%-\frac{m(m+1)}{n+1}
0\leq \frac{Q_{n+1}-Q_n-m(m+1)}{n+1}\leq
S_{n+1}-S_n= \frac{Q_{n+1}}{n+1}-\frac{Q_n}{n}-\frac{m(m+1)}{n+1}\leq 2m,$$
%Since $n\geq 2m+1$, 
showing that
$|S_{n+1}-S_n|\leq 2m$.
Finally, 
$$\mathbb E[S_n]{\leq}\mathbb E[S_{2m+1}]= 
{4m^2}
.$$
So, we can use the Azuma-Hoeffding inequality to get that 
\[\mathbb P(S_n\geq {\lambda}+4m^2)
\leq e^{-\frac{\lambda^2}{8m^2n}}.\] Then 
\[\mathbb P(Q_n\geq \lambda n^{3/2}+{4nm^2+m(m+1)n\log \frac n{2m}})\leq e^{-{\frac{\lambda^2}{8m^2}}}.\]
Since $M^2\leq Q_n$ we get that there exists a constant $C>0$ such that
\[\mathbb P(M\geq n^{7/8})\leq \mathbb P(Q_n\geq n^{7/4})\leq e^{-C\sqrt n}.\]
\end{proof}
}
{
\begin{lemma}\label{lm: cond. vs ind. PA}
Consider the conditional preferential attachment model specified in Proposition~\ref{prop: max degree PA},  let $M_t$ be the maximal degree at time $t$, and let $P^{cond}_t(k_1,\dots k_m)$ be the conditional probability from
\eqref{P-cond}.  Then
$$
P^{cond}_t(k_1,\dots k_m)\leq 
2^{(m-1)\frac{M_t}{t}}\prod_{i=1}^m\frac{d_{k_i}(t)}{2mt}.
$$
\end{lemma}
\begin{proof}
With the definitions from the previous proof,
let
$$
Z=\sum_{(k'_1,\ldots,k'_m)\in S_m(t)}\prod_{i=1}^md_{k'_i}(t).
$$
Then 
\begin{align*}
Z&=\sum_{(k'_2,\ldots,k'_m)\in S_{m-1}(n)}\left(\prod_{i=2}^md_{k'_i}(t)\right)\sum_{k_1'\notin\{k_2',\dots,k_m'\}}d_{k'_1}(t)\\
&\geq
\Big(2mt-(m-1)M_t\Big)\sum_{(k'_2,\ldots,k'_m)\in S_{m-1}(t)}\left(\prod_{i=2}^md_{k'_i}(t)\right).
\end{align*}
Since $M_t\leq t-1$, we have that
$$
2mt-(m-1)M_t\geq\left(1-\frac{M_t}{2t}\right) 2mt\geq 4^{-\frac{M_t}{2t}}{2mt}=
2^{-\frac{M_t}{t}}{2mt}.
$$
Repeating this process, we  get that
$$
Z
\geq 2^{-(m-1)\frac{M_t}{t}} (2mt)^{m}.
$$
The lemma follows.
\end{proof}
}

Now, we are ready to prove Lemma~\ref{lem: pref-attach expansion}.
\begin{proof}[Proof of Lemma~\ref{lem: pref-attach expansion}]
{We will prove the lemma for $\alpha=\frac{m-1}{20}$. Our proof will follow the general strategy of the expansion bound in \cite{MIHAIL2006239}, but requires several modifications - the main one stemming from the fact that the conditioning in the conditional model considered here will results in a extra factor growing exponentially in the largest degree, see \eqref{eq: bad set} below, which differs from Lemma 2 in  \cite{MIHAIL2006239}  by the factor $C^{n^{7/8}{\log n}}$. We will  offset this factor by an extra exponential decay stemming from the fact we only consider large set expanders, whereas 
\cite{MIHAIL2006239} proved expansion for sets which can be arbitrary small.
}

Let $V=V({P{\hskip-.2em}A}_{m,n})$ and $E=E({P{\hskip-.2em}A}_{m,n})$ be the set of vertices and edges of the graph, respectively. Vertices are indexed based on their arrival time, {with the first $2m+1$ vertices ordered in a arbitrary way.}
Recall that when the vertex $t+1$ arrives, it attaches $m$  edges to $m$ distinct old vertices according to \eqref{P-cond}; we assign  indices $mt+1,\dots,m(t+1)$ to these edges and call them their \emph{arrival} index; {for the edges in the original graph $K_{2m+1}$, we choose the indices arbitrary between $1$ and $m(2m+1)$, subject to the constraint that the $m$ edges between a vertex $t$ and a vertex of lower index lie between $(t-1)m+1$ and $tm$. We use}
$E'=E\setminus E({P{\hskip-.2em}A}_{m,2m+1})$ to denote the set of edges with index larger than $m(2m+1)$.

  {Consider a set $S$ of size $|S|=k$ with
$\epsilon n\leq k\leq \frac{n}{2}$, and call an edge good if it lies in $e(S,V\setminus S)$, and bad otherwise.  We  need to show that  for  each such $S$ there are at least $\lceil\alpha k\rceil$ good edges.
{Indeed, we will show something slightly stronger, namely that for each such $S$, there are at least 
$\lceil\alpha k\rceil$ good edges in $E'$.}
%(we call $S$ expanding if this happens, and non-expanding otherwise).  
To do so, we will
show that with high probability, for any  set $S$, and any set of edge-indices $A{\subset E'}$ of size $|A|\leq   k_\alpha=\lceil  \alpha k\rceil-1$, there must be 
at least one  good edge in $E'\setminus A$, i.e.,   we will show that for {$n\geq n_0$}}
\begin{equation}\label{eq: bad set}
    \mathbb{P}(\text{all }e\in {E'}\setminus A \text{ are bad} )\leq C^{n^{7/8}{\log n}}\frac{{mk\choose k_\alpha }}{{mn-k_\alpha\choose mk-k_\alpha  }},
\end{equation}
where {$n_0<\infty$ and}
 $C > 0$ {are constants}
 %is a constant 
 which depends on $m$,$\alpha$, and $\epsilon$. The proof of this bound is adapted from that of Lemma 2 in \cite{MIHAIL2006239}, see above.

Assuming \eqref{eq: bad set}, we first prove the lemma.  Given a {
%non-expanding 
set $S$ of size $k$ with   $0\leq k'\leq k_\alpha$ good edges in $E'$}
%bad expander $S$ of size $k$, 
there are $k_\alpha+1$ choices for $k'$, and at most ${{mn\choose k'}}\leq {mn\choose k_\alpha} $ 
%$\alpha k{mn\choose \alpha k}$ 
choices for the set of good edges $A$. 
%Since $\mathbb{P}(\text{all }e\in E\setminus A \text{ are bad} )$ is a monotone function  of $A$ with respect to inclusion, we use the same upper bound for all $|A|\leq \alpha k$. 
Noting that there are ${n\choose k}$ choices 
{for sets $S$ of size $k$}, we get that
\begin{align*}
    \mathbb{P}(\phi(G,\epsilon)< \alpha)&\leq C^{ n^{7/8}{\log n}} 
    \sum_{k={\lceil \epsilon n\rceil}    }^{{\lfloor n/2\rfloor}} 
    (k_\alpha+1){n\choose k }{mn\choose  k_\alpha}\frac{{mk\choose  k_\alpha}}{{mn- k_\alpha\choose mk- k_\alpha}}\\
    &\leq C^{ n^{7/8}{\log n}} \sum_{k=\lceil \epsilon n\rceil}^{\lfloor n/2\rfloor}  (k_\alpha+1)\Big(\frac{kem}{ k_\alpha}\Big)^{2 k_\alpha}\Big(\frac{k}{n}\Big)^{(m-1)k-2 k_\alpha}\\
    &\leq C^{ n^{7/8}{\log n}}
    \sum_{k=\lceil\epsilon n\rceil}^{\lfloor n/2\rfloor} (2k\alpha)8\Big(\frac{em}{ \alpha}\Big)^{2 k\alpha}\Big(\frac{1}{2}\Big)^{(m-1-2\alpha)k}
    \\
    &\leq 2\alpha n^2C^{ n^{7/8}{\log n}}
     \Big(\Big(\frac{em}{ \alpha}\Big)^{2 \alpha}\Big(\frac{1}{2}\Big)^{(m-1-2\alpha)}\Big)^{\epsilon n}
   \end{align*}
{provided $n$ is large enough to guarantee that $\alpha\epsilon n\geq 2$.
Here  the second step follows from the fact that
$
{{n\choose k }}
{{(m-1)n- k_\alpha\choose (m-1)k- k_\alpha}}
\leq
{{mn- k_\alpha\choose mk- k_\alpha}}
$
combined with standard bounds on binomial coefficients, the third step follows from
$(k/k_\alpha)\leq \frac 1\alpha (1-1/(k\alpha))^{-1}\leq \frac 1\alpha 4^{1/(k\alpha)}$ if $k\geq\epsilon n\geq 2/\alpha$,
and the third follows from $k\leq n/2$.
}
By the choice of $\alpha$,
{this bound is of the form}
%constant $\beta<1$ such that for $n$ large enough
\[
    \mathbb{P}(\phi(G,\epsilon) <\alpha)
    %&\leq 
    %n^2C^{ n^{7/8}} \Big(\Big(\frac{em}{\alpha}\Big)^{2\alpha}\Big(\frac{1}{2}\Big)^{m-1-2\alpha}\Big)^{\epsilon n}\\
    \leq {2\alpha}n^2C^{ n^{7/8}{\log n}} \beta^{\epsilon n},
\]
{for some constant $\beta<1$ and} drops exponentially fast as $n$ grows. This  {reduces} the proof of the lemma {to the bound \eqref{eq: bad set}.}

To prove this bound, we first note that the left
is a monotone function  of $A$ with respect to inclusion, 
showing that it is enough to prove the bound for 
$|A|=k_\alpha$. 
Let $\bar{S}=V\setminus S$, and let $B$ be the event that all edges in ${E'}\setminus A$
are bad.  The event $B$ is then the intersection of the events $B_t$, $t=2m+2, \dots n$, where 
$B_t$ is the event that all edges in ${E'}\setminus A$ whose arrival index lies between $(t-1)m+1$ and $tm$ are bad (corresponding to the edges in ${E'}\setminus A$ whose younger endpoint is the vertex $t$).  
We will want to bound the probability of the event $B_t$, conditioned on the graph at time $t-1$. 
{In fact, using the identity
$$
\mathbb P(B)=
\Big(\prod_{t=2m+3}^n\mathbb P(B_{t}\mid B_{t-1}\cap\dots\cap B_{2m+2})\Big)\mathbb P(B_{2m+2}),
$$
we will further assume that the graph at time $t-1$ is such that all edges with arrival index between $m(2m+1)+1$ and $m(t-1)$ are bad.}

We need some notations.
Assume  the number of indices in ${A}$ corresponding to  vertices in $S$ ($\bar S$) is $k_1$  ($k_2$). So, $|{A}|=k_1+k_2$.
Let $x_1< x_2<\ldots< x_{mk-k_1}$ be the arrival indices of edges in 
$E\setminus A$, 
such that their younger endpoints is in $S$, {and let} $\bar x_1< \bar x_2<\ldots< \bar x_{{mn-mk-k_2}}$ {be those whose younger endpoint lies in} $\bar S$.
If $t\in S$,  let $x_{i_1{+1}},\ldots,x_{i_1+{b_t}}$ be the arrival indices $x_i$ such that $(t-1)m+1\leq x_i\leq tm$, and similarly for  $\bar x_{i_1{+1}},\ldots \bar x_{i_1+{b_t}}$ if $t\in \bar S$. Then
 $B_t$ is  the event that for all these indices, the second endpoint lies in $S$ if $t\in S$,
and in $\bar S$ if $t\in \bar S$.
Let $d_S(t)$ be the total degree of nodes in $S$ at the time $t$, {and consider the case $t\in S$.} Then by Lemma~\ref{lm: cond. vs ind. PA}, 
\begin{align*}
    \mathbb P ( B_t\mid {P{\hskip-.2em}A}_{m,t-1})
    &\leq \frac{d_S(t{-1})^{b_t}}{(2m(t{-1}))^{b_t}}2^{(m-1)\frac{M_t}{t}}.
\end{align*}
Let $g_t$ be the number of good edges before node $t$,  let $z_i$  be the number of edges in $E\setminus A$ with an index less than $x_i$, and let  $\bar z_i$  be the of edges in $E\setminus A$ with an index less than $\bar x_i$.    {Note that
$z_i$ is bounded from below by the number of edges in $E\setminus A$ with an index less than $x_i$ such that their younger endpoint is in $S$, i.e.,
$z_i\geq(i-1)$, and similarly, $\bar z_i\geq i-1$.}

{By definition of the indices $i_j$,
$j\in{[}b_t{]}$, the number of  edges 
$E\setminus A$
with both endpoints in $S$ that appeared before time $t$ is ${i_1}$. } Since $d_S(t-1)$ is {equal} to the number of good edges plus twice the number of bad edges with at least one endpoint in $S$ that have appeared so far, we get that 
$d_S(t)\leq g_t+{2i_1}$.
{To bound the denominator $m(t-1)$, we note that the total degree of the graph at time $t-1$ is {equal to} twice the number of good edges plus twice the number of bad edges seen before node $t$. Recall that we assumed that the events $B_s$ hold for  $s=2m+2,\dots,t-1$.  Therefore, all
edges in $E\setminus A$ with arrival index between
$m(2m+1)+1$ and $m(t-1)$ are bad.  Since
$x_{i_1+j}\leq tm$ for $j\in [d_t]$,
we  know that $z_{i_1+j}-m$ is a lower bound on the number of edges in $E\setminus A$ arriving before $t$.  Therefore
$m(t-1)\geq 2g_t+2z_{i_1+j}-2m-2m(2m+1)$.} 
{Hence, for any $j\in [b_t]$,}
\begin{align*}
    \frac{d_S(t-1)}{2m(t-1)}
    &\leq \frac{2i_1+g_t}{2z_{i_1{+j}}+2g_t-2m{(2m+2)}}
   \leq \frac{2i_1+2j+g_t}{2z_{i_1{+j}}+g_t-2m{(2m+2)}}.
\end{align*}
Note that the number of good edges in {the} graph is at most $|A|$ {plus} the  edges in ${P{\hskip-.2em}A}_{m,2m+1}$,   $g_t\leq |A|{+m(2m+1)}$, we conclude that
\begin{align*}
    \frac{d_S(t)}{2m(t{-1})}
        &\leq \frac{i_1+j+\frac{|A|{+m(2m+1)}}{2}}{z_{i_1+j}-{m(2m+2)}+\frac{|A|{+m(2m+1)}}{2}}
        %\\    &
        \leq \frac{i_1+j+|A|}{z_{i_1+j}-{m(2m+2)}+|A|},
\end{align*}
where the second inequality {follows from} the fact that 
$|A|=k_\alpha\geq \alpha\epsilon n-1\geq m(2m+1)$ for large enough $n$.
{As a consquence,}
%By applying a similar bound for each $j\in[b_t]$,
\[    \mathbb P (B_t\mid {P{\hskip-.2em}A}_{m,t-1})
     \leq 2^{(m-1)\frac{M_t}{t}}\prod_{j=1}^{b_t} \frac{i_1+j+|A|}{z_{i_1+j}+|A|-{m(2m+2)}}.\]
We can get a similar bound if the vertex $t$ is in $\bar S$.  

{We will want to use these bounds starting with $t=2m+3$.  Defining  $i_0=\min\{i:x_i\geq 1+(2m+2)m\}$ and $\bar i_0=\min\{i:\bar x_i\geq 1+(2m+2)m\}$, bounding $\mathbb P(B_{2m+2})$ by $1$ and $M_t$ by $t^{7/8}\leq n^{7/8}$, we thus get that}
%\begin{equation}\label{Bad-Bound}
\begin{align}
     \mathbb{P}&(E'\setminus A \text{ is bad} )= \mathbb{P}(B_{{2m+2}}\cap\ldots\cap B_n)
     \notag
     \\
     &\leq 2^{n^{7/8}\log n}
     \prod_{i={i_0}}^{mk-k_1}
     \frac{i+|A|}{z_i+|A|-{m(2m+2)}}
     \prod_{i={\bar i_o}}^{mn-mk-k_2} 
     \frac{i+|A|}{\bar z_i+|A|-{m(2m+2)}}
     \notag\\
     &\leq 2^{n^{7/8}\log n}(mn)^{i_0+\bar i_0}
     \prod_{i=1}^{mk-k_1}
     \frac{i+|A|}{z_i+|A|-{m(2m+2)}}
     \prod_{i=1}^{mn-mk-k_2} 
     \frac{i+|A|}{\bar z_i+|A|-{m(2m+2)}},
     \label{Bad-Bound}
\end{align}
%\end{equation}
where in the last step we used that
$\frac{|A|+i}{z_i+|A|-{m(2m+2)}}\geq \frac{|A|}{|A|+mn}\geq \frac{1}{mn}$.
{Note that $x_i\geq i$ and $\bar x_i\geq i$, which implies that $i_0+\bar i_0\leq 2+2m(2m+2)$, and hence $(mn)^{i_0+\bar i_0}\leq 2^{c\log n}$ for some constants $c$ depending on $m$.}

Recall that $z_i$ (and $\bar z_i$) {is} the number of edges in $E\setminus A$ that appear before $x_i$ ($\bar x_i$). So $z_1<z_2<\ldots <z_{mk-k_1}$,  $\bar z_1<\bar z_2<\ldots <\bar z_{mk-k_2}$, and $z_i\neq \bar z_j$ for all $i\leq mk-k_1$ and $j\leq mk-k_2$.  As a result,
$$
\{ z_1,\ldots, z_{mk-k_1}\}\cup\{ \bar z_1,\ldots, \bar z_{mn-mk-k_2}\}={\{0,\ldots, mn-k_\alpha-1\}}.
$$ 
Using this and the fact that $|A|=k_\alpha=k_1+k_2$, we get
\begin{align*}
     \mathbb{P}(E\setminus A& \text{ is bad} )  {\leq \frac{2^{({c+}n^{7/8})\log n} }{(|A|!)^2}\frac{(mk+k_2)!(mn-mk+k_1)!} {\prod_{\ell=0}^{mn-k_\alpha-1}\Big(\ell+|A|-m(2m+2)\Big)}}
     \\
     &{=}2^{({c+}n^{7/8})\log n} \frac{(mk+k_2)!(mn-mk+k_1)!}{|A|!(mn)!} \Big(\prod_{i={1}}^{m(2m+2)}\frac{mn-i}{|A|-i}\Big) 
     \\
     &\leq 2^{({c+}n^{7/8})\log n}\frac{(mk)!(mn-mk)!}{|A|!(mn-k_1-k_2)!}\Big(\frac{mn-m(2m+2)}{|A|-m(2m+2)}\Big)^{m(2m+2)}\\
     &\leq \frac{{mk\choose  k_\alpha}}{{mn-k_\alpha\choose mk - k_\alpha}} \Big(\frac{2m}{\alpha\epsilon}\Big)^{m(2m+2)}2^{n^{7/8}\log n+c\log n},
\end{align*}
{provided $n\geq \frac {1+2m(2m+2)}{\alpha\epsilon}$.  Here the last bound follows from }
$|A|=\lceil\alpha k\rceil -1\geq \alpha \epsilon n-1$.
%and $c$ is such that $2^{c\log n}=(mn)^{i_0+\bar i_0}$. 
Thus,
we have \eqref{eq: bad set} for $C=4\big(\frac{2m}{\alpha\epsilon}\big)^{m(2m+2)}$.
\end{proof}

\section{Properties of the P\'olya-Point Graph}\label{sec: appendix polya point}

{In this appendix, we will prove Proposition~\ref{prop: fixed point Phi}, as well as the representation \eqref{eq: zeta PA} of $\zeta(p)$ for the P\'olya{-}point process.  As discussed in the paragraphs preceding 
Proposition~\ref{prop: fixed point Phi}, this requires us to understand the
distribution of the numbers of left and right children of a vertex with a given label $(S,x)$ after percolation, $d_L^p$ and $d_R^p$, respectively.
While the distribution of the first kind is just $Bin(m(S),p)$,
the second one requires integration out the degree distribution in the P\'olya{-}point process.  This leads to the following lemma.}

\begin{prop}\label{prop: polya degree}
Given a node of type $(S,x)$ from the P\'olya-point graph, where $S\in\{R,L\}$,
the degree distribution of off-springs of type $R$ after percolation is
\[\mathbb{P}(d_R^p=k|(S,x))= q^{m(S)+1}(1-q)^k{k+m(S)\choose k},\]
where $q=\frac{x}{x+p-xp}$.
Similarly, for the root conditioned on its position $x$,
\[\mathbb{P}(d_R^p=k|(\emptyset,x))= q^{m}(1-q)^k{k+m-1\choose k}.\]
\end{prop}
\begin{proof}
Let $d_R$ be the random variable 
giving the number of  off-springs of  type $R$ before percolation.
{Conditioned on $(S,x)$, the distribution of $d_R$ is a mixed Poisson} with parameter $\gamma \lambda(x)$, where $\lambda(x)=\frac{1-x}{x}$
and {$\gamma\sim\Gamma(\tilde m,1)$, with $\tilde m=m(S)+1$ if $S\in \{S,R\}$ and $\tilde m=m$ if  $S=\emptyset$. }
If we integrate over  $\gamma$, we get 
\begin{align*}
    \mathbb{P}(d_R=k|x)&=\int_{0}^\infty e^{-y\lambda(x)}\frac{(y\lambda(x))^{k}}{k!}\frac{y^{{\tilde m}-1}}{({\tilde m}-1)!}e^{-y}dy\\
    &=\frac{\lambda(x)^k}{(1+\lambda(x))^{k+{\tilde m}}}\frac{(k+{\tilde m}-1)!}{k!({\tilde m}-1)!}
    %\\    &
    ={k+{\tilde m}-1\choose k}x^{\tilde m}(1-x)^k,
\end{align*}
which is a negative binomial distribution with parameters $x$ and ${\tilde m}$. {See also Lemma 5.2  \cite{berger2014}, where this distribution was derived as well.}

Now, we are ready to find the distribution of $d_R^p$.
\begin{align*}
\mathbb{P}(d_R^p=k|x)=&\sum_{d\geq k}{d\choose k}p^k(1-p)^{d-k}\mathbb{P}(d_R=d|x)\\
=&x^{\tilde m}(p(1-x))^k\sum_{d\geq k}{d\choose k}((1-p)(1-x))^{d-k}{d+{\tilde m}-1\choose d}\\
=&x^{\tilde m}(p(1-x))^k{k+{\tilde m}-1\choose k}\sum_{d\geq k}((1-p)(1-x))^{d-k}{d+{\tilde m}-1\choose d-k}\\
=& x^{\tilde m}(p(1-x))^k{k+{\tilde m}-1\choose k}\frac{1}{(1-(1-x)(1-p))^{k+{\tilde m}}}\\
=& (\frac{x}{x+p-xp})^{\tilde m}(1-\frac{x}{x+p-xp})^k{k+{\tilde m}-1\choose k}.
\end{align*}
which is again a negative binomial distribution with parameters ${\tilde m}$ and $\frac{x}{x+p-xp}$.
\end{proof}

To compute $\zeta(p)$ for the  P\'olya-point graph, we 
{derive the implicit formula for the survival probability of a node of type $(S,x)$
given in Proposition \ref{prop: fixed point Phi}.}
%find an implicit function representation for survival probability of nodes. 
Given a node of type $(S,x)$, intuitively, the  probability that $(S,x)$  does not appear in an infinite cluster is equal to the probability that all of its children do not appear in an infinite cluster, i.e., the product of extinction probability of its children.  
{Calculating the extinction probability of $(S,x)$ will then  involve taking the expectation of the products of these extinction probabilities over labels of these children and the degrees $d_R^p$ and $d_L^p$, and thus will lead us to consider expectations of the form given in the next lemma.}

\begin{lemma}\label{lem: product offspring polya point}
{Fix} $p\in[0,1]$,   a measurable function {$g$} on $\{\emptyset,R,L\}\times[0,1]$, and a type $(S,x)$.
Let $d_R^p$ and $d_L^p$ be defined as above,
{and let $\mathbb{E}[\cdot\mid (S,x)] $ denote expectations with respect to the 
forward degrees $d_R^p$ and $d_L^p$ of $(S,x)$, and the positions $x_i$ of its children.}
If $S\in\{R,L\}$, then
\[\mathbb{E}[\prod_{i=1}^{d_L^p}(1-g(L,x_i))\prod_{j=1}^{d_R^p}(1-g(R,y_j))|(S,x)]=  \frac{x^{m(S)+1}(1-p(T_Lg){(x)})^{m(S)}}{\big(x+p(1-x)(T_Rg){(x)}\big)^{m(S)+1}},\]
and if $S=\emptyset$, then
\[\mathbb{E}[\prod_{i=1}^{d_L^p}(1-g(L,x_i))\prod_{j=1}^{d_R^p}(1-g(R,y_j))|(\emptyset,x)]=  \Big(\frac{x(1-p(T_Lg){(x)})}{x+p(1-x)(T_Rg){(x)}}\Big)^{m},\]
where $(T_R g){(x)}=\frac{1}{1-x}\int_{x}^{1}g(R,y)dy$, and 
$(T_L g)(x)=\frac{1}{x}\int_{0}^{x}g(L,y)dy$.
\end{lemma}

{Note that the right hand side of the first identity is just $1-\Phi(g)$, with $\Phi$ as defined in \eqref{phi-def}.
}

\begin{proof}
We give the proof for the case that $S\in\{R,L\}$, and the other case can be derived similarly.
Since the left and right neighbors given $x$ are independant we calculate the expectation independently. 
First, on left neighbors conditioning on $d_L^p$ $x_1\ldots,x_{d_L^p}$ are generated uniformly at random from $[0,x]$. Then
\begin{align*}
    \mathbb{E}[\prod_{i=1}^{d_L^p}(1-g(x_i))|d_L^p=N,(S,x)]
    &=%\prod_{i=1}^{d_L^p}
    \prod_{i=1}^{N}
    (1-\mathbb{E}g(x_i))=(1-(T_Lg)(y))^N
\end{align*}
Therefore,
\begin{align*}
    \mathbb{E}[\prod_{i=1}^{d_L^p}(1-g(x_i))|(S,x)]&=\sum_{N=0}^{m(S)}(1-(T_L)g(x))^N{m(S)\choose N}p^N(1-p)^{m(S)-N}\\
    &=(1-(T_L)g(x)p)^{m(S)}.
\end{align*}
The position of right off-springs conditioned on $d_R^p$ is i.i.d. uniformly at random from $[x,1]$. Hence, similarly
\begin{align*}
    \mathbb{E}[&\prod_{i=1}^{d_R^p}(1-g(x_i))|(S,x)]=\sum_{N\geq 0} (1-(T_R)g(x))^N\mathbb{P}(d_R^p=N|(S,x))\\
    &=\sum_{N\geq 0} (1-(T_R)g(x))^N{N+m(S)\choose N}q^{m(S)+1}(1-q)^{N}\\
    &=q^{m(S)+1}\frac{1}{(1-(1-(T_R)g(x))(1-q))^{m(S)+1}},
    %\\
    %&
    =\Big(\frac{x}{x+p(1-x)(T_Rg)(x)}\Big)^{m(S)+1},
\end{align*}
where in the second step we used Proposition~\ref{prop: polya degree}, and in the last step we used that $q=\frac{x}{x+p-xp}$. The product of right and left off-springs gives the statement.
\end{proof}

\begin{proof}[Proof of Proposition \ref{prop: fixed point Phi}]
\phantom{x}

  1.  For a vertex of type $(S,x)$, the probability that it does not reach level $k$ is equal to the probability that none of its children reach level $k-1$. Combined with Lemma \ref{lem: product offspring polya point}, we therefore have that
    \[\rho_k(S,x)=1-\mathbb{E}[\prod_{i=1}^{d_L^p}(1-\rho_{k-1}(L,x_i))\prod_{j=1}^{d_R^p}(1-\rho_{k-1}(R,y_j))|(S,x)]
    =\Phi\rho_{k-1}(S,x).\]
    Given that all nodes reach level 0, $\rho_0=1$, this implies the first statement of the proposition.
    
    2.  By monotone convergence, $\rho_k(S,x)\downarrow \rho(S,x)$ for all $(S,x)\in \{L,R\}\times [0,1]$, so  $\lim_{k\rightarrow\infty }(\Phi^k1)(S,x)=\rho(S,x)$. Then by dominated convergence we get that $\rho$ is a fixed point of \eqref{eq: fixed point}.
    Let $f$ be another solution of \eqref{eq: fixed point}. For any solution we have that $f(S,x)=\Phi f(S,x)\leq 1$. Therefore, for any $k$, 
    \[f(S,x)=\Phi^k f(S,x)\leq \Phi^k 1 (S,x)=\rho_k(S,x).\]
    As a result, $f(S,x)\leq \rho(S,x)$. 
\end{proof}
{Note} that Lemma \ref{lem: product offspring polya point}, together with the fact that the position of the root is $\sqrt{x}$ for $x$ chosen uniformly at random from $[0,1]$, immediately implies \eqref{eq: zeta PA}.
{We close this appendix with the proof of Proposition \ref{cor: lower bnd zeta pref-attach}.}
\begin{proof}[proof of Proposition \ref{cor: lower bnd zeta pref-attach}]
{We start with some simple observations which we will use throughout the proof.
First, we}
%To prove the {lower bounds}, we will construct
% function $f$ {such} that ${e^{-1/p}\leq} f\leq 1$ %for all $(S,x)\in\{R,L\}\times[0,1]$ 
%and $\Phi f\geq f$ point-wise.  {To see that this is indeed enough,
%to prove the proposition, we first}
note that if $0\leq g\leq g'\leq 1$ point-wise, 
%and {$x>0$}, 
then
$$0=\Phi (0)\leq(\Phi g){(S,x)}\leq (\Phi g'){(S,x)}\leq (\Phi 1){(S,x)}
%=1-\frac{(1-p)^{m{(S)}}}{(1-p+\frac{p}{x})^{m{(S)}+1}}<
{\leq} 1.
$$ 
{Next, by dividing both the numerator and the denominator in the expressions for $1-({\Phi} f)(S,x)$  in \eqref{phi-def} 
and in \eqref{eq: zeta PA} by $x^{m(S)+1}$ and $x^m$, respectively, we see that
\begin{equation}
\begin{aligned}
1-\left(\frac{1-\frac px\int_{0}^{x}f(L,z)dz}{1+\frac px\int_{x}^{1}f(R,z)dz}\right)^{m-1}
&\leq ({\Phi} f)(R,x)
\leq ({\Phi} f)(\emptyset,x)\\
\leq ({\Phi} f)(L,x)
&\leq
1-\left(\frac{1-\frac px\int_{0}^{x}f(L,z)dz}{1+\frac px\int_{x}^{1}f(R,z)dz}\right)^{m+1}
\end{aligned}
\end{equation}
To prove upper and lower bounds on $\rho(S,x)$, we then use that $\rho(S,x)$ is the pointwise monotone limit, $\rho_k(S,x)\downarrow \rho(S,x)$, where $\rho_k={\Phi}^k 1$, see Proposition~\ref{prop: fixed point Phi} and its proof.}

For the upper bound, we   {will inductively bound $\rho_k$ from above by functions $f_k$ that don't depend on the discrete variable $S$. 
  Assume thus that $f$ is of this form.}
Then
\begin{align*}
(\Phi f)(S,x)
&\leq
1-
\left(
\frac{1-\frac px\int_0^xf(y)dy}
{1+\frac px\int_x^1f(y)dy}
\right)^{m+1}
\\&=
1-
\left(
\frac{1-\frac px\int_0^xf(y)dy}
{1-\frac px\int_0^xf(y)dy+\frac px\int_0^1f(y)dy}
\right)^{m+1}
\\
&\leq
1-
\left(
\frac{1-p}
{1-p+\frac px\int_0^1f(y)dy}
\right)^{m+1}
%\\&
=1-
\left(\frac{1}%
{1+\frac p{x(1-p)}\int_0^1f(y)dy}\right)^{m+1}.
\end{align*}
{where in the second bound we used that 
$\frac 1x\int_0^xf(y)dy\leq 1$.}
{As a consequence,}
$$
\rho(S,x)\leq \rho_k(S,x)\leq f_{k}(x)
%\qquad\text{for all}\qquad k\geq 1,
\qquad\text{where}\qquad
f_{k}(x)=1-
\Big(\frac{1}{1+\frac{\epsilon_k}x}\Big)^{m+1}
$$
and $\epsilon_k$ is inductively defined by $\epsilon_{1}=\frac{p}{1-p}\int_0^11=\frac p{1-p}$ and
$$
\epsilon_{k+1}={F}(\epsilon_k)\qquad\text{where}\qquad
F(\epsilon)=\frac p{1-p}\int_0^1\Big(1-\Big(\frac{1}{1+\frac{\epsilon}x}\Big)^{m+1}
\Big){dx}.$$
The function $F:[0,\frac{p}{1-p}]\to [0,\frac{p}{1-p}]$ is monotone increasing and concave, with $F(0)=0$ and $F(\frac{p}{1-p}{)}<\frac p{1-p}$, showing that it has two fix-points, the trivial fix-point $0$ and another fix-point $\epsilon^+>0$, with the latter giving the limit $\epsilon^+=\lim_{k\to\infty}\epsilon_k$ and the upper bound
\[\rho(S,x)\leq f_+(x)=1-
\Big(\frac{1}{1+\frac{\epsilon^+}x}\Big)^{m+1}.
\]
{To convert this into an upper bound on $\zeta(p)$, we first bound
\[
({\Phi} \rho)(\emptyset,x)\leq ({\Phi} f_+)(\emptyset, x)
\leq 1-
\left(\frac{1}
{1+\frac p{x(1-p)}\int_0^1f_+(y)dy}\right)^{m}
=1-\Big(\frac{1}{1{+}\frac{\epsilon^+}x}\Big)^m,
\]
which   gives}
$$
\zeta(p)\leq
%\int_0^1\left(1-\Big(\frac{1}{1+\frac p{x(1-p)}\int_0^1f_\infty(y)dy}\Big)^m\right)2xdx=
\int_0^1\Big(1-\Big(\frac{1}{1{+}\frac{\epsilon^+}x}\Big)^m\Big)2xdx
\leq{\int_{0}^1\Big(m\frac{\epsilon^+}x\Big)2xdx= 2m\epsilon^+.}
$$
{Next we use that}
%Using the bound
$$
F(\epsilon)
\leq\frac {p}{1-p}\int_0^\epsilon dx+\frac {p}{1-p}\int_\epsilon^1 \Big((m+1)\frac{\epsilon}x\Big){dx}
=\epsilon\frac {p}{1-p}\Big(1+{(m+1)}\log\Big(\frac 1\epsilon\Big))\Big),
$$
{implying that the non-trivial fix-point 
of $F$
obeys the bound
$1\leq \frac {p}{1-p}(1+{(m+1)}\log(1/\epsilon^+))$.  This  in turn implies that $\epsilon^+\leq e^{-\frac{1-2p}{p(m+1)}}$, giving the desired upper bound on $\zeta(p)$.}
%is bounded from above by

In a similar way, one can obtain lower bounds on $\rho_k$, and thus on $\rho=\lim_{k\to\infty}\rho_k$.
All that changes is that the power $m+1$ in our  upper bound now becomes a power  $m-1$, and  the upper bound $\frac 1x\int_0^xf(y)dy\leq 1$ gets replaced by the lower bound $\frac 1x\int_0^xf(y)dy\geq 0$.  The resulting lower bound is  of the form
\[\rho(S,x)\geq f_-(x)=1-
\Big(\frac{1}{1+\frac{\epsilon^-}x}\Big)^{m-1},
\]
where $ \epsilon^->0$ is the non-trivial fix-point of the function $\tilde F:[0,p]\to [0,p]$ defined by
$$
\tilde F(\epsilon)= p\int_0^1\Big(1-\Big(\frac{1}{1+\frac{\epsilon}x}\Big)^{m-1}
\Big){dx}.
$$
Next, we use the fact that $(1+\frac\epsilon x)^{m-1}\geq 1+(m-1)\frac \epsilon x$ to bound
\begin{equation}\label{eq: lw bnd rho}
% $$
\rho(S,x)\geq 1-\frac{1}{1+(m-1)\frac \epsilon x}=
\frac{(m-1)\epsilon^-}{(m-1)\epsilon^-+x}
\geq \frac{(m-1)\epsilon^-}{(m-1)\epsilon^-+1}.
% $$
\end{equation}
By the same reasoning, we may bound $\tilde F$ from below by
$$
\tilde F(\epsilon)\geq 
p\int_0^1\frac{(m-1)\epsilon}{(m-1)\epsilon+x}dx
=p(m-1)\epsilon \log\frac{1+(m-1)\epsilon}{(m-1)\epsilon}
%\geq p\int_{(m-1)\epsilon}^1\frac{(m-1)\epsilon} x=p(m-1)\epsilon\log\Big(\frac 1{(m-1)\epsilon}\Big),
$$
showing that $\epsilon^-$ is bounded from below by the solution of
$1=p(m-1)\log\frac{1+(m-1)\epsilon}{(m-1)\epsilon}$, which inserted into \eqref{eq: lw bnd rho} gives
$
\rho(S,x)
\geq e^{-\frac 1{p(m-1)}}$, as claimed.  Inserted into \eqref{eq: zeta PA}, this also gives the lower bound on $\zeta(p)$.
\end{proof}
\end{appendix}
\end{document}